\documentclass[pagebackref,colorlinks,citecolor=blue,linkcolor=blue,urlcolor=blue,filecolor=blue]{article}
\pdfoutput=1
\usepackage[margin=1in]{geometry}
\usepackage{tikz-cd}
\usetikzlibrary{calc}
\usetikzlibrary{decorations.pathmorphing}
\usepackage{float}
\usepackage{amsmath}
\usepackage{amsfonts}
\usepackage{amsthm}
\usepackage{enumitem}
\usepackage{dsfont}
\usepackage{array}
\usepackage{longtable}
\usepackage{amssymb}
\usepackage{pifont}
\usepackage{mathtools}
\usepackage{comment}
\usepackage{graphicx, caption} 
\usepackage{float}
\usepackage{spectralsequences}
\usetikzlibrary{matrix}
\usepackage[T1]{fontenc}
\usepackage[utf8]{inputenc}
\usepackage{epigraph}

\usepackage{pinlabel} 

\usepackage{hyperref}

\newtheorem*{T1}{Theorem \ref{space homological stability unordered}}

\newtheorem*{T2}{Theorem \ref{first homology of height 2 box}}

\newtheorem{thm}{Theorem}[section]
\newtheorem{lem}[thm]{Lemma}
\newtheorem{prop}[thm]{Proposition}
\newtheorem{conj}[thm]{Conjecture}
\newtheorem{cor}[thm]{Corollary}

\theoremstyle{definition}

\theoremstyle{definition}

\newtheorem{defn}[thm]{Definition}
\newtheorem{quest}[thm]{Question}
\newtheorem{remark}[thm]{Remark}

\newtheorem*{claim*}{Claim}
\newtheorem*{quest*}{Question}
\newtheorem*{remark*}{Remark}
\newtheorem*{fact*}{Fact}

\newcommand{\newthmtext}{
For integers $w,h\ge 3$ and $n\le wh-3$, the inclusion of $SC_{n}(R_{w,h})$ into $C_{n}(\R^{2})$
induces an isomorphism
\[
H_{1}\big(SC_{n}(R_{w,h})\big)\cong H_{1}\big(C_{n}(\R^{2})\big).
\]
}

\newcommand{\thmtextone}{
For an integer $w\ge h=2$,
\[
H_{1}\big(SC_{n}(R_{w,2})\big)\cong \begin{cases}
\Z^{n-1}&\mbox{if }1\le n\le2w-2,\\
\Z^{w-1}&\mbox{if }n=2w-1, \mbox{and}\\
0&\mbox{otherwise}.
\end{cases}
\]
}

\newcommand{\Z}{\ensuremath{\mathbb{Z}}}

\newcommand{\R}{\ensuremath{\mathbb{R}}}



\title{The first homology of square-section braid groups}
\author{Minh Pham, Nicholas Wawrykow, and Jacob Zweiback}

\date{}

\begin{document}

\maketitle

\begin{abstract}
We compute the first homology of the unordered configuration space of $n$ open unit-squares in the $w$ by $h$ rectangle for all positive integers $n$, $w$, and $h$.
Letting $w$ and $h$ vary, these square configuration spaces form a bi-filtration of the unordered configuration space of $n$ points in the plane, so their fundamental groups approximate the $n$-strand braid group.
Using a pair of spectral sequences, we prove that these square and point configuration spaces have naturally isomorphic first homology groups if $w,h\ge 3$ and $n\le wh-3$.
In the $h=2$ case, we show that the first homology groups of these square configuration spaces are free of rank $n-1$ for $n\le 2w-2$, whereas they are free of rank $w-1$ for $n=2w-1$.
\end{abstract}

\section{Introduction}
We compute the first homology of the square-section braid group $SBr_{n}(R_{w,h})$, which is the fundamental group of $SC_{n}(R_{w,h})$, the unordered configuration space of $n$ open unit-squares in the $w$ by $h$ rectangle.
We show that for integers $w,h\ge 3$ and $n\le wh-3$, the inclusion of $SC_{n}(R_{w,h})$ into $C_{n}(\R^{2})$, the unordered configuration space of $n$ points in the plane, induces an isomorphism between $H_{1}\big(SBr_{n}(R_{w,h})\big)$ and $H_{1}(Br_{n})$, where $Br_{n}$ is the braid group on $n$ strands, i.e., $SBr_{n}(R_{w,h})$ and $Br_{n}$ have isomorphic abelianizations.
Thin rectangles, where $h=2$ and $w\ge2$, behave differently, as the first homology of the square-section braid group is free of rank $n-1$ for $n\le 2w-2$, proving that $SBr_{n}(R_{w,2})$ is almost never isomorphic to $Br_{n}$.
Rather than using discrete Morse theoretic techniques used elsewhere to study these spaces, we approach them by defining a pair of spectral sequences.
This allows us to make the first computations of $H_{1}\big(SBr_{n}(R_{w,h}) \big)$ for $w,h\ge 3$ outside of the already understood ranges where $n\le w,h$ or $n\ge wh-2$.

\subsection{Square configuration spaces} 
The braid group $Br_{n}$ is the fundamental group of $C_{n}(\R^{2})$, the unordered configuration space of $n$ points in the plane.
A braid is a loop in $C_{n}(\R^{2})$, and the strands of a braid correspond to the movement of points around each other in the plane.
Since points have trivial diameter, a braid can be realized in an arbitrarily small region.
This idealization does not fit our ``real-world'' experiences with braids, as the strands of any physical realization of a braid have non-trivial cross-section.
In the plane, thickening the points is not an issue, but in a compact subspace of $\R^{2}$ this thickening might preclude the existence of certain braids.
This raises a pair of questions: 
\begin{quest}\label{how much space quest}
How big does the ambient space need to be to realize a thickened braid?
\end{quest}
\begin{quest}\label{topology of thick braids quest}
What is the topology of the thick-point configuration space?
\end{quest}

We approach these questions by replacing points with unit-squares and studying the resulting square configuration spaces.

\begin{defn}   
Given positive integers $w$ and $h$, the \emph{ordered configuration space of $n$ open unit-squares in the $w$ by $h$ rectangle} $R_{w,h}$ is the space of embeddings of $n$ labeled squares in the Euclidean rectangle $R_{w,h}$ such that the sides of the squares are parallel to the sides of the rectangle.
We write \emph{$SF_{n}(R_{w,h})$} for this configuration space.
\end{defn}

If we identify $R_{w,h}$ with the Euclidean rectangle in $\R^{2}$ whose vertices lie at $(0,0)$, $(w, 0)$, $(0,h)$, and $(w,h)$, one can view $SF_{n}(R_{w,h})$ as a subspace of $\R^{2n}$:
\[
SF_{n}(R_{w,h}):=\left\{(x_{1}, y_{1},\dots, x_{n}, y_{n})\in \R^{2n}\bigg|\substack{{\displaystyle\frac{1}{2}\le x_i\le w-\frac{1}{2}, \frac{1}{2}\le y_i\le h-\frac{1}{2}, \text{ and }}\\\displaystyle\max\{|x_{i}-x_{j}|, |y_{i}-y_{j}|\}\ge 1\text{ for all }i\neq j}\right\}.
\]
This identification replaces points with squares, and it induces an inclusion of $SF_{n}(R_{w,h})$ into $F_{n}(\R^{2})$, the ordered configuration space of $n$ points in the plane, by identifying each square with its center.
The $n$-th symmetric group $S_{n}$ acts freely on $SF_{n}(R_{w,h})$ and $ F_{n}(\R^{2})$ by relabeling the squares and points, and taking the quotient of these spaces by the $S_{n}$-actions yields unordered configuration spaces.

\begin{defn}
The quotient of $SF_{n}(R_{w,h})$ by the $S_{n}$-action is $SC_{n}(R_{w,h})$, the \emph{unordered configuration space of $n$ open unit-squares in $R_{w,h}$}.
The fundamental group of $SC_{n}(R_{w,h})$ is the \emph{square-section braid group $SBr_{n}(R_{w,h})$}.
\end{defn}

The inclusion $SF_{n}(R_{w,h})\hookrightarrow F_{n}(\R^{2})$ is equivariant with respect to the $S_{n}$-action, so it yields an inclusion of $SC_{n}(R_{w,h})$ into $C_{n}(\R^{2})$, and this inclusion induces a map $SBr_{n}(R_{w,h})\to Br_{n}$. 

\subsection{Main results}\label{Main results subsection}
Though $SC_{n}(R_{w,h})$ can be thought of as a subspace of $C_{n}(\R^{2})$, its topology is rather mysterious, and for $n\le wh$, little is known about these square configuration spaces for integers $w,h\ge 3$  other than the cases where $n\le w,h$ \cite{plachta2021configuration} and the cases $n=wh-2$ and $wh-1$ \cite{alvarado2025square, MattOmarJesus}.
Most prior results on these spaces were achieved with (discrete) Morse theory, but there seem to be limitations to applying these techniques to smaller $n$.
Since decreasing $n$ results in an increase to the homotopy dimension of the space, this makes these spaces resistant to finding tractable maximal gradient vector fields.
We overcome these difficulties of determining how well $SC_{n}(R_{w,h})$ approximates $C_{n}(\R^{2})$, by approaching Questions~\ref{how much space quest} and \ref{topology of thick braids quest} via the homology of $SC_{n}(R_{w,h})$.
Namely, we use a pair of spectral sequences to show that for $w,h\ge 3$ and $n\le wh-3$, the inclusion $SC_{n}(R_{w,h})\hookrightarrow C_{n}(\R^{2})$ induces an isomorphism on first homology. 

\begin{thm}\label{space homological stability unordered}
\newthmtext
\end{thm}

This proves that $SBr_{n}(R_{w,h})^{ab}\cong Br_{n}^{ab}$ in this range; in Section~\ref{calcs and conjs} we present calculations and an argument that suggest that $SBr_{n}(R_{w,h})\cong Br_{n}$ for $w,h\ge 3$ and $n\le wh-6$.
Since $H_{1}\big(C_{n}(\R^{2})\big)$ is isomorphic to $\Z$ for all $n\ge 2$ and trivial otherwise, this theorem computes $H_{1}\big(SC_{n}(R_{w,h})\big)$ for $w,h\ge 3$ and $n\le wh-3$.
Combined with the results of Alvarado-Gardu\~{n}o and Gonz\'{a}lez \cite[Theorem 2.2] {alvarado2025square} for $n=wh-2$  and Alvarado-Gardu\~{n}o, Gonz\'{a}lez, and Kahle \cite[Theorem 1.2]{MattOmarJesus} for $n=wh-1$, we obtain an explicit description of $H_{1}\big(SC_{n}(R_{w,h})\big)$ for all $n$ as long as $w,h\ge 3$:

\begin{cor}\label{First homology for w,h ge 3}
For integers $w,h\ge3$,
\[
H_{1}\big(SC_{n}(R_{w,h})\big)\cong \begin{cases}
\Z& 2\le n\le wh-3,\\
\Z^{(w-1)(h-1)+1}&n=wh-2,\\
\Z^{(w-1)(h-1)}&n=wh-1,\\
0&\mbox{otherwise.}
\end{cases}
\]
\end{cor}

Since these square configuration spaces are contractible if $h=1$, it remains to consider the case where $h=2$---rotating $R_{w,h}$ a quarter turn allows us to assume that $w\ge h$.
In this case, the space restrictions prevent $SBr_{n}(R_{w,2})$ from being isomorphic to $Br_{n}$ for $n>2$ as there is not enough space to realize the braid relations $\sigma_{i}\sigma_{i+1}\sigma_{i}=\sigma_{i+1}\sigma_{i}\sigma_{i+1}$.
Indeed, Alvarado-Gardu\~{n}o, Gonz\'{a}lez, and Kahle proved that $SBr_{n}(R_{w,2})$ is a right-angled Artin group on $n-1$ generators if $n\le w$ or $n\le 2w-5$, and that $SBr_{3}(R_{2,2})\cong \Z$ \cite[Theorem 1.5, Theorem 1.2]{MattOmarJesus}.
With the spectral techniques used to handle the $w,h\ge 3$ case, we determine $H_{1}\big(SC_{n}(R_{w,2})\big)$ in the remaining cases.

\begin{thm}\label{first homology of height 2 box}
\thmtextone
\end{thm}

These calculations complete the argument that $SBr_{n}(R_{w,2})$ is almost never isomorphic to $Br_{n}$, which is not right-angled for $n\ge 3$.

\begin{cor}\label{different braid groups height 2}
If $w\ge 2$ and $3\le n\le 2w$, then 
\[
SBr_{n}(R_{w,2})\not\cong Br_{n}.
\]
\end{cor}

Despite this, our computation of $H_{1}\big(SC_{n}(R_{w,2})\big)$ suggests that if $n\le 2w-3$, then, much like the braid group, one might be able to generate $SBr_{n}(R_{w,2})$ with half-twists.

\subsection{A short history of disk and square configurations}
\emph{Point braid groups:} The braid groups $Br_{n}$ were first formally defined by Artin in \cite{artin1925theorie}, though closely related ideas can be found in Hurwitz's paper \cite{hurwitz1891riemann} on monodromy.
In addition to finding presentations for $Br_{n}$, Artin proved that it is the fundamental group of the unordered configuration space of points in the plane in \cite{artin1947theory}.
These early papers sparked a century of mathematical work on braids including work on fibrations \cite{fadell1962configuration}, homological stability \cite{mcduff1975configuration, segal1979topology, church2012homological}, $E_{n}$-algebras \cite{cohen2007homology}, and explicit computations \cite{fuks1970cohomology} among many others.
Unordered configuration spaces and their braid groups remain an active subject, with much recent work on the plane and surfaces in general  \cite{ bharathram2026burau, dlugie2026braid, de2026rigidity, idrissi2026odd, teh2026topological} and on graphs \cite{an2026large, jain2026manifold, khan2026borel}, to name only a few of the results in the past year.
Beyond point configurations, there has been a recent flurry of work on disk and square configurations.

\emph{Disk configuration spaces:} Disk configuration spaces arise when one replaces points with disks of fixed diameter---squares are $l^{\infty}$-disks in $\R^{2}$.
Despite the naturality of these spaces, which can be seen as approximating point configuration spaces and have been studied by statistical physicists and chemists due to their utility in modeling real-world particle behavior since the mid-twentieth century, e.g., \cite{kirkwood1935statistical, alder1957phase,erpenbeck1985equation, franzosi2007topology},
they only became objects of interest to topologists with the seminal work of Diaconis \cite{diaconis2009markov}.
Shortly thereafter, it became clear that the geometry of these spaces leads to topological behavior not seen in point configuration spaces \cite{carlsson2012computational, kahle2012sparse, BBK}.

Building on these works, Alpert, Kahle, and MacPherson turned to studying the ordered configuration space of $n$ open unit-disks in the infinite strip of height $h$  \cite{alpert2021configuration}.
They showed that the homology of these spaces goes through a phase transition as $h$ changes, with a jump in the rank of homology appearing at $n=h+1$.
Alpert and Manin found a basis for all of the homology groups of these spaces for all $h$ and $n$ \cite{alpert2021configuration1}.
In \cite{wawrykow2023representation}, the second author used these results to show that the homology of these spaces satisfies a notion of representation stability while connecting the natural product structure on the homology of the spaces to an interpolation between the $E_{1}$- and $E_{2}$-algebras arising from the homology of the ordered configuration space of points in the plane.

\emph{Square configuration spaces and a discrete model:} Unlike the infinite strip, the rectangle $R_{w,h}$ is bounded on all sides, and this significantly complicates the topology of $SF_{n}(R_{w,h})$ and $SC_{n}(R_{w,h})$.
Plachta proved that $SF_{n}(R_{w,h})$ is homotopy equivalent to $F_{n}(\R^{2})$ if and only if $n\le w,h$ \cite{plachta2021configuration}; by $S_{n}$-equivariance, a similar statement holds for $SC_{n}(R_{w,h})$ and $C_{n}(\R^{2})$.
A major advancement in the study of square configuration spaces came with the work of Alpert, Bauer, Kahle, MacPherson, and Spendlove, who proved that $SF_{n}(R_{w,h})$ is homotopy equivalent to a discrete configuration space called $DF_{n}(R_{w-1,h-1})$ \cite[Theorem 3.5]{alpert2023homology}. 
Their proof shows that $SC_{n}(R_{w,h})$ is homotopy equivalent to $DC_{n}(R_{w-1,h-1})$, the $n$-th unordered configuration space of $R_{w-1,h-1}$, treated as a cube complex; see Figure~\ref{Configsin2by2}.
Unfortunately, there is no obvious homotopy equivalence between the $k$-skeleta of these complexes and any known $k$-skeleta of $F_{n}(\R^{2})$ and $C_{n}(\R^{2})$ for any $k$.
Despite this, these models have been a boon to the study of configurations of squares in a rectangle.

Alpert, Bauer, Kahle, MacPherson, and Spendlove used this discrete model to prove that the homology groups of the ordered square configuration spaces undergo phase transitions \cite{alpert2023homology}, and in certain ranges Alpert, Kahle, and MacPherson computed asymptotic Betti numbers and determined a bound on the homotopy dimension of these spaces \cite{alpert2023asymptotic}.
Also in the ordered setting, Gonz\'{a}lez, Kahle, and the second author used this discrete model to define a spectral sequence that yields an ordered analog of Theorem~\ref{space homological stability unordered} for $n\le wh-6$ \cite[Theorem 1.5]{gonzalez202615}, but they were unable to determine $H_{1}\big(SF_{n}(R_{w,h})\big)$ at the top of the range.
In the unordered setting, Alvarado-Gardu\~{n}o and Gonz\'{a}lez \cite{alvarado2025square} and Alvarado-Gardu\~{n}o, Gonz\'{a}lez, and Kahle \cite{MattOmarJesus} used discrete Morse theory to determine presentations of $SBr_{n}(R_{w,h})$ for $n$ near $wh$ if $h=3$ and for almost all $n$ if $h=2$.
We use two spectral sequences to build on these ideas, computing the homology of the square-section braid groups.

\begin{figure}[h]
\centering
\captionsetup{width=.8\linewidth}
\includegraphics[width=16cm]{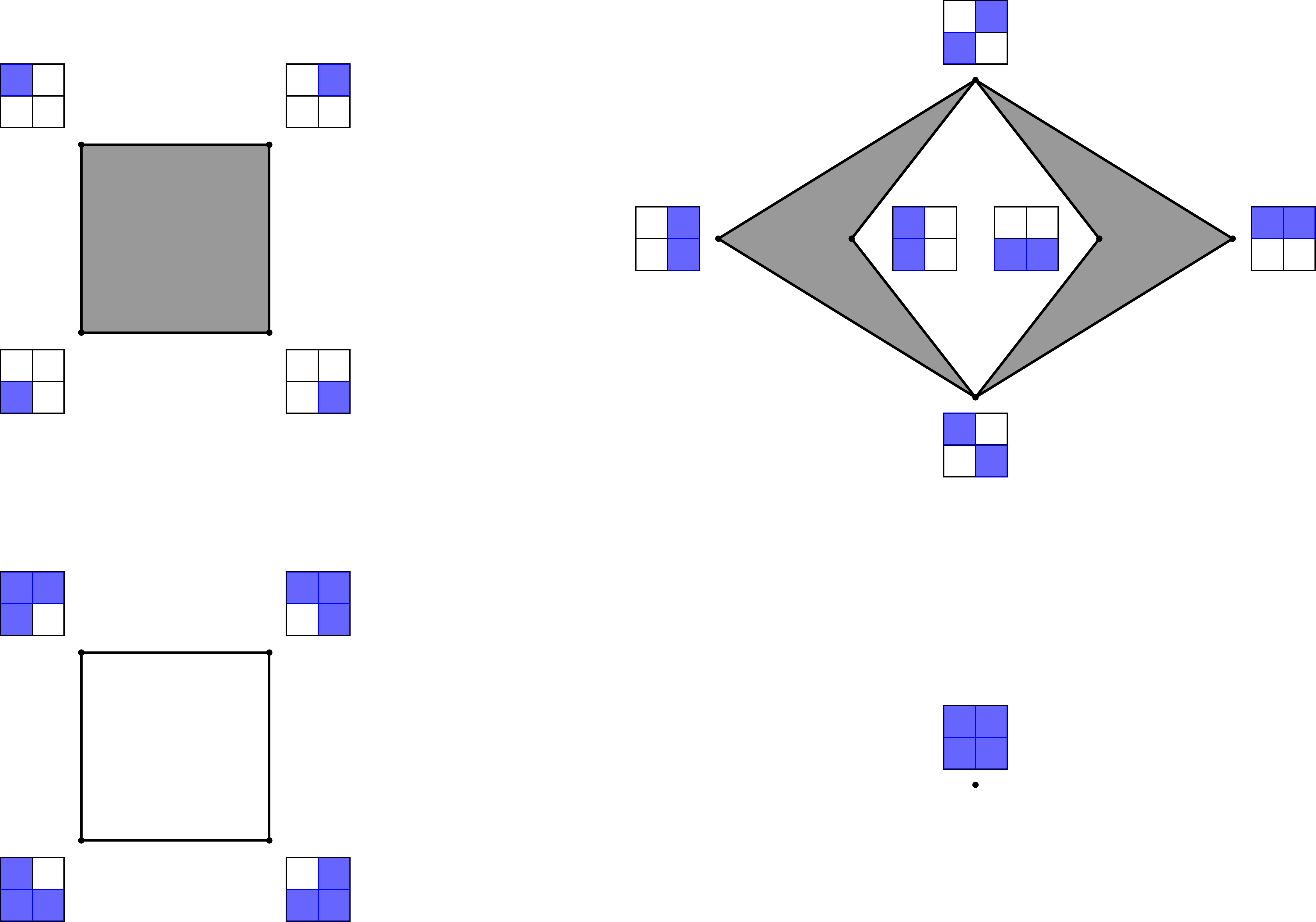}
\caption{The discrete models $DC_{n}(R_{1,1})$ homotopy equivalent to $SC_{n}(R_{2,2})$ for $1\le n\le 4$.
We have labeled the zero cells with the corresponding configuration.
}
\label{Configsin2by2}
\end{figure}

\subsection{Proof strategy and outline of the paper}

We prove our results in an inductive manner.
First, we prove Theorem~\ref{first homology of height 2 box} by inducting on how much empty space $2w-n$ there is in a configuration in $SC_{n}(R_{w,2})$.
Then, we use this as the base case of a second inductive argument on $h$ to prove Theorem~\ref{space homological stability unordered}.
The inductive step yields two types of generators of homology, the second of which proves to be trivial under the right conditions.

More formally, our proofs of Theorems~\ref{space homological stability unordered} and \ref{first homology of height 2 box} rely on two spectral sequences that we define in Section~\ref{spectral sequence section}.
The first of these spectral sequences, a geometric realization spectral sequence that we call the augmented right Mayer--Vietoris spectral sequence, is closely related to the augmented arc resolution spectral sequence of Kupers and Miller \cite{kupers2018n}.
This spectral sequence plays the role of the first Mayer--Vietoris spectral sequence of \cite{gonzalez202615}, which was used to compute the first homology of almost every ordered square configuration space, but does not directly generalize to the unordered setting.
In Subsection~\ref{right Mayer--Vietoris spectral sequence subsection} we define this spectral sequence, which calculates the homology of $SC_{n}(R_{w,h})$ from the homology of certain partially ordered configuration spaces $SC_{n}(R_{w,h};p)$ and $SC_{n}(R_{w,h};s;p)$.
Additionally, we prove several facts about the differentials on its first and second pages.
In Subsection~\ref{slide spectral sequence subsection}, we define the second spectral sequence, the augmented slide spectral sequence, which we use to calculate the entries of our augmented right Mayer--Vietoris spectral sequence.

We prove Theorem~\ref{first homology of height 2 box} in Section~\ref{h=2 section}.
This result is broken up into the cases $n=2w-1$, $n=2w-2$, $n=2w-3$, and $n\le 2w-4$.
The $n=2w-1$ case is handled by a theorem of Alvarado-Gardu\~{n}o, Gonz\'{a}lez, and Kahle \cite[Theorem 1.2]{MattOmarJesus} in Subsection~\ref{n=2w-1 subsection}.
The other three cases are proven via an inductive argument on $d:=2w-n$ and a three step process, and the details for each case are given in Subsections~\ref{n=2w-2 subsection}, \ref{n=2w-3 subsection}, and \ref{n le 2w-4 subsection}.
We describe this process in Subsection~\ref{three step ladder subsection}.
First, we use the augmented slide spectral sequence to compute $H_{1}\big(SC_{n}(R_{w,2};1;1)\big)$ from $H_{1}\big(SC_{n-2}(R_{w-2,2})\big)$ and $H_{1}\big(SC_{n-2}(R_{w-1,2})\big)$.
Then, we use this to determine $H_{1}\big(SC_{n}(R_{w,2};1)\big)$ via the augmented right Mayer--Vietoris spectral sequence.
Finally, we use another augmented right Mayer--Vietoris spectral sequence to determine $H_{1}\big(SC_{n}(R_{w,2})\big)$.
In addition to proving Theorem~\ref{first homology of height 2 box}, we give explicit bases for $H_{1}\big(SC_{n}(R_{w,2})\big)$.

In Section~\ref{section h>2} we consider rectangles of height $h\ge 3$ and prove Theorem~\ref{space homological stability unordered}.
To do this, we determine in detail the $E^{1}_{0,1}$- and $E^{2}_{0,1}$-entries of the augmented right Mayer--Vietoris spectral sequence of $SC_{n}(R_{w,h};p)$ as these are the two entries that contribute to $H_{1}\big(SC_{n}(R_{w,h};p)\big)$.
We show that the $E^{2}_{0,1}$-entry can contribute certain exotic classes in Subsection~\ref{E101 entry subsection}, and that these entries come from the augmented slide spectral sequence.
Subsection~\ref{generation by orbiting pair subsection} contains the main technical lemma of the paper (Lemma~\ref{intermediate generated by orbiting square classes}), which we call the descent lemma.
This lemma proves that if $H_{1}\big(SC_{n}(R_{w,h};p+1)\big)$ is generated by exotic classes and well-behaved orbiting pair square classes, then so is $H_{1}\big(SC_{n}(R_{w,h};p)\big)$.
Using Theorem~\ref{first homology of height 2 box} as a starting point, we use this lemma to inductively prove that all of the $H_{1}\big(SC_{n}(R_{w,h};p)\big)$ are generated by these two types of classes.
Proposition~\ref{d1 on exotic classes is 0} shows the exotic classes are trivial in $H_{1}\big(SC_{n}(R_{w,h})\big)$.
In Subsection~\ref{orbiting pair square classes are homologous subsection}, we prove that we only need a single orbiting pair square class to generate $H_{1}\big(SC_{n}(R_{w,h})\big)$ for all $w,h\ge 3$ and $2\le n\le wh-3$.
Finally, Subsection~\ref{Proof of Theorem space homological stability unordered subsection} completes the proof of Theorem~\ref{space homological stability unordered} by noting that an orbiting pair square class generator of $H_{1}\big(SC_{n}(R_{w,h})\big)$ maps to the generator of $H_{1}\big(C_{n}(\R^{2})\big)$ under the inclusion $SC_{n}(R_{w,h})\hookrightarrow C_{n}(\R^{2})$.
As $H_{1}\big(C_{n}(\R^{2})\big)$ is isomorphic to $\Z$, it follows that $H_{1}\big(SC_{n}(R_{w,h})\big)$ must be as well.

In Section~\ref{calcs and conjs} we perform further calculations and use them to make several conjectures about the square-section braid groups and their homology groups.
First, noting that the augmented right Mayer--Vietoris spectral sequence is equivalent to the augmented arc resolution spectral sequence in certain situations, we compute the differentials of the latter with $\Z/2\Z$-coefficients.
This is done in Subsection~\ref{higher differential subsection}.
Next, in Subsection~\ref{computer calculations subsection} we give a table of the $\Z/2\Z$ and rational Betti numbers of $SC_{n}(R_{w,h})$ for several $w$, $h$, and $n$.
Finally in Subsection~\ref{Conjectures and future directions subsection}, we use these calculations and several observations about the amount of unoccupied space in $SC_{n}(R_{w,h})$ to make several conjectures about the homology of $SC_{n}(R_{w,h})$ and the square-section braid groups $SBr_{n}(R_{w,h})$, as well as suggesting future research directions.

\subsection{Acknowledgements and AI disclosure}

The authors would like to thank the University of Chicago’s Research Computing Center for their help in carrying out the computations appearing in Subsection~\ref{computer calculations subsection}.

The authors used Claude's Opus 5 and Opus 5.5 in August, September, and October 2026 to assist with the editing and organization of the paper. 
Other than computations used in Subsection~\ref{Conjectures and future directions subsection} to make conjectures, Claude contributed no new mathematics. 
All statements, proofs, and computations were verified by the authors, who take full responsibility for the contents of the paper.

\section{Two spectral sequences}\label{spectral sequence section}

In this section, we define the two spectral sequences used to prove Theorems~\ref{space homological stability unordered} and \ref{first homology of height 2 box}.
We use the first to compute the homology of $SC_{n}(R_{w,h})$; we use the second to calculate the entries of the first.
In addition to defining these spectral sequences, we prove several results that will simplify the calculations of $H_{1}\big(SC_{n}(R_{w,h})\big)$ in the subsequent sections.

\subsection{The augmented right Mayer--Vietoris spectral sequence}\label{right Mayer--Vietoris spectral sequence subsection}

The first spectral sequence is a geometric realization spectral sequence inspired by the first Mayer--Vietoris spectral sequence used in \cite{gonzalez202615} and the arc resolution spectral sequence defined in \cite{kupers2018n}.
After defining this spectral sequence, which arises by considering a semisimplicial complex that naturally maps to $SC_{n}(R_{w,h})$, we prove several results about its first page to simplify the calculations of Sections~\ref{h=2 section} and \ref{section h>2}.
We begin by defining the spaces that come together to form the semisimplicial complex defining the spectral sequence.

\begin{defn}\label{rightmost config space unordered}
Given a $k$-tuple $\lambda=(\lambda_{1},\dots, \lambda_{k})$ of non-negative integers, let $|\lambda|:=\lambda_{1}+\cdots+\lambda_{k}$.
We write $SC_{\lambda}(R_{w,h})$ for the \emph{partially ordered configuration space of $|\lambda|$ open unit-squares in the $w$ by $h$ rectangle $R_{w,h}$ such that $\lambda_{j}$ squares are labeled $j$}.
We write \emph{$SC_{\lambda}(R_{w,h}; \lambda_{l};\lambda_{l+1};\dots; \lambda_{k})$} for the subspace of $SC_{\lambda}(R_{w,h})$ consisting of configurations in which the only squares that are to the right, i.e., have larger $x$-coordinates, of a square labeled $j\ge l$ are those labeled $i$ where $i>j$.
If $l=2$, we write $SC_{|\lambda|}(R_{w,h}; \lambda_{l};\lambda_{l+1};\dots; \lambda_{k})$ for $SC_{\lambda}(R_{w,h}; \lambda_{l};\lambda_{l+1};\dots; \lambda_{k})$.
\end{defn}

\begin{remark}
In $SC_{\lambda}(R_{w,h}; \lambda_{l};\dots; \lambda_{k})$ all squares labeled a fixed $i\ge l$ are aligned horizontally (in a column), and the squares labeled $j<l$ can be anywhere as long as they do not pass to the right of any square with label at least $l$.
\end{remark}

These partially ordered square configuration spaces interpolate between unordered and ordered square configuration spaces: If $|\lambda|=n$ and $\lambda=(1,\dots, 1)$, then $SC_{\lambda}(R_{w,h})$ is the ordered square configuration space $SF_{n}(R_{w,h})$; if $\lambda=(n)$, then $SC_{\lambda}(R_{w,h})=SC_{n}(R_{w,h})$.
In \cite{gonzalez202615},
Gonz\'{a}lez, Kahle, and the second author covered $SF_{n}(R_{w,h})$ with ordered variants of these partially ordered configuration spaces to produce an augmented Mayer--Vietoris spectral sequence computing the homology of $SF_{n}(R_{w,h})$.
While no analogous cover exists in the unordered setting, the partially ordered configuration spaces can be used to form a semisimplicial complex whose homology agrees with that of $SC_{n}(R_{w,h})$.

\begin{defn}\label{arc complex defn}
Given $\lambda=(\lambda_{1},\dots, \lambda_{k})$ and $\lambda'=(\lambda_{1}+\lambda_{2}, 0, \lambda_{3},\dots, \lambda_{k})$, let
\[
MV_{m}\big(SC_{\lambda}(R_{w,h}; \lambda_{3};\dots; \lambda_{k})\big):=SC_{|\lambda|}(R_{w,h}; m+1; \lambda_{3};\dots; \lambda_{k}).
\]
Letting $m$ vary, the $MV_{m}\big(SC_{\lambda}(R_{w,h}; \lambda_{3};\dots; \lambda_{k})\big)$ come together to form an augmented semisimplicial space $MV_{\bullet}\big(SC_{\lambda}(R_{w,h}; \lambda_{3};\dots; \lambda_{k})\big)$.
Ordering the $m+1$
squares labeled $2$ from top to bottom, the $l$-th face map 
\[
d_{l}:MV_{m}\big(SC_{\lambda}(R_{w,h}; \lambda_{3};\dots; \lambda_{k})\big)\to MV_{m-1}\big(SC_{\lambda}(R_{w,h}; \lambda_{3};\dots; \lambda_{k})\big)
\]
changes the label of the $l$-th square labeled $2$ to $1$.
Additionally, as $MV_{-1}\big(SC_{\lambda}(R_{w,h}; \lambda_{3};\dots; \lambda_{k})\big)=SC_{\lambda'}(R_{w,h}; \lambda_{3};\dots; \lambda_{k})\simeq SC_{|\lambda'|}(R_{w,h}; \lambda_{3};\dots; \lambda_{k})$, there is an augmentation map
\[
\Big\|MV_{\bullet}\big(SC_{\lambda}(R_{w,h}; \lambda_{3};\dots; \lambda_{k})\big)\Big\|\to SC_{|\lambda'|}(R_{w,h}; \lambda_{3};\dots; \lambda_{k}).
\]
We call this complex the \emph{right Mayer--Vietoris complex of $SC_{|\lambda'|}(R_{w,h}; \lambda_{3};\dots; \lambda_{k})$}.
\end{defn}

In general, an augmented semisimplicial space $A_{\bullet}$ gives rise to an augmented spectral sequence that satisfies
\[
E^{1}_{s,t}=H_{t}(A_{s})\implies H_{s+t+1}\big(A_{-1},\|A_{\bullet}\|\big);
\]
for a proof of this, see, for example, \cite[Section 2.3]{randal2013homological}.
We will focus on a pair of these augmented geometric realization spectral sequences that come from right Mayer--Vietoris complexes and compute the homology of $SC_{n}(R_{w,h})$ and $SC_{n}(R_{w,h};p)$.

\begin{defn}
The augmented geometric realization spectral sequence arising from the right Mayer--Vietoris complex of $SC_{n}(R_{w,h}; p)$ is the \emph{augmented right Mayer--Vietoris spectral sequence} of $SC_{n}(R_{w,h}; p)$.
We write $E^{r}_{s,t}\big[SC_{n}(R_{w,h}; p)\big]$ for the $(s,t)$-entry on the $r$-th page of this spectral sequence.
See Figure~\ref{right MV spectral sequence for SC} for the $E^{1}$-page of the augmented right Mayer--Vietoris spectral sequence for $SC_{n}(R_{w,h};p)$.
\end{defn}

\begin{figure}[h!] 
\centering
\begin{tikzpicture}
\footnotesize

\matrix (m) [
  matrix of math nodes,
  nodes in empty cells,
  nodes={
    minimum width=3em,
    minimum height=3em,
    anchor=center
  },
  row sep=0ex,
  column sep=1ex
]{
\vdots &
\vdots &
\vdots &
&
\vdots &
\\
t &
H_{t}\big(SC_{n}(R_{w,h};p)\big) &
H_{t}\big(SC_{n}(R_{w,h};1;p)\big) &
\cdots &
H_{t}\big(SC_{n}(R_{w,h};s+1;p)\big) &
\cdots
\\
\vdots &
\vdots &
\vdots &
&
\vdots
\\
t=1 &
H_{1}\big(SC_{n}(R_{w,h};p)\big) &
H_{1}\big(SC_{n}(R_{w,h};1;p)\big) &
\cdots &
H_{1}\big(SC_{n}(R_{w,h};s+1;p)\big) &
\cdots
\\
t=0 &
H_{0}\big(SC_{n}(R_{w,h};p)\big) &
H_{0}\big(SC_{n}(R_{w,h};1;p)\big) &
\cdots &
H_{0}\big(SC_{n}(R_{w,h};s+1;p)\big) &
\cdots
\\
&
s=-1 &
s=0 &
\cdots &
s &
\cdots
\\
};

\draw[thick]
  ($(m-1-1.east)+(0em,3ex)$)
  --
  ($(m-6-1.east)+(0em,0ex)$);

\draw[thick]
  ($(m-5-1.south)+(-2em,-0.5ex)$)
  --
  ($(m-5-6.south)+(1em,-0.5ex)$);
\end{tikzpicture}
\caption{The potentially nonzero $E^{1}_{s,t}$-entries of the augmented right Mayer--Vietoris spectral sequence of $SC_{n}(R_{w,h};p)$.} 
\label{right MV spectral sequence for SC}
\end{figure}

This spectral sequence recovers $H_{k}\big(SC_{n}(R_{w,h};p)\big)$:

\begin{prop}\label{augmented right MVSS has trivial abutment}
The abutment of the augmented right Mayer--Vietoris spectral sequence of $SC_{n}(R_{w,h};p)$ is trivial.
\end{prop}

\begin{proof}
Gonz\'{a}lez, Kahle, and the second author defined an augmented Mayer--Vietoris spectral sequence \cite[Section 3.2]{gonzalez202615} that is equivalent to the geometric realization spectral sequence on the ordered analog of $MV_{\bullet}\big(SC_{n}(R_{w,h}; p)\big)$, so the resulting augmented spectral sequence has trivial abutment.
Moreover, this complex is a finite dimensional acyclic complex of free $\Z[S_{n}]$-modules.
The $S_{n}$-action is free on this space, 
and the quotient is $MV_{\bullet}\big(SC_{n}(R_{w,h}; p)\big)$.
As such, the abutment of the augmented right Mayer--Vietoris spectral sequence of $SC_{n}(R_{w,h};p)$ is trivial.
\end{proof}

Proposition~\ref{augmented right MVSS has trivial abutment} tells us that we can compute $H_{1}\big(SC_{n}(R_{w,h})\big)\cong E^{1}_{-1,1}\big[SC_{n}(R_{w,h})\big]$ as long as we understand the $E^{1}_{0,1}$- and $E^{2}_{1,0}$-entries and the $d^{1}$- and $d^{2}$-differentials of the augmented right Mayer--Vietoris spectral sequence.
Next, we determine conditions under which the $E^{2}_{1,0}$-entry is $0$ ensuring that
\[
d^{1}:E^{1}_{0,1}\big[SC_{n}(R_{w,h})\big]\cong H_{1}\big(SC_{n}(R_{w,h};1)\big)\to H_{1}\big(SC_{n}(R_{w,h})\big)\cong E^{1}_{-1,1}\big[SC_{n}(R_{w,h})\big]
\]
is a surjection.
Since the $E^{1}_{p, 0}$-entry of this spectral sequence is $H_{0}\big(SC_{n}(R_{w,h}; p+1)\big)$, we determine conditions ensuring that $SC_{n}(R_{w,h}; p+1)$ is connected.

\begin{prop}\label{connected unordered square configuration spaces one rightmost}
If $w,h\ge 2$, $n<wh$, and $0\le p\le \min\{h, n\}$, the square configuration space $SC_{n}(R_{w,h}; p)$ is connected.
\end{prop}

\begin{proof}
We will show that the squares labeled $2$ can be moved to the top of the rightmost column in $R_{w,h}$.
Since all of the squares labeled $2$ are indistinguishable, this would prove that $SC_{n}(R_{w,h}; p)$ is connected as the squares labeled $1$ can be freely arranged to any other configuration in the rest of the rectangle.

Given any configuration in $SC_{n}(R_{w,h}; p)$, we may slide the squares labeled $2$ into the rightmost column of $R_{w,h}$. 
If we are given a configuration in which the squares labeled $2$ are not the top $p$ squares in the rightmost column, then there is at least one square labeled $1$ in the top $p$ spots in this column. 
Since $n<wh$, we may slide the squares labeled $1$ not in the rightmost column so that the square immediately to the left of the topmost square labeled $1$ in the rightmost column is unoccupied.
Therefore, we may slide it out of the rightmost column, allowing us to slide the remaining squares in that column up one.
Repeating this process if necessary, we see that the squares labeled $2$ can be moved to the top of the rightmost column in $R_{w,h}$, proving that $SC_{n}(R_{w,h}; p)$ is connected.
\end{proof}

It follows that $E^{1}_{p,0}\big[SC_{n}(R_{w,h})\big]\cong \Z$ for $-1\le p<\min\{h, n\}$.
Since $SC_{n}(R_{w,h};p)$ is empty for $p>\min\{h,n\}$, we see that the only way that $E^{2}_{1,0}\big[SC_{n}(R_{w,h})\big]$ can be nonzero is if $n=2$ or $h=2$; the latter possibility is the subject of Section~\ref{h=2 section}, and the former appears in Section~\ref{section h>2}.
Other than these cases, the map 
\[
d^{1}:E^{1}_{0,1} = H_{1}\big(SC_{n}(R_{w,h};1)\big)\to H_{1}\big(SC_{n}(R_{w,h})\big)= E^{1}_{-1,1}
\]
is a surjection, so it would be beneficial to understand  $H_{1}\big(SC_{n}(R_{w,h};1)\big)$.
We will do this by considering the augmented right Mayer--Vietoris spectral sequence for $SC_{n}(R_{w,h};1)$.
As in the simpler setting, in most cases the map 
\[
d^{1}:E^{1}_{0,1} = H_{1}\big(SC_{n}(R_{w,h};1;1)\big)\to H_{1}\big(SC_{n}(R_{w,h};1)\big)= E^{1}_{-1,1}
\]
is a surjection.
This turns the problem of determining $H_{1}\big(SC_{n}(R_{w,h};1)\big)$ into the problem of determining $H_{1}\big(SC_{n}(R_{w,h};1;1)\big)$.
Though we could try to repeat this process, computing more and more of the $H_{1}\big(SC_{n}(R_{w,h};1;\dots;1)\big)$, we would like to avoid this as the $SC_{n}(R_{w,h};1;\dots;1)$ get progressively more complex.
We do so with a second type of spectral sequence, which we describe in the next subsection. 

\subsection{The augmented slide spectral sequence}\label{slide spectral sequence subsection}

As mentioned at the end of the previous subsection, computing $H_{1}\big(SC_{n}(R_{w,h})\big)$ via the  augmented right Mayer--Vietoris spectral sequence often comes down to computing $H_{1}\big(SC_{n}(R_{w,h}; 1)\big)$, which in turn can be done by computing $H_{1}\big(SC_{n}(R_{w,h}; 1; 1)\big)$ via another augmented right Mayer--Vietoris spectral sequence.
To avoid a long recursive process, in this subsection we introduce a second type of spectral sequence that calculates $H_{1}\big(SC_{n}(R_{w,h}; 1; p)\big)$.
Unlike the first spectral sequence, this is a true Mayer--Vietoris spectral sequence as it arises from a cover of $SC_{n}(R_{w,h}; 1; p)$.
We begin with the following observation:
if $p+1\le h$, there are $p+1$ ways of shuffling an ordered set consisting of a single element into an ordered set consisting of $p$ elements.
We use this observation to define sets $V_{l}$ that cover $SC_{n}(R_{w,h};1;p)$ leading to an augmented Mayer--Vietoris spectral sequence.

\begin{defn}\label{V sets for slide ss}
If $0\le p<h$ and $1\le l\le p+1$, let $V_{l}$ be the union of the following two subspaces of $SC_{n}(R_{w,h}; 1; p)$:
\begin{itemize}
\item The subspace of configurations in which, after sliding  the squares labeled $3$ up and down, the square labeled $2$ can be slid directly to the right so that it is horizontally aligned with the squares labeled $3$ in a way so that of the squares labeled $2$ and $3$, it is in the $l$-th position when read from top to bottom.
\item The subspace of configurations in which the center of the square labeled $2$ straddles two rows, such that after sliding this square and the other squares that are directly above or below it up or down to be vertically aligned with either of the two rows of grid squares that partially contain it, the squares labeled $3$ can be slid up or down so that the square labeled $2$ can be slid to the right to be horizontally aligned with the squares labeled $3$ so that it is in the $l$-th position when read from top to bottom.
\end{itemize}

The sets $V_{l}$ cover $SC_{n}(R_{w,h}; 1; p)$, and we call the resulting augmented Mayer--Vietoris spectral sequence for $SC_{n}(R_{w,h}; 1;p)$ the \emph{augmented slide spectral sequence}.
We write $\tilde{E}^{r}_{s,t}\big[SC_{n}(R_{w,h}; 1; p)\big]$ for the $(s,t)$-entry on the $r$-th page of this spectral sequence.
\end{defn}

\begin{remark}
The cover $\{V_{l}\}_{l=1}^{p+1}$ of $SC_{n}(R_{w,h};1;p)$ is cellular ensuring the well-definedness of the augmented slide spectral sequence.
This cellularity is obtained by subdividing the complex $DC_{n}(R_{w-1,h-1}; 1; p)$.
Here,  $DC_{n}(R_{w-1,h-1}; 1; p)$ is the quotient of a subspace of $DF_{n}(R_{w-1,h-1})$, which is Alpert, Bauer, Kahle, MacPherson, and Spendlove's cubical model for $SF_{n}(R_{w, h})$
\cite{alpert2023homology}.
For more details see \cite[Section 2]{gonzalez202615}.
\end{remark}

By definition, the entries of the first page of a Mayer--Vietoris spectral sequence for a space $X$ arising from a cellular cover $\{U_{i}\}_{i\in I}$ are of the form \[
\bigoplus_{J\subset I\big||J|=s+1}H_{t}\left(\bigcap_{j\in J}U_{j}\right);
\]
as a result, the augmented slide spectral sequence takes the form depicted in Figure~\ref{general slide spectral sequence}.

\begin{figure}[h!] 
\centering
\begin{tikzpicture}
\footnotesize

\matrix (m) [
  matrix of math nodes,
  nodes in empty cells,
  nodes={
    minimum width=3em,
    minimum height=3em,
    anchor=center
  },
  row sep=0ex,
  column sep=1ex
]{
\vdots &
\vdots &
\vdots &
&
\vdots &
\\
t &
H_{t}\big(SC_{n}(R_{w,h};1;p)\big) &
\bigoplus_{1\le l_{0}\le p+1}H_{t}(V_{l_{0}}) &
\cdots &
\bigoplus_{1\le l_{0}<\cdots<l_{p}\le p+1}H_{t}\big(\bigcap_{i=0}^{p}V_{l_{i}}\big) &
0 &
\cdots
\\
\vdots &
\vdots &
\vdots &
&
\vdots
\\
t=1 &
H_{1}\big(SC_{n}(R_{w,h};1;p)\big) &
\bigoplus_{1\le l_{0}\le p+1}H_{1}(V_{l_{0}}) &
\cdots &
\bigoplus_{1\le l_{0}<\cdots<l_{p}\le p+1}H_{1}\big(\bigcap_{i=0}^{p}V_{l_{i}}\big) &
0 &
\cdots
\\
t=0 &
H_{0}\big(SC_{n}(R_{w,h};1;p)\big) &
\bigoplus_{1\le l_{0}\le p+1}H_{0}(V_{l_{0}}) &
\cdots &
\bigoplus_{1\le l_{0}<\cdots<l_{p}\le p+1}H_{0}\big(\bigcap_{i=0}^{p}V_{l_{i}}\big) &
0 &
\cdots
\\
&
s=-1 &
s=0 &
\cdots &
s=p &
s=p+1 &
\cdots
\\
};

\draw[thick]
  ($(m-1-1.east)+(0em,3ex)$)
  --
  ($(m-6-1.east)+(0em,0ex)$);

\draw[thick]
  ($(m-5-1.south)+(-2em,-0.5ex)$)
  --
  ($(m-5-7.south)+(1em,-0.5ex)$);
\end{tikzpicture}
\caption{The first page of the augmented slide spectral sequence of $SC_{n}(R_{w,h};1;p)$.} 
\label{general slide spectral sequence}
\end{figure} 

By the definition of $V_{l}$, one can slide the square labeled $2$ to the right to be aligned with the squares labeled $3$ so that it is in the $l$-position among these squares.
This gives a deformation retraction of $V_{l}$ onto a subspace of $SC_{n}(R_{w,h};1;p)$ that is homotopy equivalent to $SC_{n}(R_{w,h};p+1)$, which proves the following proposition:

\begin{prop}\label{Vl is SCp2}
The sets $V_{l}$ used to define the augmented slide spectral sequence of $SC_{n}(R_{w,h}; 1;p)$ are homotopy equivalent to $SC_{n}(R_{w,h}; p+1)$.
\end{prop}

In Section~\ref{h=2 section}, we show that if $h=2$, the map
\[
\tilde{d}^{1}:\tilde{E}^{1}_{0,1} \cong H_{1}\big(SC_{n}(R_{w,2};p+1)\big)^{\oplus{p+1}}\to H_{1}\big(SC_{n}(R_{w,2};1;p)\big)=\tilde{E}^{1}_{-1,1}
\]
is a surjection.
This allows us to convert information about $H_{1}\big(SC_{n}(R_{w,2};p+1)\big)$ into information about $H_{1}\big(SC_{n}(R_{w,2};1;p)\big)$.
For $h>2$, this differential does not have to be a surjection, and the $\tilde{E}^{2}_{1,0}$-entry of the augmented slide spectral sequence of $SC_{n}(R_{w,h};1;p)$ comes into play.
In Section~\ref{section h>2}, we control this entry to produce a description of $H_{1}\big(SC_{n}(R_{w,h};1;p)\big)$.
In both the $h=2$ and $h>2$ cases, we use this knowledge of $H_{1}\big(SC_{n}(R_{w,h};1;p)\big)$ and the augmented right Mayer--Vietoris spectral sequence to determine $H_{1}\big(SC_{n}(R_{w,h};p)\big)$.
Setting $p=0$ yields a description of $H_{1}\big(SC_{n}(R_{w,h})\big)$.

\section{$H_{1}\big(SC_{n}(R_{w,2})\big)$}\label{h=2 section}

In this section, we calculate $H_{1}\big(SC_{n}(R_{w,2})\big)$ for all $n$ and all $w\ge2$, that is, we prove Theorem~\ref{first homology of height 2 box}, which we now recall.

\begin{T2}
  \thmtextone
\end{T2}

Let $d:=2w-n$ denote the amount of unoccupied space in a configuration of $n$ open unit-squares in $R_{w,2}$.
In terms of empty space, Theorem~\ref{first homology of height 2 box} tells us $H_{1}\big(SC_{n}(R_{w,2})\big)\cong\Z^{n-1}$ as long as $d\ge2$.
If $d=1$, then $H_{1}\big(SC_{n}(R_{w,2})\big)\cong\Z^{w-1}$, and if $d=0$, then $H_{1}\big(SC_{n}(R_{w,2})\big)$ is trivial.
Indeed, the parameter $d$, and not $n$ or $w$ separately, is what controls the topology, as it measures how much space there is to move the squares.

We prove Theorem~\ref{first homology of height 2 box} via a double induction.
First, we fix $d$ and induct on $w$, then we induct on $d$.
The base case of the induction on $d$ is Proposition~\ref{first homology height 2 for 2w-1}, and it is a result of Alvarado-Gardu\~{n}o, Gonz\'{a}lez, and Kahle \cite[Theorem 1.2]{MattOmarJesus}.
The proof of inductive step depends on $d-1$, and we must consider the $d=2$, $d=3$, and $d\ge 4$ cases separately.
Fortunately, all three cases are handled by the same three step process, which we describe in detail in Subsection~\ref{three step ladder subsection}, leaving the case analysis to Subsections~\ref{n=2w-2 subsection},~\ref{n=2w-3 subsection}, and~\ref{n le 2w-4 subsection}.

Alvarado-Gardu\~{n}o, Gonz\'{a}lez, and Kahle prove that if $n\le 2w-5$, there is an isomorphism $SBr_{n}(R_{w,2})\cong SBr_{n}(R_{n,2})$, in addition to providing a presentation of the latter \cite[Theorem 1.5]{MattOmarJesus}. 
Similarly, Alvarado-Gardu\~{n}o and  Gonz\'{a}lez find a presentation for $SBr_{2w-2}(R_{w,2})$ \cite[Theorem 2.2]{alvarado2025square}.
While we could use these results for the $d=2$ and $d\ge 5$ cases to get a computation of $H_{1}\big(SC_{n}(R_{w,2})\big)$, their techniques do not yield the explicit bases that our results provide, which we need to prove Theorem~\ref{space homological stability unordered} in Section~\ref{section h>2}.

Before proceeding, we record what does and does not depend on $d$.
Though $H_{1}\big(SC_{n}(R_{w,2})\big)$ is uniform for $d\ge 2$, the auxiliary groups $H_{1}\big(SC_{n}(R_{w,2};1)\big)$ and $H_{1}\big(SC_{n}(R_{w,2};1;1)\big)$ appearing in the three steps and the two types of classes needed to generate them vary:

\medskip
 
\begin{center}
\begin{tabular}{c|c|c|c|c}
$d = 2w-n$ & $H_{1}\big(SC_{n}(R_{w,2})\big)$ & $H_{1}\big(SC_{n}(R_{w,2};1)\big)$ & $H_{1}\big(SC_{n}(R_{w,2};1;1)\big)$ & generators\\ 
\hline
0 & 0 & 0& 0 & none\\
$1$      & $\Z^{w-1}$ & ---                 & ---                  & triples\\
$2$      & $\Z^{n-1}$ & $\Z^{3w-6}$ & $\Z^{4w-10}$ & pairs and triples\\
$3$      & $\Z^{n-1}$ & $\Z^{n-2}$  & $\Z^{3w-9}$  & pairs\\
$\ge 4$  & $\Z^{n-1}$ & $\Z^{n-2}$  & $\Z^{n-3}$   & pairs\\
\end{tabular}
\end{center}
\medskip
Reading down the $H_{1}\big(SC_{n}(R_{w,2};1)\big)$ and $H_{1}\big(SC_{n}(R_{w,2};1;1)\big)$ columns shows why we must consider the $d=1$, $d=2$, $d=3$, and $d\ge4$ cases separately.
Both auxiliary groups deviate from the generic values $\Z^{n-2}$ and $\Z^{n-3}$ at $d = 2$, and $H_{1}\big(SC_{n}(R_{w,2};1;1)\big)$ deviates at $d = 3$.
Each step absorbs one unit of the anomaly, so one can see why three non-trivial cases occur and why orbiting triple square classes are needed when $d \le 2$.

Before describing our three step process, we define the orbiting pair and triple classes that generate $H_{1}\big(SC_{n}(R_{w,2})\big)$, and show how they arise from the $d^{2}$-differential of the augmented right Mayer--Vietoris spectral sequence.

\begin{defn}
An \emph{orbiting pair square class} $\sigma\in H_{1}\big(SC_{2}(R_{2,2})\big)$ is the class depicted on the left in Figure~\ref{Orbiting Square Classes}.
An \emph{orbiting triple square class} $\theta\in H_{1}\big(SC_{3}(R_{2,2})\big)$ is the class depicted on the right in Figure~\ref{Orbiting Square Classes}.

\begin{figure}[h]
\centering
\captionsetup{width=.8\linewidth}
\includegraphics[width=10cm]{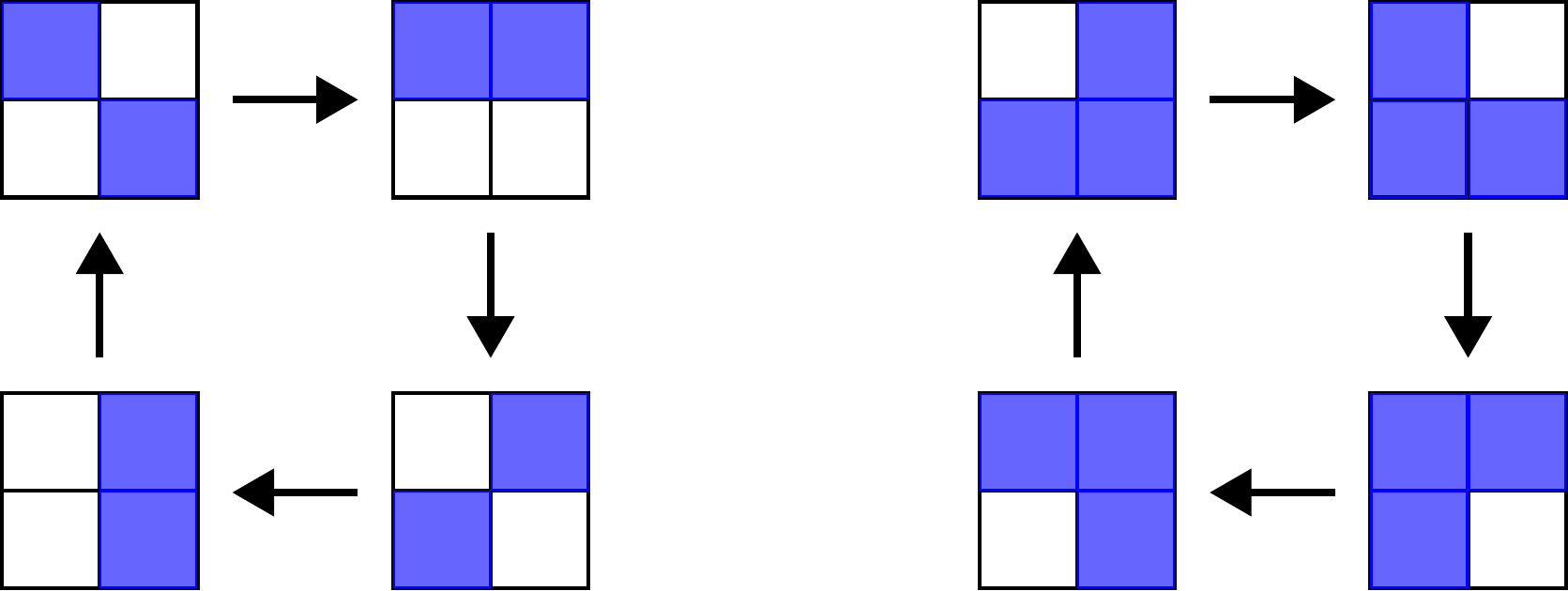}
\caption{On the left, an orbiting pair square class. 
On the right, an orbiting triple square class.
}
\label{Orbiting Square Classes}
\end{figure}

Embedding $R_{2,2}$ into $R_{w,h}$ yields classes in $H_{1}\big(SC_{n}(R_{w,h};s; p)\big)$ that we also call
\emph{orbiting pair} and \emph{triple square classes}. 
Inside the embedded $R_{2,2}$ two or three squares orbit each other.
The remaining $n-2$ or $n-3$ squares are fixed outside of the embedded $R_{2,2}$ and are called \emph{singletons}. 
See Figure~\ref{Embedded Orbiting Pair Class} for an example.
Additionally, in $H_{1}\big(SC_{n}(R_{w,2};s; p)\big)$, we say that an orbiting pair (triple) square class is \emph{even} if there is an even number of singleton squares to the left of the orbiting pair (triple).

\begin{figure}[h]
\centering
\captionsetup{width=.8\linewidth}
\includegraphics[width=5cm]{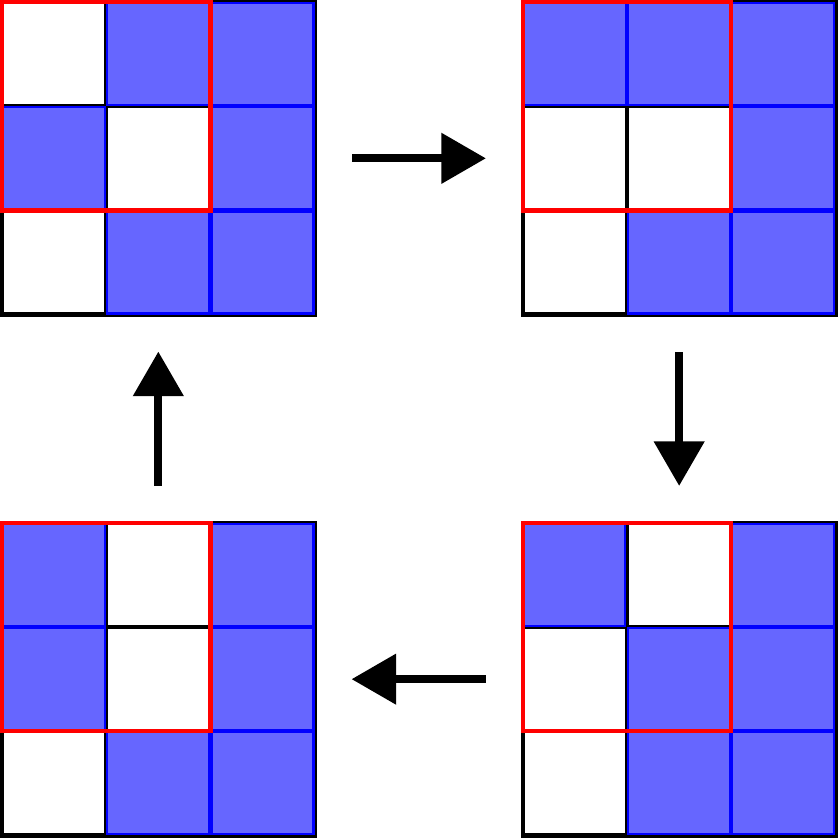}
\caption{An orbiting pair square class in $H_{1}\big(SC_{6}(R_{3,3})\big)$.
The class arises from embedding $R_{2,2}$ at the top left of $R_{3,3}$, which we have outlined in red.
There are four singleton squares.
}
\label{Embedded Orbiting Pair Class}
\end{figure}
\end{defn}

\begin{remark}
Returning to the language of the braid group, orbiting pair and triple square classes have nice interpretations.
Namely, under the inclusion
\[
SC_{n}(R_{w,h})\hookrightarrow C_{n}(\R^{2})
\]
an orbiting pair square class gets sent to the image of a half-twist generator of the braid group $Br_{n}$, and an orbiting triple square class gets sent to the cube root of the image of a generator of the center of $Br_{3}\subset Br_{n}$. 
\end{remark}

We conclude by noting that the orbiting pair and triple square classes arise from the $d^{2}$-differential of the augmented right Mayer--Vietoris spectral sequence.

\begin{prop}\label{spectral sequence d2 differential }
The image of 
\[
d^{2}:E^{2}_{1,0}\big[SC_{2}(R_{2,2})\big]\to E^{2}_{-1,1}\big[SC_{2}(R_{2,2})\big]
\]
is an orbiting pair square class.
Similarly, the image of 
\[
d^{2}:E^{2}_{1,0}\big[SC_{3}(R_{2,2})\big]\to E^{2}_{-1,1}\big[SC_{3}(R_{2,2})\big]
\]
is an orbiting triple square class.
\end{prop}

\begin{proof}
In the $SC_{2}(R_{2,2})$ setting, note that $E^{1}_{1,0}\big[SC_{2}(R_{2,2})\big]=H_{0}\big(SC_{2}(R_{2,2};2)\big)\cong \Z$, and it is generated by the class which can be represented by the $0$-chain on the left of Figure~\ref{Orbitingpairsquaredifferential1}.
The image of this class under the $d^{1}$-differential is the difference of the two classes depicted on the right of Figure~\ref{Orbitingpairsquaredifferential1}.
As these classes are homologous and $H_{0}\big(SC_{2}(R_{2,2};3)\big)$ is trivial by definition, we have that $E^{2}_{1,0}\big[SC_{2}(R_{2,2})\big]\cong \Z$.
\begin{figure}[h]
\centering
\captionsetup{width=.8\linewidth}
\includegraphics[width=7.5cm]{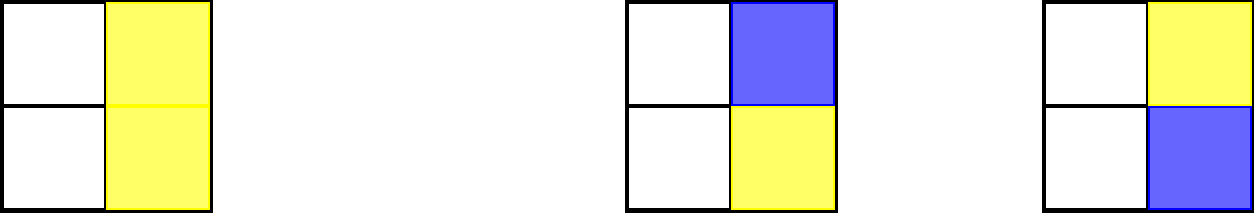}
\caption{On the left, a chain representing the generator of $E^{1}_{1,0}\big[SC_{2}(R_{2,2})\big]=H_{0}\big(SC_{2}(R_{2,2};2)\big)$.
On the right, the two chains in $E^{1}_{0,0}\big[SC_{2}(R_{2,2})\big]=H_{0}\big(SC_{2}(R_{2,2};1)\big)$ whose difference is its image under the $d^{1}$-differential.
They are homologous.
We have colored the squares labeled $1$ blue and the squares labeled $2$ yellow.}
\label{Orbitingpairsquaredifferential1}
\end{figure}

The two classes on the right of Figure~\ref{Orbitingpairsquaredifferential1} are the boundary of the $1$-chain depicted on the left of Figure~\ref{Orbitingpairsquaredifferential2}, and so the image of this class in $E^{2}_{-1,1}\big[SC_{2}(R_{2,2})\big]$ under $d^{1}$ is the class depicted on the right of the figure.
That class is an orbiting pair square class, and our argument shows that it is the image of the generator of $E^{2}_{1,0}\big[SC_{2}(R_{2,2})\big]$ under the $d^{2}$-differential as desired.

\begin{figure}[h]
\centering
\captionsetup{width=.8\linewidth}
\includegraphics[width=3.75cm]{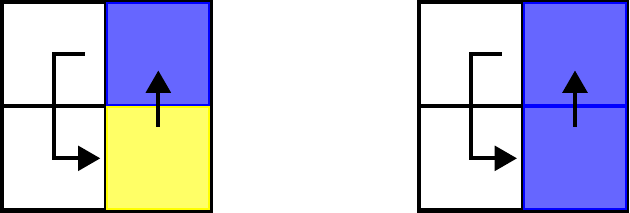}
\caption{On the left, the $1$-chain whose boundary is the difference of the two classes depicted on the right of Figure~\ref{Orbitingpairsquaredifferential1}.
On the right, the image of this chain in $E^{2}_{-1,1}\big[SC_{2}(R_{2,2})\big]=H_{1}\big(SC_{2}(R_{2,2})\big)$ under the $d^{1}$-differential.
It follows that this orbiting pair square class is the image of the class at the left of Figure~\ref{Orbitingpairsquaredifferential1} under the $d^{2}$-differential.}
\label{Orbitingpairsquaredifferential2}
\end{figure}

The analogous argument also works in the $SC_{3}(R_{2,2})$ setting with an extra square labeled $1$ on the left side of $R_{2,2}$.
\end{proof}

As long as we can arrange the squares inside $R_{w,h}$ so that there are two unoccupied squares immediately to the left of a pair of squares labeled $2$, Proposition~\ref{spectral sequence d2 differential } will allow us to determine the image of the $d^{2}$-differential.
In the following subsections we use this fact to compute $H_{1}\big(SC_{n}(R_{w,2})\big)$.

\subsection{The three step ladder}\label{three step ladder subsection}

Every case of Theorem~\ref{first homology of height 2 box} beyond the $d=1$ case is proved by climbing the same three rungs of our ladder, each a spectral sequence from Section~\ref{spectral sequence section} that allows us to move from $H_{1}\big(SC_{n-2}(R_{w-2,2})\big)$ and $H_{1}\big(SC_{n-2}(R_{w-1,2})\big)$ to $H_{1}\big(SC_{n}(R_{w,2})\big)$:

\[\begin{tikzcd}
	H_{1}\big(SC_{n-2}(R_{w-2,2})\big) &&& \\
	& H_{1}\big(SC_{n}(R_{w,2};1;1)\big) & H_{1}\big(SC_{n}(R_{w,2};1)\big) & H_{1}\big(SC_{n}(R_{w,2})\big) \\
	H_{1}\big(SC_{n-2}(R_{w-1,2})\big).
	\arrow[squiggly, from=1-1, to=2-2]
	\arrow[squiggly, from=2-2, to=2-3]
	\arrow[squiggly, from=2-3, to=2-4]
	\arrow[squiggly, from=3-1, to=2-2]
\end{tikzcd}\]
The first rung is the pair of arrows to $H_{1}\big(SC(R_{w,2};1;1)\big)$ and is an instance of the augmented slide spectral sequence of Subsection~\ref{slide spectral sequence subsection}.
The second and third rungs are the arrows to $H_{1}\big(SC_{n}(R_{w,2};1)\big)$ and $H_{1}\big(SC_{n}(R_{w,2})\big)$, and are instances of the augmented right Mayer--Vietoris spectral sequence of Subsection~\ref{right Mayer--Vietoris spectral sequence subsection}.
The point of this subsection is that all three rungs of the ladder can be described in generality.

First, we record three facts on which the rungs depend.
Each is stated for all $w$ and $n$, and each is used in every case.
The first calculates the $E^{1}_{0,0}$-entry of the augmented right Mayer--Vietoris spectral sequence of $SC_{n}(R_{w,2};1)$.

\begin{prop}\label{SCnRwh two components}
If $w\ge 2$ and $2w-1\le n\le 2w$, then $SC_{n}(R_{w,2};1;1)$ consists of two components; if $2\le n\le 2w-2$, it consists of a single component.
\end{prop}

\begin{proof}
If $w\ge 2$ and $2w-1\le n\le 2w$, the components correspond to the square labeled $2$ being above or below the square labeled $3$; see the top of Figure~\ref{ComponentsofSCnRw211}.
If $2\le n\le 2w-2$, then we can slide the squares labeled $2$ and $3$ into the two rightmost columns of $R_{w,2}$ and the squares labeled $1$ out of these two columns; see the bottom of Figure~\ref{ComponentsofSCnRw211}.
This space is connected.
\begin{figure}[h]
\centering
\captionsetup{width=.8\linewidth}
\includegraphics[width=6.25cm]{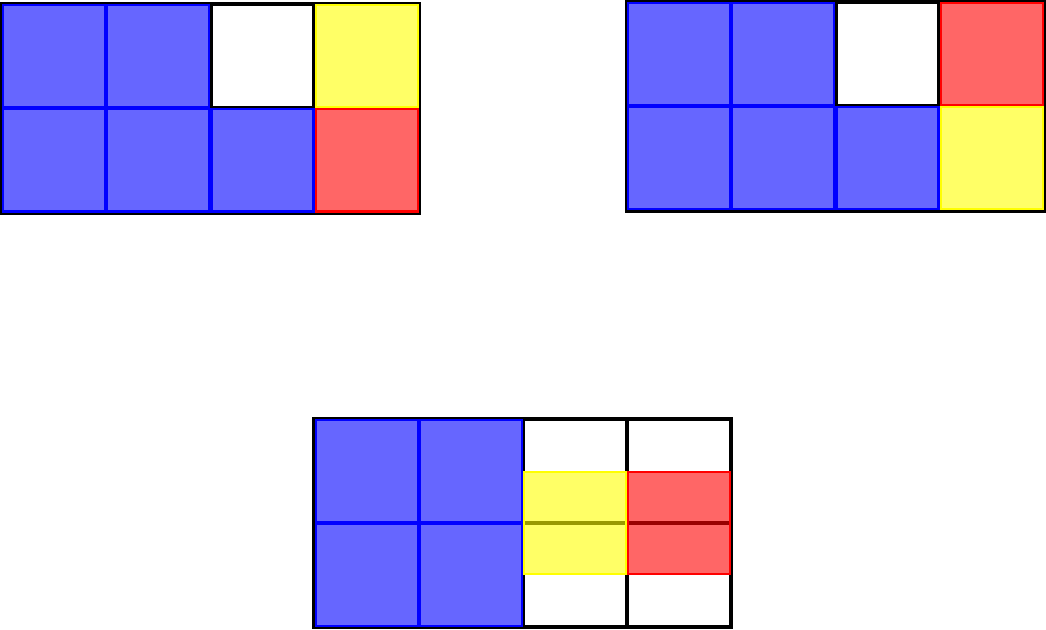}
\caption{At the top, the two components of $SC_{7}(R_{4,2};1;1)$.
Note that in the left configuration there is no way of moving the yellow square (the square labeled $2$) below the red square (the square labeled $3$), while not letting any of the blue squares (the squares labeled $1$) pass to the right of it, so as to move to the configuration on the right.
At the bottom, a configuration in the sole component of $SC_{6}(R_{4,2};1;1)$.
In this component the yellow square can be slid above or below the red square.}
\label{ComponentsofSCnRw211}
\end{figure}
\end{proof}

The second proposition records the $\tilde{E}^{1}_{1,0}$-entry of the augmented slide spectral sequence of $SC_{n}(R_{w,2};1;1)$, and the third proves that on the second page of this spectral sequence the zeroth row vanishes.
This shows that in the augmented slide spectral sequence $H_{1}\big(SC_{n}(R_{w,2};1;1)\big)$ is entirely determined by the $\tilde{E}^{1}_{0,1}$-entry and the $\tilde{d}^{1}$-differential, a key fact for our first rung.

\begin{prop}\label{intersection of V1 and V2 h=2}
If $w\ge2$ and $V_{1}$ and $V_{2}$ are the sets used to define the augmented slide spectral sequence for $SC_{n}(R_{w,2}; 1;1)$, then $V_{1}\cap V_{2}\simeq SC_{n-2}(R_{w-2, 2})$.
\end{prop}

\begin{proof}
In $V_{1}$ and $V_{2}$, after sliding the square labeled $3$ to the far right of $R_{w,2}$, we may slide the square labeled $2$ to the right so that it is in one of the $2$ rightmost columns of $R_{w,2}$.
It follows from the definition of $V_{1}$ and $V_{2}$ that in $V_{1}\cap V_{2}$ the square labeled $2$ must be in the second rightmost column moving between the top and bottom grid-squares; see Figure~\ref{V1intersectV2}.
Thus, no square labeled $1$ can be in the two rightmost columns of $R_{w,2}$.
Therefore, if we remove the two rightmost columns of $R_{w,2}$ and the squares labeled $2$ and $3$, we get $V_{1}\cap V_{2}\simeq SC_{n-2}(R_{w-2, 2})$.
\begin{figure}[h]
\centering
\captionsetup{width=.8\linewidth}
\includegraphics[width=2.5cm]{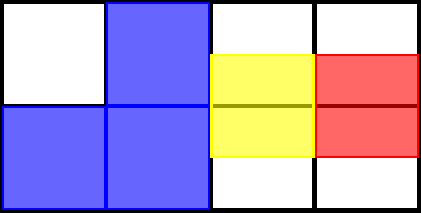}
\caption{A configuration in $V_{1}\cap V_{2}$, where $V_{1}$ and $V_{2}$ are the sets used to define the slide spectral sequence for $SC_{5}(R_{4,2};1;1)$.
Note that the square labeled $2$, which we color yellow, must be in between the top and the bottom grid squares in the second rightmost column.
This follows from the fact that after sliding that square up or down it must be able to be slid into the rightmost column of $R_{w,2}$ so as to be either above or below the square labeled $3$, which we have colored red.}
\label{V1intersectV2}
\end{figure}
\end{proof}

Proposition~\ref{connected unordered square configuration spaces one rightmost} tells us that $SC_{n-2}(R_{w-2,2})$ is connected for all $n\le 2w-2$.
We use this to prove that the zeroth row of the $\tilde{E}^{2}$-page of the slide spectral sequence is $0$.
Since $\tilde{E}^{1}_{-1,1}\big[SC_{n}(R_{w,2};1;1)\big]= H_{1}\big(SC_{n}(R_{w,2};1;1)\big)$, this shows that $H_{1}\big(SC_{n}(R_{w,2};1;1)\big)$ is determined by the image of $H_{1}(V_{1})\oplus H_{1}(V_{2})$.

\begin{prop}\label{ E2 entry of slide spectral sequence h=2}
If $w\ge 2$, then $\tilde{E}^{2}_{p,0}\big[SC_{n}(R_{w,2}; 1;1)\big]=0$ for all $p$.
\end{prop}

\begin{proof}
Consider the zeroth row of the $\tilde{E}^{1}$-page of the augmented slide spectral sequence for $SC_{n}(R_{w,2}; 1;1)$.
It has three potentially non-trivial entries:
\begin{align*}
&\tilde{E}^{1}_{-1,0}\big[SC_{n}(R_{w,2}; 1;1)\big]=H_{0}\big(SC_{n}(R_{w,2}; 1;1)\big),\\
&\tilde{E}^{1}_{0,0}\big[SC_{n}(R_{w,2}; 1;1)\big]\;\;=H_{0}(V_{1})\oplus H_{0}(V_{2}),\text{ and}\\
&\tilde{E}^{1}_{1,0}\big[SC_{n}(R_{w,2}; 1;1)\big]=H_{0}(V_{1}\cap V_{2}).
\end{align*}

If $n<2$ or $n>2w$, then $SC_{n}(R_{w,2};1;1)$ is empty; if $2\le n\le 2w-2$, then it is connected; and if $2w-2<n\le 2w$, then by Proposition~\ref{SCnRwh two components} it consists of two components.
By Proposition~\ref{Vl is SCp2}, we have that $V_{1}\simeq V_{2}\simeq SC_{n-2}(R_{w-1,2})$, which is connected for all $2\le n< 2w$ by Proposition~\ref{connected unordered square configuration spaces one rightmost}; for $n=2w$ this space is a point.
Otherwise, it is empty.
Finally, by Proposition~\ref{intersection of V1 and V2 h=2}, we have that $V_{1}\cap V_{2}\simeq SC_{n-2}(R_{w-2,2})$.
Proposition~\ref{connected unordered square configuration spaces one rightmost} implies this is connected if $2\le n\le 2w-2$; otherwise, it is empty.
It follows that the zeroth row of the first page of the augmented slide spectral sequence of $SC_{n}(R_{w,2};1;1)$ is of one of the three forms depicted in Figure~\ref{slide spectral sequence for SC h=2 n s=1 p=1 simplfied}.

\begin{figure}[h!] 
\centering
\begin{tikzpicture}
\footnotesize

\matrix (m) [
  matrix of math nodes,
  nodes in empty cells,
  nodes={
    minimum width=3em,
    minimum height=3em,
    anchor=center
  },
  row sep=0ex,
  column sep=1ex
]{
t=0 &
0 &
0 &
0 &
0 &
\cdots
\\
&
s=-1 &
s=0 &
s=1 &
s=2 &
\cdots
\\
};

\draw[thick]
  ($(m-1-1.east)+(0em,3ex)$)
  --
  ($(m-2-1.east)+(0em,0ex)$);

\draw[thick]
  ($(m-1-1.south)+(-2em,-0.5ex)$)
  --
  ($(m-1-6.south)+(1em,-0.5ex)$);
\end{tikzpicture}\\
\begin{tikzpicture}
\footnotesize

\matrix (m) [
  matrix of math nodes,
  nodes in empty cells,
  nodes={
    minimum width=3em,
    minimum height=3em,
    anchor=center
  },
  row sep=0ex,
  column sep=1ex
]{
t=0 &
\Z &
\Z^{2} &
\Z &
0 &
\cdots
\\
&
s=-1 &
s=0 &
s=1 &
s=2 &
\cdots
\\
};

\draw[thick]
  ($(m-1-1.east)+(0em,3ex)$)
  --
  ($(m-2-1.east)+(0em,0ex)$);

\draw[thick]
  ($(m-1-1.south)+(-2em,-0.5ex)$)
  --
  ($(m-1-6.south)+(1em,-0.5ex)$);
\end{tikzpicture}\\

\begin{tikzpicture}
\footnotesize

\matrix (m) [
  matrix of math nodes,
  nodes in empty cells,
  nodes={
    minimum width=3em,
    minimum height=3em,
    anchor=center
  },
  row sep=0ex,
  column sep=1ex
]{
t=0 &
\Z^{2} &
\Z^{2} &
0 &
0 &
\cdots
\\
&
s=-1 &
s=0 &
s=1 &
s=2 &
\cdots
\\
};

\draw[thick]
  ($(m-1-1.east)+(0em,3ex)$)
  --
  ($(m-2-1.east)+(0em,0ex)$);

\draw[thick]
  ($(m-1-1.south)+(-2em,-0.5ex)$)
  --
  ($(m-1-6.south)+(1em,-0.5ex)$);
\end{tikzpicture}
\caption{The zeroth row of the $\tilde{E}^{1}$-page of the augmented slide spectral sequence for $SC_{n}(R_{w,2};1;1)$ for various $n$.
At the top, $n<2$ or $n>2w$; in the middle, $2\le n\le 2w-2$; and at the bottom, $2w-2<n\le 2w$.} 
\label{slide spectral sequence for SC h=2 n s=1 p=1 simplfied}
\end{figure}
Since the augmented slide spectral sequence is an augmented Mayer--Vietoris spectral sequence, it has trivial abutment.
Since the $\tilde{E}^{1}$-page of this spectral sequence only has non-trivial entries in the negative first, zeroth, and first columns, the $\tilde{E}^{2}_{-1,0}$- and $\tilde{E}^{2}_{0,0}$-entries are $0$.
In all three cases, this is only possible if the $\tilde{E}^{2}_{1,0}$-entry is trivial, implying that the entire zeroth row of the second page of the augmented slide spectral sequence is trivial.
\end{proof}

With these facts in mind, we describe the three rungs.
Note that none of these propositions mentions $d$, as we restrict our analysis of the cokernels appearing in each rung to the subsequent subsections, though in Propositions~\ref{rung B prop} and \ref{rung C prop} it is implicit that $d\ge 2$---we handle the $d=1$ case in Subsection~\ref{n=2w-1 subsection} by other means.

\begin{prop}[Rung A: from $H_{1}\big(SC_{n-2}(R_{w-2,2})\big)$ to $H_{1}\big(SC_{n}(R_{w,2};1;1)\big)$]\label{rung A prop}
Let $w \ge 2$ and $2 \le n \le 2w$. Then
\[
H_{1}\big(SC_{n}(R_{w,2};1;1)\big)
\cong \operatorname{coker}(\alpha),
\]
where $\alpha$ is the $\tilde{d}^{1}$-differential
\[
\tilde{d}^{1}:\tilde{E}^{1}_{1,1}\big[SC_{n}(R_{w,2};1;1)\big]=H_{1}(V_{1} \cap V_{2}) \to H_{1}(V_{1}) \oplus H_{1}(V_{2})=\tilde{E}^{1}_{0,1}\big[SC_{n}(R_{w,2};1;1)\big]
\]
of the augmented slide spectral sequence of $SC_{n}(R_{w,2};1;1)$.
\end{prop}
 
\begin{proof}
The augmented slide spectral sequence is an augmented Mayer--Vietoris spectral sequence, so its abutment is trivial.
Since only the negative first, zeroth, and first columns of the $\tilde{E}^{1}$-page are non-trivial, the map 
\[
\tilde{d}^{1} : \tilde{E}^{1}_{0,1}\big[SC_{n}(R_{w,2};1;1)\big]=H_{1}(V_{1}) \oplus H_{1}(V_2) \to H_{1}\big(SC_{n}(R_{w,2};1;1)\big)=\tilde{E}^{1}_{-1,1}\big[SC_{n}(R_{w,2};1;1)\big]
\]
has kernel equal to the image of
$\alpha$ as the $\tilde{E}^{2}_{0,1}\big[SC_{n}(R_{w,2};1;1)\big]$-entry must be $0$.
Moreover, this differential is surjective since the $\tilde{E}^{2}_{1,0}\big[SC_{n}(R_{w,2};1;1)\big]$-entry vanishes by Proposition~\ref{ E2 entry of slide spectral sequence h=2}.
It follows that $H_{1}\big(SC_{n}(R_{w,2};1;1)\big)\cong \operatorname{coker}(\alpha)$.
\end{proof}

\begin{prop}[Rung B: from $H_{1}\big(SC_{n}(R_{w,2};1;1)\big)$ to $H_{1}\big(SC_{n}(R_{w,2};1)\big)$] \label{rung B prop}
Let $w \ge 2$ and $2 \le n \le 2w-2$. Then
\[
H_{1}\big(SC_{n}(R_{w,2};1)\big)\cong \operatorname{coker}(\beta) \oplus \Z,
\]
where $\beta$ is the $d^{1}$-differential
\[
d^{1}:E^{1}_{1,1}\big[SC_{n}(R_{w,2};1)\big]=H_{1}\big(SC_{n}(R_{w,2};2;1)\big) \to H_{1}\big(SC_{n}(R_{w,2};1;1)\big)=E^{1}_{0,1}\big[SC_{n}(R_{w,2};1)\big]
\]
of the augmented right Mayer--Vietoris spectral sequence of $SC_{n}(R_{w,2};1)$, and the free summand $\mathbb{Z}$ is the image of the $d^{2}$-differential in $E^{2}_{-1,1}\big[SC_{n}(R_{w,2};1)\big]$.
\end{prop}
 
\begin{proof}
Recall $E^{1}_{s,t}\big[SC_{n}(R_{w,2};1)\big] = H_{t}\big(SC_{n}(R_{w,2};s+1;1)\big)$. 
Every nontrivial entry of the zeroth row is $\Z$.
For $s = -1$ this is Proposition~\ref{connected unordered square configuration spaces one rightmost}, for $s = 0$ it is Proposition~\ref{SCnRwh two components}, and for $s = 1$ it is
Proposition~\ref{connected unordered square configuration spaces one rightmost} together with the equivalence
$SC_{n}(R_{w,2};2;1) \simeq SC_{n-3}(R_{w-2,2})$ obtained by sliding the squares labeled $2$ and $3$ to the far right of $R_{w,2}$ and noting that no square labeled $1$ can enter the two rightmost columns since the squares labeled $2$ occupy the entire height of the second rightmost column.
That same equivalence identifies $E^{1}_{1,1}\big[SC_{n}(R_{w,2};1)\big]$ with $H_{1}\big(SC_{n-3}(R_{w-2,2})\big)$, the source of $\beta$.
 
By Proposition~\ref{augmented right MVSS has trivial abutment} the abutment of this spectral sequence is trivial.
Since only the negative first, zeroth, and first columns of the $E^{1}$-page of this spectral sequence can be nonzero, the image of
\[
d^{1}:E^{1}_{0,1}\big[SC_{n}(R_{w,2};1)\big]=H_{1}\big(SC_{n}(R_{w,2};1;1)\big) \to H_{1}\big(SC_{n}(R_{w,2};1)\big)=E^{1}_{-1,1}\big[SC_{n}(R_{w,2};1)\big]
\]
is isomorphic to $\operatorname{coker}(\beta)$. 

By the definition of the
$d^{1}$-differentials on the zeroth row, $E^2_{1,0}\big[SC_{n}(R_{w,2};1)\big] \cong \Z$, and
triviality of the abutment forces
\[
d^2 : E^2_{1,0}\big[SC_{n}(R_{w,2};1)\big] \to E^2_{-1,1}\big[SC_{n}(R_{w,2};1)\big] \cong H_{1}\big(SC_{n}(R_{w,2};1)\big)/\operatorname{coker}(\beta)
\]
to be an isomorphism.
Since the quotient is a free module, we get the desired splitting.
\end{proof}

\begin{prop}[Rung C: from $H_{1}\big(SC_{n}(R_{w,2};1)\big)$ to $H_{1}\big(SC_{n}(R_{w,2})\big)$]
\label{rung C prop}
Let $w \ge 2$ and $2 \le n \le 2w-2$. Then
\[
H_{1}\big(SC_{n}(R_{w,2})\big)\cong \operatorname{coker}(\gamma) \oplus \Z,
\]
where $\gamma$ is the $d^1$-differential 
\[
d^{1}:E^{1}_{1,1}\big[SC_{n}(R_{w,2})\big]=H_{1}\big(SC_{n}(R_{w,2};2)\big) \to H_{1}\big(SC_{n}(R_{w,2};1)\big)=E^{1}_{0,1}\big[SC_{n}(R_{w,2})\big]
\]
of the augmented right Mayer--Vietoris spectral sequence of $SC_{n}(R_{w,2})$, and the free summand $\Z$ is the image of the $d^{2}$-differential in $E^{2}_{-1,1}\big[SC_{n}(R_{w,2})\big]$.
\end{prop}
 
\begin{proof}
The proof is nearly identical to the proof of Proposition~\ref{rung B prop}, after observing that $E^{1}_{p,q}\big[SC_{n}(R_{w,2})\big]\cong H_{q}\big(SC_{n}(R_{w,2};p+1)\big)$ and noting that for $q=0$, these groups are isomorphic to $\Z$ by Proposition~\ref{connected unordered square configuration spaces one rightmost}.
\end{proof}

\begin{remark} \label{ladder remark}
Combining the three rungs gives a description of $H_1\big(SC_n(R_{w,2})\big)$ in terms of the homology of $SC_{n-2}(R_{w-1,2})$, $SC_{n-2}(R_{w-2,2})$, and $SC_{n-3}(R_{w-2,2})$. 
The first has the same value of $d$ as $SC_n(R_{w,2})$ and drives the induction on $w$. 
The second has $d$ two smaller, and the third has $d$ one smaller, driving the induction on $d$. 
In each of the cases below, the only work is to identify $\alpha$, $\beta$, and $\gamma$.
This comes down to determining when there is enough unoccupied space in $R_{w,2}$ to identify two classes by sliding squares around.
\end{remark}

Having described the three rungs of our ladder, we turn to handling the $d=1$, $d=2$, $d=3$, and $d\ge 4$ casework.

\subsection{$d=1$: the case $n=2w-1$}\label{n=2w-1 subsection}

In this subsection, we prove that $H_{1}\big(SC_{2w-1}(R_{w,2})\big)$ is spanned by orbiting triple square classes.
This is the only case in which we do not use our ladder; instead we appeal to \cite{MattOmarJesus}.

\begin{prop}\label{first homology height 2 for 2w-1}
If $w\ge h=2$, then 
\[
H_{1}\big(SC_{2w-1}(R_{w,2})\big)\cong \Z^{w-1}.
\]
Moreover, $H_{1}\big(SC_{2w-1}(R_{w,2})\big)$ is generated by the $w-1$ orbiting triple square classes corresponding to the $w-1$ grid-preserving  embeddings of $R_{2,2}$ in $R_{w,2}$.
\end{prop}

\begin{proof}
Alvarado-Gardu\~{n}o, Gonz\'{a}lez, and Kahle proved that $SC_{2w-1}(R_{w,2})\simeq DC_{2w-1}(R_{w-1,1})$ is homotopy equivalent to the ladder graph $\Gamma_{w}$ on $2w$ vertices, which is homotopy equivalent to a wedge of $w-1$ circles \cite[Theorem 1.2]{MattOmarJesus}; see Figure~\ref{Gamma42}.
It follows that $H_{1}\big(SC_{2w-1}(R_{w,2})\big)\cong \Z^{w-1}$.
Moreover, the $w-1$ square loops in $\Gamma_{w}$ are the $w-1$ orbiting triple square classes coming from the $w-1$ grid-preserving embeddings of $R_{2,2}$ in $R_{w,2}$.
\begin{figure}[h]
\centering
\captionsetup{width=.8\linewidth}
\includegraphics[width=12cm]{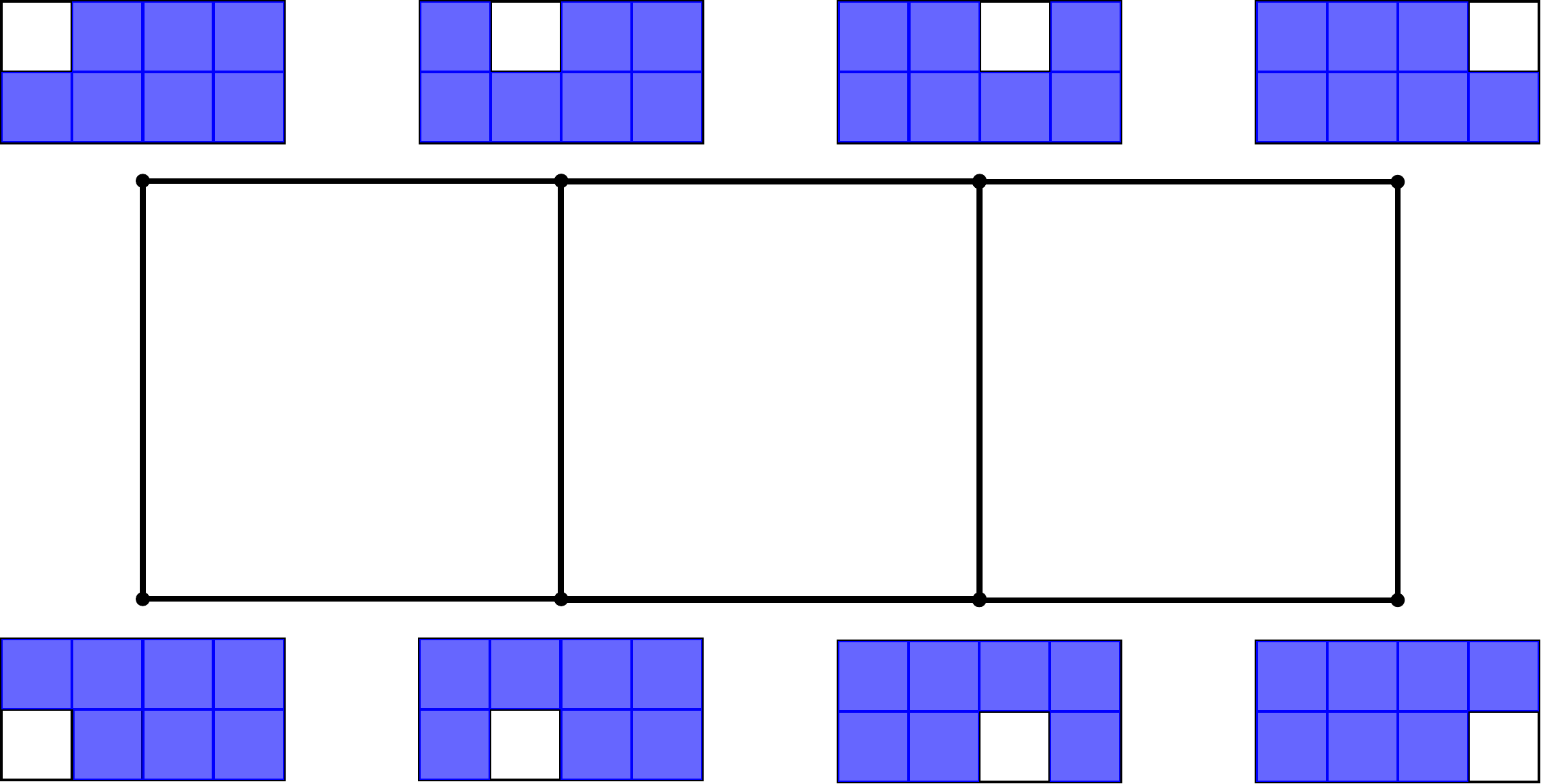}
\caption{The graph $\Gamma_{4}$, where the vertices are labeled by the corresponding configurations in $SC_{7}(R_{4,2})$.}
\label{Gamma42}
\end{figure}
\end{proof}

Having handled the $n=2w-1$ case of Theorem~\ref{first homology of height 2 box}, we turn to the $n\le2w-2$ cases, each of which relies on the ladder described in Subsection \ref{three step ladder subsection}.

\subsection{$d=2$: the case $n=2w-2$}\label{n=2w-2 subsection}

In this subsection, we compute $H_{1}\big(SC_{2w-2}(R_{w,2})\big)$.
Namely, we prove

\begin{lem}\label{first homology height 2 for 2w-2}
If $w\ge h=2$, then 
\[
H_{1}\big(SC_{2w-2}(R_{w,2})\big)\cong \Z^{2w-3}.
\]
Moreover, $H_{1}\big(SC_{2w-2}(R_{w,2})\big)$ is generated by the $w-1$ even orbiting pair square classes and the $w-2$ even orbiting triple square classes.
\end{lem}

\begin{proof}
We proceed by induction on $w$.
In the case $w=2$, work of Alpert, Bauer, Kahle, MacPherson, and Spendlove \cite[Theorem 3.5]{alpert2023homology} proves that $SC_{2}(R_{2,2})$ is homotopy equivalent to the figure at the top right of Figure~\ref{Configsin2by2}.
It follows that $H_{1}\big(SC_{2}(R_{2,2})\big)\cong \Z$, and that $H_{1}\big(SC_{2}(R_{2,2})\big)$ is generated by a single even orbiting pair square class.

For the induction step, assume the lemma holds for all $w\le \omega$ for some $\omega\ge 2$.
We will show that the lemma holds for $w=\omega+1$ by using the three steps of our ladder, first using the inductive hypothesis to compute $H_{1}\big(SC_{2\omega}(R_{\omega+1, 2};1;1)\big)$, next using this to compute $H_{1}\big(SC_{2\omega}(R_{\omega+1, 2};1)\big)$, and finally using that to compute $H_{1}\big(SC_{2\omega}(R_{\omega+1, 2})\big)$.

\emph{Rung A:} By Proposition~\ref{intersection of V1 and V2 h=2} we have that $V_{1}\cap V_{2}\simeq SC_{2\omega-2}(R_{\omega-1,2})\simeq pt$, so $\alpha$ is the trivial map.
Since $V_{1}\simeq V_{2}\simeq SC_{2\omega}(R_{\omega+1,2};2)\simeq SC_{2\omega-2}(R_{\omega,2})$ by Proposition~\ref{Vl is SCp2}, Proposition~\ref{rung A prop} implies that
\[
H_{1}\big(SC_{2\omega}(R_{\omega+1, 2};1;1)\big)\cong H_{1}(V_{1})\oplus H_{1}(V_{2})\cong H_{1}\big(SC_{2\omega-2}(R_{\omega, 2})\big)^{\oplus 2}.
\]
Moreover, the inductive hypothesis implies that  $H_{1}\big(SC_{2\omega}(R_{\omega+1, 2};1;1)\big)$ is generated by the $\omega-1$ pairs of even orbiting pair square classes and the $\omega-2$ pairs of even orbiting triple square classes. 
Here each pair consists of the identical classes with the locations of the squares labeled $2$ and $3$ swapped; see Figure~\ref{GeneratorsH1SC6R4211}.
\begin{figure}[h]
\centering
\captionsetup{width=.8\linewidth}
\includegraphics[width=10cm]{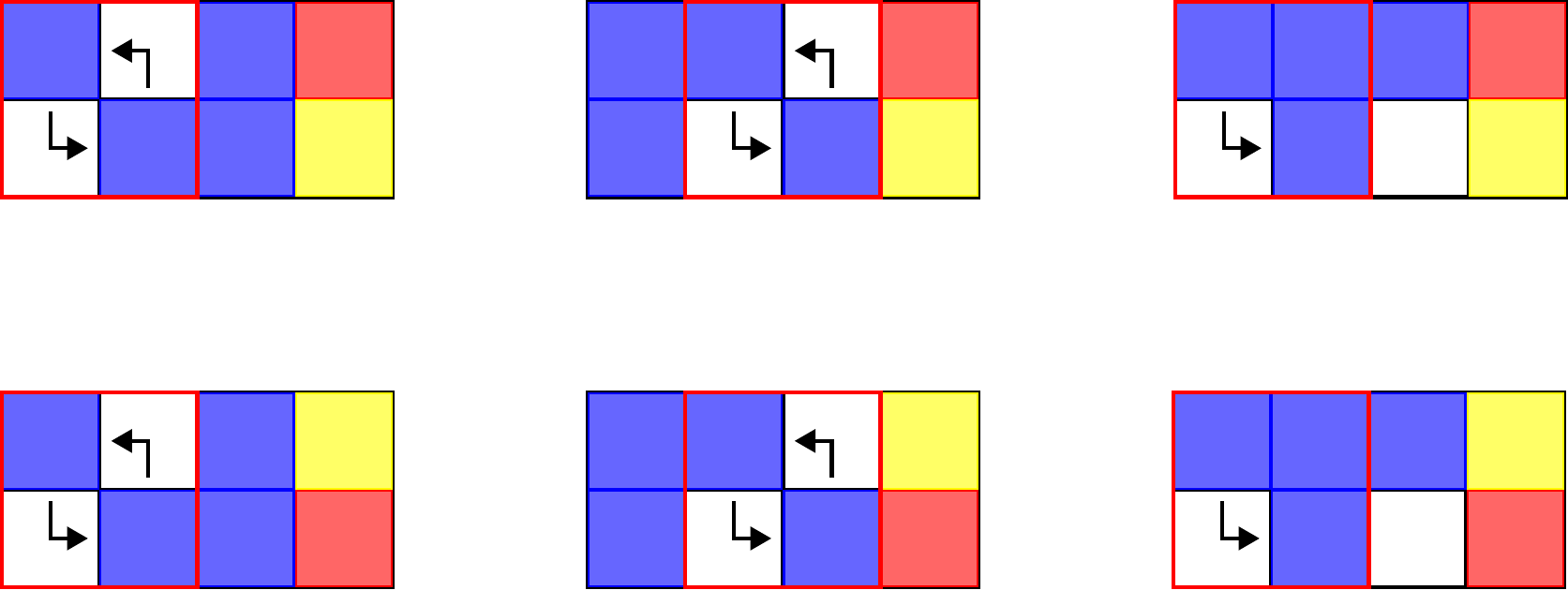}
\caption{The 6 generators of $H_{1}\big(SC_{6}(R_{4,2};1;1)\big)$.
Note that they come in pairs, where one element of each pair has the yellow square (the square labeled $2$) above the red square (the square labeled $3$), and the other below.
There are two locations for orbiting pair square class, but only one for the orbiting triple square class---in our bases we will never have such a class be as right as possible.}
\label{GeneratorsH1SC6R4211}
\end{figure}

\emph{Rung B:} 
We claim that $\beta$ is an injection. 
Observe that $SC_{2\omega}(R_{\omega+1,2};2;1)\simeq SC_{2\omega-3}(R_{\omega-1,2})$.
This combined with Proposition~\ref{first homology height 2 for 2w-1} tells us that  
\[
H_{1}\big(SC_{2\omega}(R_{\omega+1,2};2;1)\big)\cong H_{1}\big(SC_{2\omega-3}(R_{\omega-1,2})\big)\cong\Z^{\omega-2},
\]
and this group is generated by $\omega-2$ orbiting triple square classes, one for each grid-preserving way of embedding $R_{2,2}$ into the $R_{\omega-1, 2}$ found on the left side of $R_{\omega+1,2}$.
The map $\beta$ sends orbiting triple square classes to the difference of a pair of orbiting triple square classes, which identifies the two classes in a pair of generators of $H_{1}\big(SC_{2\omega}(R_{\omega+1, 2};1;1)\big)$ described in Rung A; see Figure~\ref{ClassinH1SC6R4221anditsimage}.
It follows from Proposition~\ref{rung B prop} that 
\[
H_{1}\big(SC_{2\omega}(R_{\omega+1,2};1)\big)\cong \Z^{3\omega-4} \oplus \Z,
\]
where the $\Z^{3\omega-4}$ summand is generated by the $\omega-1$ pairs of even orbiting pair square classes and the $\omega-2$ even orbiting triple square classes. 
The $\Z$ summand is the image of the $d^{2}$-differential, which is an even orbiting triple class with only the square labeled $2$ to its right by Proposition~\ref{spectral sequence d2 differential } as any $R_{2,2}$ embedded in $R_{\omega+1,2}$ in a grid respecting way contains at least three unit-squares.
\begin{figure}[h]
\centering
\captionsetup{width=.8\linewidth}
\includegraphics[width=6.25cm]{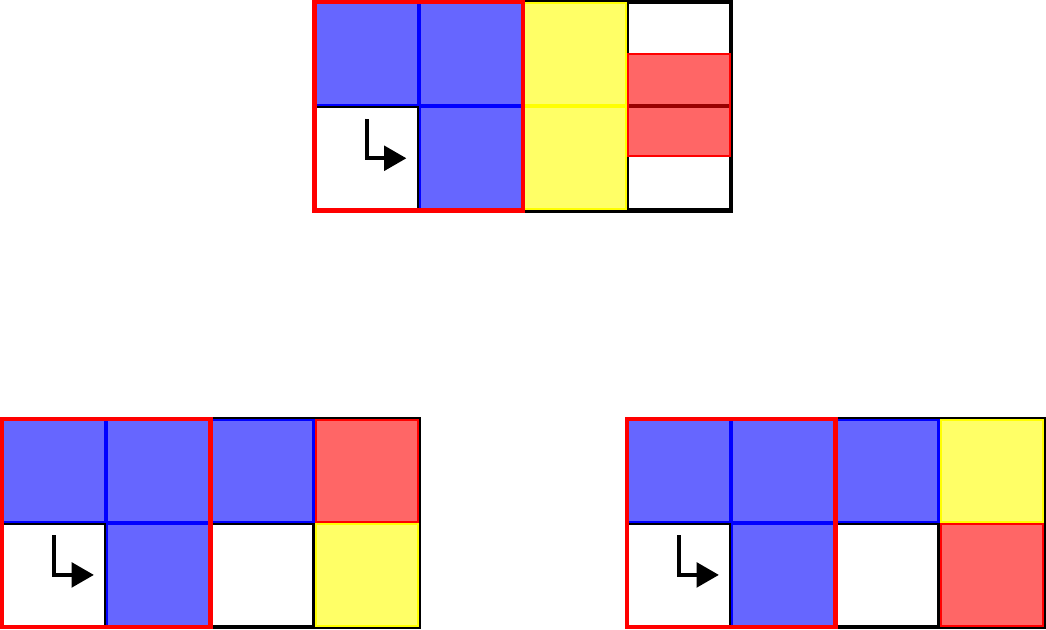}
\caption{At the top a generator of $H_{1}\big(SC_{6}(R_{4,2};2;1)\big)$.
At the bottom, the two classes in $H_{1}\big(SC_{6}(R_{4,2};1;1)\big)$, whose difference is the image of the above class under the $d^{1}$-differential.
See Figure~\ref{GeneratorsH1SC6R4211}.}
\label{ClassinH1SC6R4221anditsimage}
\end{figure}

\emph{Rung C:} We claim that $\operatorname{coker}(\gamma)$ is a direct summand of $H_{1}\big(SC_{2\omega}(R_{\omega+1, 2};1)\big)$, and that it is isomorphic to $\Z^{2\omega-2}$.
By the inductive hypothesis, 
\[
E^{1}_{1,1}\big[SC_{2\omega}(R_{\omega+1, 2})\big]\cong H_{1}\big(SC_{2\omega}(R_{\omega+1, 2};2)\big)\cong H_{1}\big(SC_{2\omega-2}(R_{\omega, 2})\big)\cong \Z^{2\omega-3}.
\]
This follows from the observation that $SC_{2\omega}(R_{\omega+1, 2};2)\simeq SC_{2\omega-2}(R_{\omega, 2})$ as the two squares labeled $2$ can be pushed to the far right of $R_{\omega+1, 2}$ and there is not enough space for any other square to enter the rightmost column.
Moreover, the inductive hypothesis tells us that this entry is generated by the $\omega-1$ even orbiting pair square classes and the $\omega-2$ even orbiting triple square classes. 
The map $\gamma$ sends the $\omega-2$ orbiting triple square classes to $0$ and an orbiting pair square class to the difference of a pair of the orbiting pair square class generators of $H_{1}\big(SC_{2\omega}(R_{\omega+1, 2});1\big)$ described in Rung B; see Figure~\ref{ClassesinH1SC6R422anditsimage}.
Therefore, its cokernel is isomorphic to $\Z^{2\omega-2}$. 
It follows from Proposition~\ref{rung C prop} that 
\[
H_{1}\big(SC_{2\omega}(R_{\omega+1,2})\big)\cong \Z^{2\omega-2} \oplus \Z,
\]
where the $\Z^{2\omega-2}$ summand is generated by the $\omega-1$ even orbiting pair square classes and the $\omega-1$ even orbiting triple square classes, and the $\Z$ summand is the image of the $d^{2}$-differential.
By Proposition~\ref{spectral sequence d2 differential }, this image corresponds to an even orbiting pair square class at the far right of $R_{\omega+1,2}$ as we can move the squares labeled $1$ about $R_{\omega+1,2}$ so that the grid squares immediately to the left of the squares labeled $2$ are unoccupied.
This completes the inductive step and hence the proof.
\begin{figure}[h]
\centering
\captionsetup{width=.8\linewidth}
\includegraphics[width=15cm]{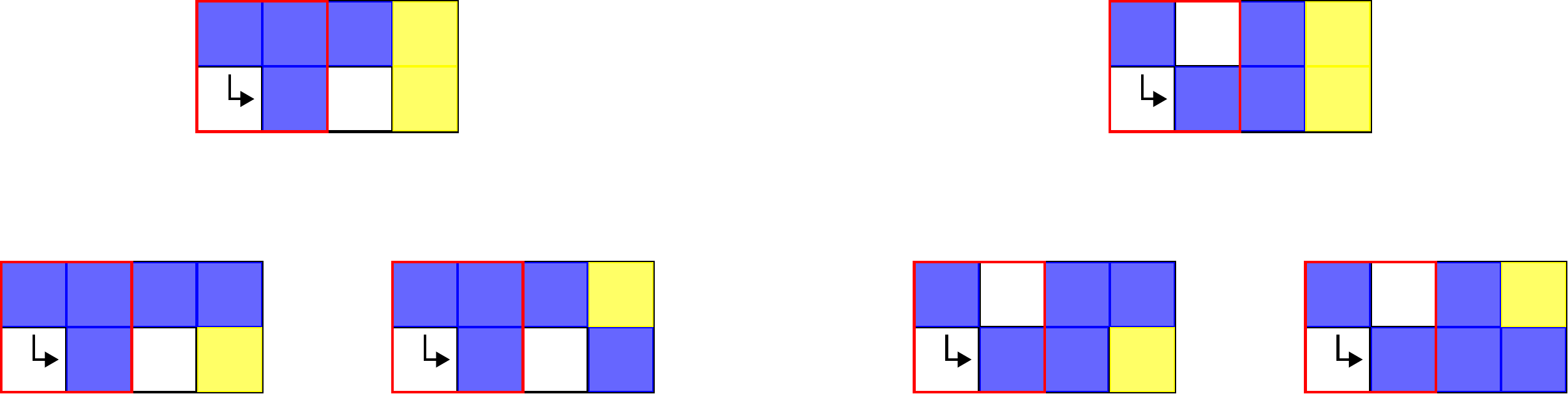}
\caption{On the left, the orbiting triple square class in $H_{1}\big(SC_{6}(R_{4,2};2)\big)$ at the top, and the two classes in $H_{1}\big(SC_{6}(R_{4,2};1)\big)$ at the bottom whose difference is its image under the $d^{1}$-differential. Note these classes are homologous, so the $d^{1}$-differential sends this class to $0$.
On the right, an orbiting pair square class in $H_{1}\big(SC_{6}(R_{4,2};2)\big)$ at the top, and the two classes in $H_{1}\big(SC_{6}(R_{4,2};1)\big)$ at the bottom whose difference is its image under the $d^{1}$-differential.
Note these two classes are distinct generators of $H_{1}\big(SC_{6}(R_{4,2};1)\big)$, see Figure~\ref{GeneratorsH1SC6R421}, so the $d^{1}$-differential identifies them.}
\label{ClassesinH1SC6R422anditsimage}
\end{figure}
\begin{figure}[h]
\centering
\captionsetup{width=.8\linewidth}
\includegraphics[width=10cm]{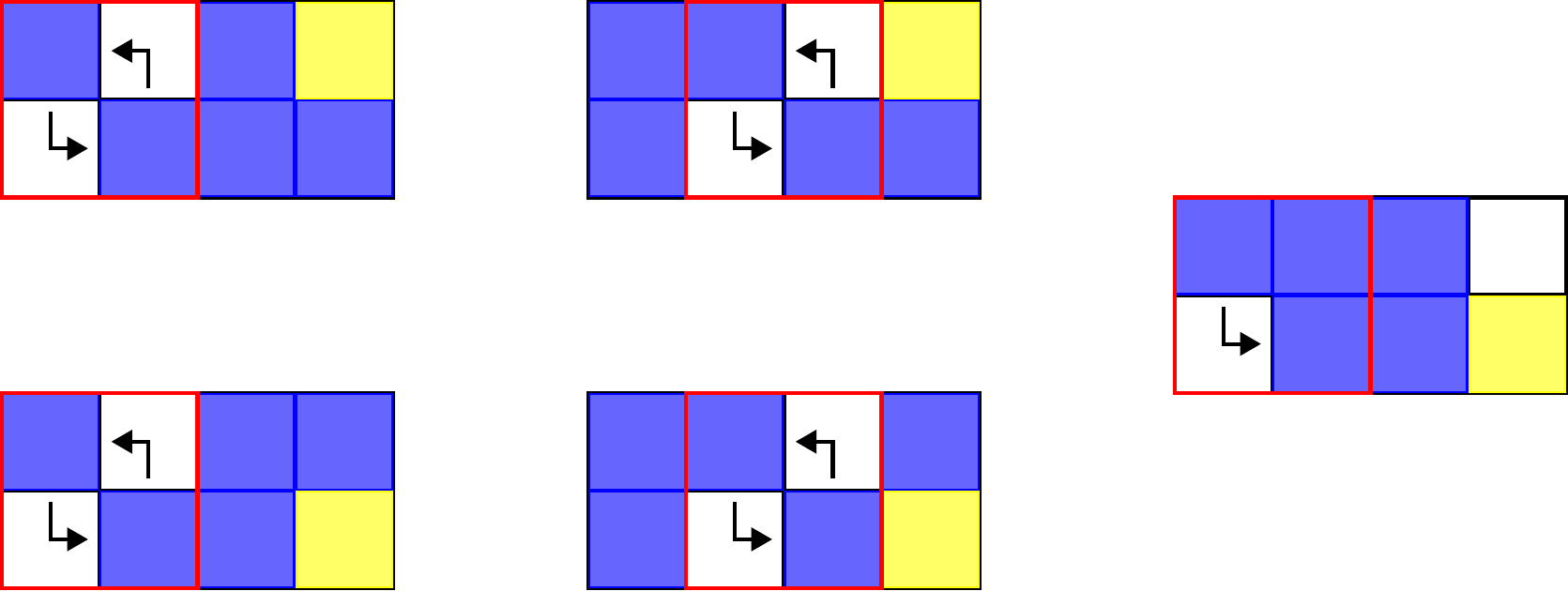}
\caption{The $5$ generators of $H_{1}\big(SC_{6}(R_{4,2};1)\big)$ coming from the $d^{1}$-differential.
There are $2$ pairs of orbiting pair square classes and $1$ orbiting triple square class.
Here we have colored the square labeled $2$ red as we are implicitly using the identification $SC_{n}(R_{w,h};1)\simeq SC_{n}(R_{w,h};0;1)$.}
\label{GeneratorsH1SC6R421}
\end{figure}
\end{proof}

\subsection{$d=3$: the case $n=2w-3$}\label{n=2w-3 subsection}

In this subsection, we build on the previous two subsections, proving Lemma~\ref{first homology height 2 for 2w-3} which handles the $d=3$ or $n=2w-3$ case of Theorem~\ref{first homology of height 2 box}.
Though the details of the proof differ substantially from the $d=2$ case, the proof of this lemma is another application of our three step ladder.
First, we make an observation about an orbiting triple square class in $H_{1}\big(SC_{3}(R_{3,2})\big)$ that will allow us to do away with these classes for $d\ge 3$.

\begin{prop}\label{orbiting triple classes are sums of squares}
The orbiting triple square class $\theta\in H_{1}\big(SC_{3}(R_{3,2})\big)$ is a sum of two orbiting pair square classes $\sigma_{1},\sigma_{2}\in H_{1}\big(SC_{3}(R_{3,2})\big)$.
\end{prop}

\begin{proof}
Figure \ref{Decomposetripletosumofsquares} depicts the orbiting triple class $\theta$ and the two orbiting pair square classes $\sigma_{1},\sigma_{2}$ that sum to it.
\begin{figure}[h]
\centering
\captionsetup{width=.8\linewidth}
\includegraphics[width=11.25cm]{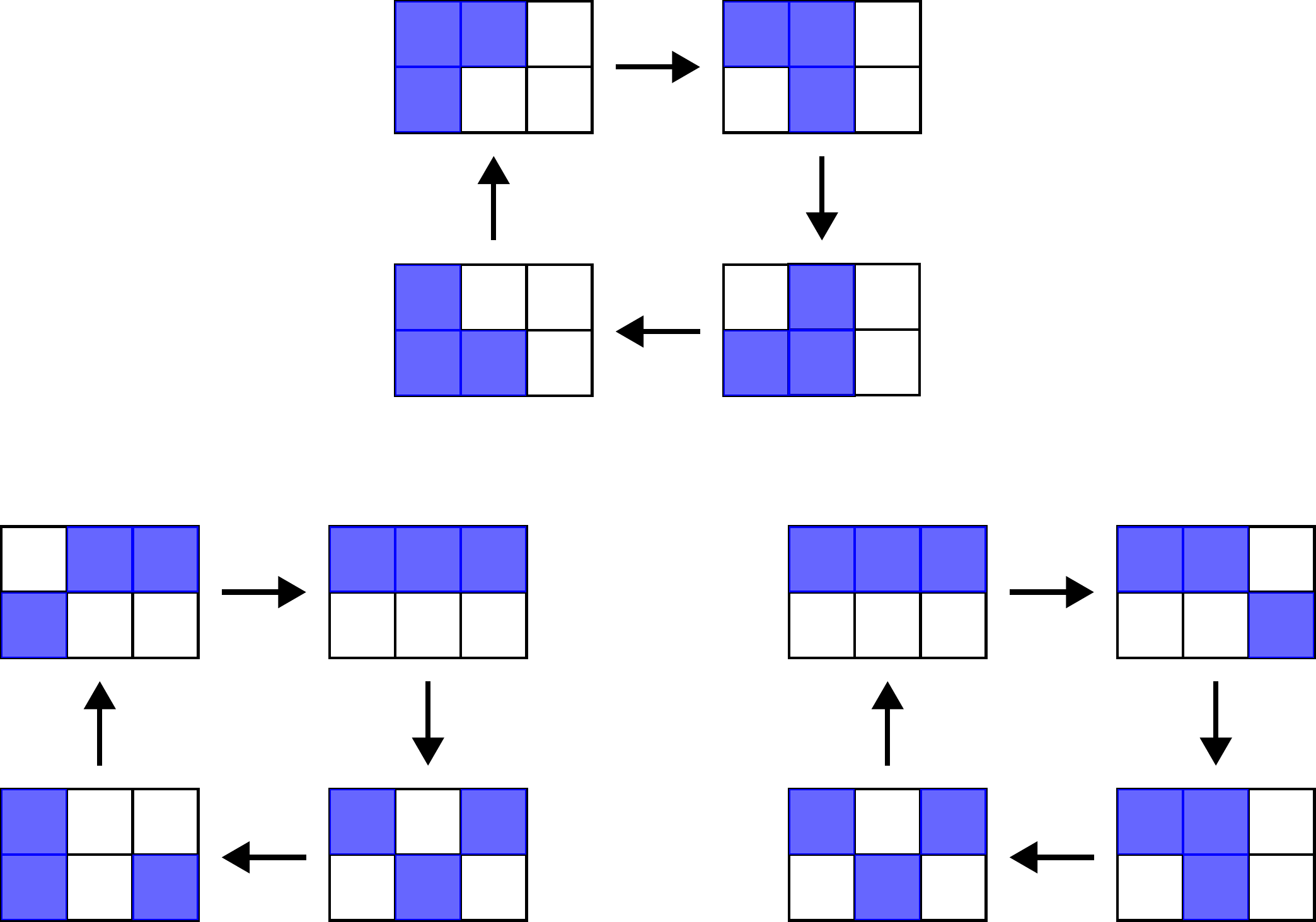}
\caption{At the top, an orbiting triple square class  $\theta\in H_{1}\big(SC_{3}(R_{3,2})\big)$, and, at the bottom, two orbiting pair square classes $\sigma_{1},\sigma_{2}\in H_{1}\big(SC_{3}(R_{3,2})\big)$ that sum to $\theta$.
}
\label{Decomposetripletosumofsquares}
\end{figure}
\end{proof}

\begin{lem}\label{first homology height 2 for 2w-3}
If $w\ge h=2$, then 
\[
H_{1}\big(SC_{2w-3}(R_{w,2})\big)\cong \Z^{2w-4},
\]
and it is generated by the $2w-4$ orbiting pair square classes corresponding to the possible ways of placing the $2w-5$ singleton squares on either side of the orbiting pair.
\end{lem}

\begin{proof}
We proceed by induction on $w$.
For $w=2$, Alpert, Bauer, Kahle, MacPherson, and Spendlove's work \cite[Theorem 3.5]{alpert2023homology} shows that $SC_{1}(R_{2,2})$ is homotopy equivalent to the figure at the top left of Figure~\ref{Configsin2by2}.
Therefore, $H_{1}\big(SC_{1}(R_{2,2})\big)=0$, proving the base case of our induction.

For the induction step, assume the lemma holds for all $w\le \omega$ for some $\omega\ge 2$.
We will show that the lemma holds for $w=\omega+1$ by using the three steps of our ladder, first using the inductive hypothesis to compute $H_{1}\big(SC_{2\omega-1}(R_{\omega+1, 2};1;1)\big)$, next using this to compute $H_{1}\big(SC_{2\omega-1}(R_{\omega+1, 2};1)\big)$, and finally using that to compute $H_{1}\big(SC_{2\omega-1}(R_{\omega+1, 2})\big)$.

\emph{Rung A:} The map $\alpha$ is injective.
To see this, note that $V_{1}\cap V_{2}\simeq SC_{2\omega-3}(R_{\omega-1,2})$ by Proposition~\ref{intersection of V1 and V2 h=2}.
This fact and Proposition~\ref{first homology height 2 for 2w-1} prove that $H_{1}(V_{1}\cap V_{2})\cong \Z^{\omega-2}$; moreover, they prove that this group is generated by the $\omega-2$ orbiting triple square classes.
Next, note that the sets $V_{1}$ and $V_{2}$ are homotopy equivalent to $SC_{2\omega-1}(R_{\omega+1, 2};2)\simeq SC_{2\omega-3}(R_{\omega, 2})$ by Proposition~\ref{Vl is SCp2}, and it follows from the inductive hypothesis that $H_{1}(V_{1})\oplus H_{1}(V_{2})\cong \Z^{4\omega-8}$, and that this group is generated by the $2\omega-4$ pairs of orbiting pair square classes corresponding to the $2\omega-4$ possible ways of placing the $2\omega-5$ singleton squares labeled $1$ on either side of the orbiting pair.
Here, each pair of classes is identical up to swapping the locations of the squares labeled $2$ and $3$; see the bottom of Figure~\ref{ClassinV1intersectV2anditsimage}.
By Proposition~\ref{orbiting triple classes are sums of squares} one can check that the image of an orbiting triple square class is a signed sum of four orbiting pair square classes as in Figure ~\ref{ClassinV1intersectV2anditsimage}.
\begin{figure}[h]
\centering
\captionsetup{width=.8\linewidth}
\includegraphics[width=13.75cm]{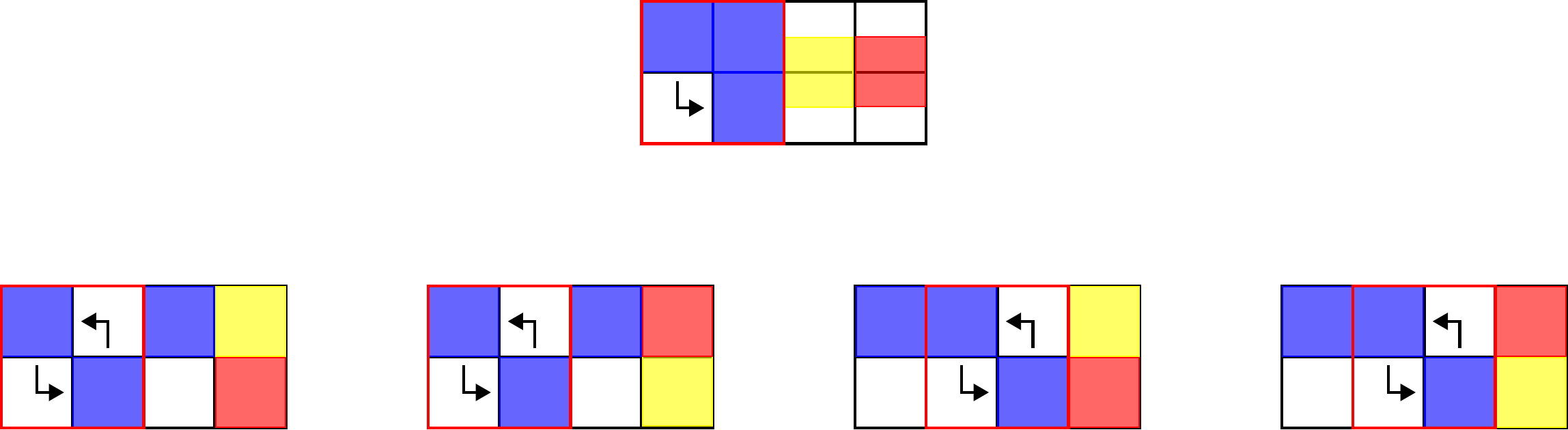}
\caption{At the top, a class in $H_{1}(V_{1}\cap V_{2})$, and, at the bottom, the four generators of $H_{1}(V_{1})\oplus H_{1}(V_{2})$ whose signed sum makes up its image under the $d^{1}$-differential.
Namely, its image is the difference of the first and third classes with the second and fourth classes.
It follows that the first three of the classes at the bottom form a basis for $H_{1}\big(SC_{5}(R_{4,2};1;1)\big)$.
This relies on the fact that in $H_{1}\big(SC_{3}(R_{3,2})\big)$ an orbiting triple square class can be decomposed into the sum of two orbiting pair square classes via Proposition~\ref{orbiting triple classes are sums of squares}.
}
\label{ClassinV1intersectV2anditsimage}
\end{figure}
These four orbiting pair square classes come in pairs of two, and each class only appears in the image of the differential of one of the generators of $H_{1}(V_{1}\cap V_{2})$.
It follows that $\operatorname{coker}(\alpha)\cong \Z^{3\omega-6}$, so $H_{1}\big(SC_{2\omega-1}(R_{\omega+1, 2};1;1)\big)$ has a basis consisting of $3\omega-6$ orbiting pair square classes by Proposition~\ref{rung A prop}.

\emph{Rung B:}  
Observe that $SC_{2\omega-1}(R_{\omega+1,2};2;1)\simeq SC_{2\omega-4}(R_{\omega-1,2})$.
This combined with Lemma~\ref{first homology height 2 for 2w-2}, tells us that  
\[
H_{1}\big(SC_{2\omega-1}(R_{\omega+1,2};2;1)\big)\cong H_{1}\big(SC_{2\omega-4}(R_{\omega-1,2})\big)\cong\Z^{2\omega-5},
\]
and that it is generated by the $\omega-2$ even orbiting pair square classes and the $\omega-3$ even orbiting triple square classes. 

The map $\beta$ sends the orbiting triple square classes generating $H_{1}\big(SC_{2\omega-1}(R_{\omega+1,2};2;1)\big)$ to $0$ and the orbiting pair square classes generating $H_{1}\big(SC_{2\omega-1}(R_{\omega+1,2};2;1)\big)$ to the difference of the pair of orbiting pair square classes generating $H_{1}\big(SC_{2\omega-1}(R_{\omega+1,2};1;1)\big)$ described in Rung A, serving to identify them; see Figure~\ref{ClassinH1SC5R4221anditsimage}.
\begin{figure}[h]
\centering
\captionsetup{width=.8\linewidth}
\includegraphics[width=6.25cm]{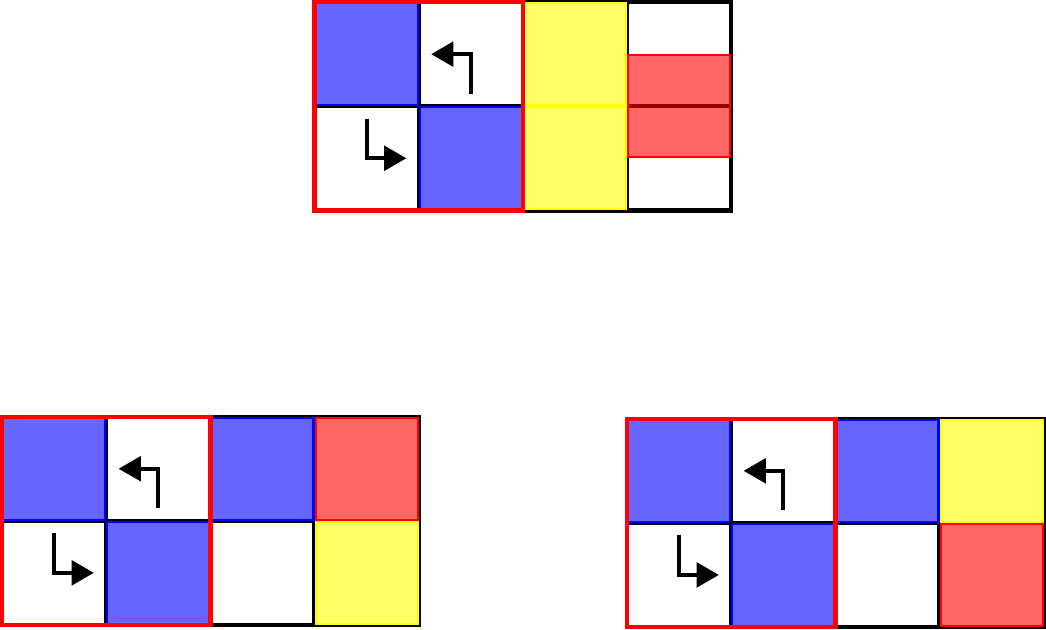}
\caption{On the top, a class in $H_{1}\big(SC_{5}(R_{4,2};2;1)\big)$, and on the bottom, the two classes whose difference is its image in $H_{1}\big(SC_{5}(R_{4,2};1;1)\big)$; compare with Figure~\ref{ClassinV1intersectV2anditsimage}.
For $\omega>3$, one can see that an orbiting triple square class in $H_{1}\big(SC_{2\omega-3}(R_{\omega+1,2};2;1)\big)$ can be realized in such a way that the column immediately to the left of the yellow and red squares (the squares labeled $2$ and $3$, respectively) is unoccupied.
This proves that under the $d^{1}$-differential, such a class goes to $0$.
}
\label{ClassinH1SC5R4221anditsimage}
\end{figure}
This identification of orbiting pair square classes tells us that $\operatorname{coker}(\beta)\cong \Z^{2\omega-4}$; moreover, it shows that we can take the $2\omega-4$ orbiting pair square classes corresponding to the $2\omega-4$ ways of placing the $2\omega-5$ singleton squares labeled $1$ on either side of the orbiting pair square class.
Therefore, Proposition~\ref{rung B prop} tells us that 
\[
H_{1}\big(SC_{2\omega-1}(R_{\omega+1,2};1)\big)\cong \Z^{2\omega-4} \oplus \Z,
\]
where the $\Z^{2\omega-4}$ summand is generated by the $2\omega-4$ orbiting pair square classes with at least one singleton square labeled $1$ to the right of the orbiting pair.
The $\Z$ summand is the image of the $d^{2}$-differential, which is an orbiting pair square class with only the square labeled $2$ to its right by Proposition~\ref{spectral sequence d2 differential } as we can ensure that the grid squares immediately to the left of the squares labeled $2$ are unoccupied.

\emph{Rung C:} We claim that $\gamma$ has trivial image, so $\operatorname{coker}(\gamma)=\Z^{2\omega-3}$ as $H_{1}\big(SC_{2\omega-1}(R_{\omega+1, 2};1)\big)\cong \Z^{2\omega-3}$ by Rung B.
To see this note that $H_{1}\big(SC_{2\omega-1}(R_{\omega+1, 2};2)\big)\cong \Z^{2\omega-4}$ by the inductive hypothesis, as $SC_{2\omega-1}(R_{\omega+1, 2};2)\simeq SC_{2\omega-3}(R_{\omega, 2})$ by sliding the squares labeled $2$ to the far right of $R_{\omega+1,2}$.
Moreover, this group is generated by the $2\omega-4$ orbiting pair square classes.
The map $\gamma$ sends an orbiting pair square class to $0$ as the two classes whose difference constitutes its image are homologous via the relations produced in Rung B; see Figure~\ref{ClassinH1SC5R422anditsimage}. 
It follows from Proposition~\ref{rung C prop} that 
\[
H_{1}\big(SC_{2\omega-1}(R_{\omega+1,2})\big)\cong \Z^{2\omega-3} \oplus \Z,
\]
where the $\Z^{2\omega-3}$ summand is generated by $2\omega-3$ orbiting pair square classes in which at least one singleton square is to the right of the orbiting pair, and the $\Z$ summand is the image of the $d^{2}$-differential.
By Proposition~\ref{spectral sequence d2 differential }, this image corresponds to an orbiting pair square class at the far right of $R_{\omega+1,2}$, completing the inductive step and hence the proof.
\begin{figure}[h]
\centering
\captionsetup{width=.8\linewidth}
\includegraphics[width=6.25cm]{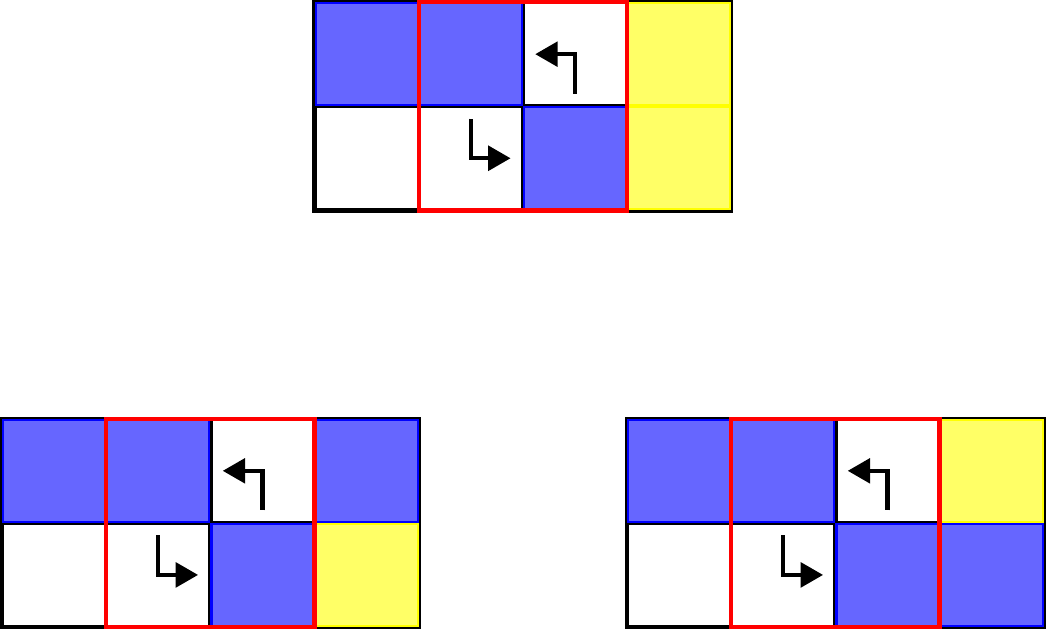}
\caption{On the top, an orbiting pair square class in $H_{1}\big(SC_{5}(R_{4,2};2)\big)$; on the bottom, the two orbiting pair square classes in $H_{1}\big(SC_{5}(R_{4,2};1)\big)$ whose difference is its image under the $d^{1}$-differential. 
These classes are homologous via the relations produced in Rung B.}
\label{ClassinH1SC5R422anditsimage}
\end{figure}
\end{proof}

We use Lemma~\ref{first homology height 2 for 2w-3} as the base case of an inductive argument in the next subsection, where we handle the $n\le 2w-4$ situation.

\subsection{$d\ge 4$: the case $n\le 2w-4$}\label{n le 2w-4 subsection}

In this subsection we complete the proof of Theorem~\ref{first homology of height 2 box} by computing $H_{1}\big(SC_{n}(R_{w,2})\big)$ for $1\le n\le 2w-4$.
To take advantage of our previous results, we compute $H_{1}\big(SC_{n}(R_{w,2})\big)$ for $1\le n\le 2w-3$.
We do this by first inducting on $d$, using Lemma~\ref{first homology height 2 for 2w-3} to handle the base case.
Then, for fixed $d\ge 4$, we induct on $w$, where the argument is similar to the previous two cases, as we rely on our three step ladder to handle the inductive step of our argument, though this case is the most straightforward of the three.

\begin{lem}\label{first homology height 2 for n le 2w-4}
If $w\ge h=2$ and $1\le n\le 2w-3$, then 
\[
H_{1}\big(SC_{n}(R_{w,2})\big)\cong \Z^{n-1},
\]
and it is generated by the $n-1$ orbiting pair square classes corresponding to the $n-1$ possible ways of placing the $n-2$ singleton squares on either side of the orbiting pair.
\end{lem}

\begin{proof}
We proceed by induction on $d=2w-n$.
Lemma~\ref{first homology height 2 for 2w-3} handles the base case.
Next, assume that the lemma holds for all $3\le d\le \delta$ for some $\delta\ge 3$; we will show that the lemma holds for $d=\delta+1$ by inducting on $w$.
In the first non-trivial case we have $w=\left\lceil\frac{\delta+2}{2}\right\rceil$, which must is at least $3$.
In either case, there are straight-line homotopies between $SC_{1}(R_{w,2})$ and $SC_{2}(R_{w,2})$ and the complexes on the top right and top left of Figure~\ref{Configsin2by2}, respectively.
From this picture it follows that $H_{1}\big(SC_{1}(R_{w,2})\big)=0$ and $H_{1}\big(SC_{2}(R_{w,2})\big)\cong\Z$, proving the base case of our induction on $w$.

Next, assume that the lemma holds for all $w\le \omega$, for some $\omega\ge 3$.
We will show that the lemma holds for $w=\omega+1$ and all $n=2\omega-\delta+1$ by using the three steps of our ladder, first using the inductive hypothesis to compute $H_{1}\big(SC_{n}(R_{\omega+1, 2};1;1)\big)$, next using this to compute $H_{1}\big(SC_{n}(R_{\omega+1, 2};1)\big)$, and finally using that to compute $H_{1}\big(SC_{n}(R_{\omega+1, 2})\big)$.

\emph{Rung A:} We claim that $\alpha$ is injective.
To see this, first note that $V_{1}$ and $V_{2}$ are homotopy equivalent to $SC_{n}(R_{\omega+1, 2};2)\simeq SC_{n-2}(R_{\omega, 2})$ by Proposition~\ref{Vl is SCp2}.
Therefore, $H_{1}(V_{1}) \oplus H_{1}(V_{2})\cong \Z^{2n-6}$ by the inductive hypothesis. 
Furthermore, this direct sum is generated by $n-3$ pairs of orbiting pair square classes.
By Proposition~\ref{intersection of V1 and V2 h=2}, we have that $V_{1}\cap V_{2}\simeq SC_{n-2}(R_{\omega-1,2})$.
Therefore $H_{1}(V_{1}\cap V_{2})\cong \Z^{n-3}$ by our inductive hypothesis on $d$ if $\delta+1>4$, i.e.,  $n<2\omega-2$, and by Lemma~\ref{first homology height 2 for 2w-2} if $\delta+1=4$, i.e., $n=2\omega-2$.
Additionally, if $n<2\omega-2$, this group is generated by $n-3$ orbiting pair square classes; if $n=2\omega-2$, it is generated by $\omega-2$ orbiting pair square classes and $\omega-3$ orbiting triple square classes.

Next, note that the image of an orbiting pair square class generating $H_{1}(V_{1}\cap V_{2})$ under $\alpha$ is the difference of a pair of generators of $H_{1}(V_{1})\oplus H_{1}(V_{2})$.
For $n=2\omega-2$, one can similarly check that each of the orbiting triple square classes gets sent of a pair of orbiting triple square classes in $H_{1}(V_{1})\oplus H_{1}(V_{2})$ that decomposes as a sum of orbiting pair square classes, one of which is already accounted for by the orbiting pair square classes coming from $H_{1}(V_{1}\cap V_{2})$, the other of which is not.
In either case, $\alpha$ is injective, so $\operatorname{coker}(\alpha)\cong \Z^{n-3}$, and $H_{1}\big(SC_{n}(R_{\omega+1, 2};1;1)\big)$ is generated by $n-3$ orbiting pair square classes corresponding to the ways of placing the $n-4$ singleton squares labeled $1$ on either side of the orbiting pair square class.

\emph{Rung B:} We claim that $\beta$ is the zero map.
To see this observe that $SC_{n}(R_{\omega+1,2};2;1)\simeq SC_{n-3}(R_{\omega-1,2})$.
This, combined with Lemma~\ref{first homology height 2 for 2w-3} if $\delta+1=4$ or the inductive hypothesis on $d$ if $\delta+1>4$, tells us that  
\[
H_{1}\big(SC_{n}(R_{\omega+1,2};2;1)\big)\cong H_{1}\big(SC_{n-3}(R_{\omega-1,2})\big)\cong\Z^{n-4},
\]
and that it is generated by the $n-4$ orbiting pair square classes corresponding to the ways of putting the $n-5$ singleton squares labeled $1$ on either side of the orbiting pair. 
There is enough space in $R_{\omega+1,2}$ to rearrange the squares to ensure that the two orbiting pair square classes $\beta$ produces are homologous; compare with Figure~\ref{ClassinH1SC5R4221anditsimage}.
It follows that $\beta$ is the zero map.

By Rung A we have that 
\[
H_{1}\big(SC_{n}(R_{\omega+1,2};1;1)\big)\cong \Z^{n-3}.
\]
Since $\beta$ is the zero map, this is the cokernel of $\beta$.
Therefore, Proposition~\ref{rung B prop} tells us that 
\[
H_{1}\big(SC_{n}(R_{\omega+1,2};1)\big)\cong \Z^{n-3} \oplus \Z,
\]
where the $\Z^{n-3}$ summand is generated by the $n-3$ orbiting pair square classes with at least one singleton square labeled $1$ to the right of the orbiting pair.
The $\Z$ summand is the image of the $d^{2}$-differential, which is an orbiting pair square class with only the square labeled $2$ to its right by Proposition~\ref{spectral sequence d2 differential }.

\emph{Rung C:}
We claim that $\gamma$ has trivial image, so $\operatorname{coker}(\gamma)=\Z^{n-2}$ as $H_{1}\big(SC_{n}(R_{\omega+1, 2};1)\big)\cong \Z^{n-2}$ by Rung B.
By the inductive hypothesis, $H_{1}\big(SC_{n}(R_{\omega+1, 2};2)\big)\cong \Z^{n-3}$ since $SC_{n}(R_{\omega+1, 2};2)\simeq SC_{n-2}(R_{\omega, 2})$ by sliding the squares labeled $2$ to the far right of $R_{\omega+1,2}$.
Moreover, this entry is generated by the $n-3$ orbiting pair square classes.
We claim that every orbiting pair square class generating $H_{1}\big(SC_{n}(R_{\omega+1, 2};2)\big)$ has trivial image in $H_{1}\big(SC_{n}(R_{\omega+1, 2};1)\big)$.
This follows from the fact that the two classes whose difference constitutes its image are homologous; see Figure~\ref{ClassinH1SC4R422anditsimage}.
It follows from Proposition~\ref{rung C prop} that $H_{1}\big(SC_{n}(R_{\omega+1, 2})\big)\cong \Z^{n-1}$.
Moreover, that proposition and Proposition~\ref{spectral sequence d2 differential } prove that this group is generated by the $n-1$ orbiting pair square classes corresponding to the ways of placing the $n-2$ singleton squares on either side of the orbiting pair.
This completes the inductive step on $w$, hence the inductive step on $d$ and the proof of the lemma.
\begin{figure}[h]
\centering
\captionsetup{width=.8\linewidth}
\includegraphics[width=6.25cm]{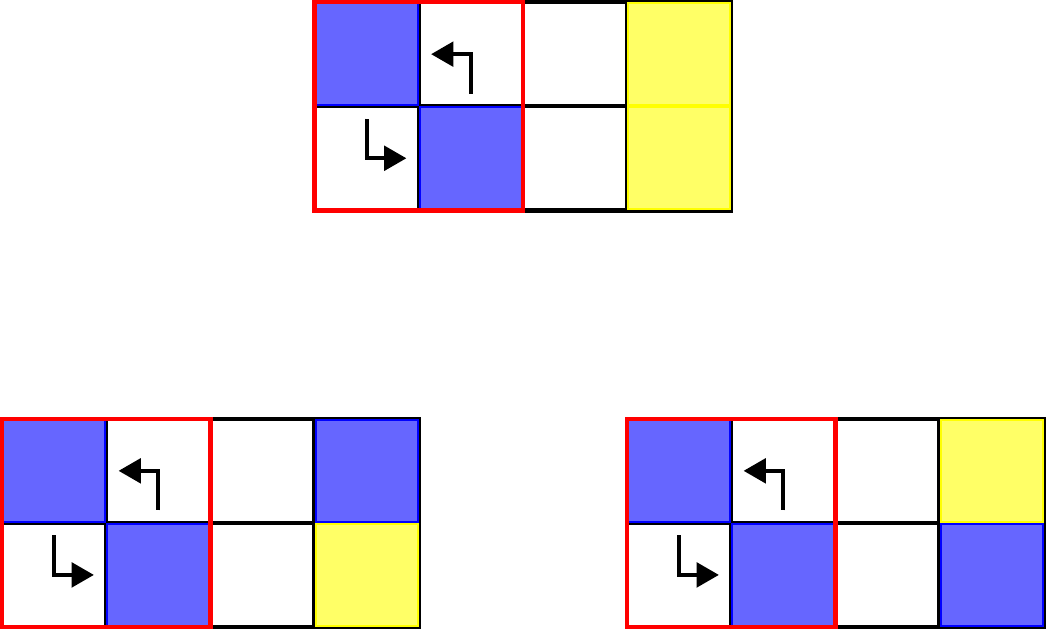}
\caption{On the top, an orbiting pair square class in $H_{1}\big(SC_{4}(R_{4,2};2)\big)$; on the bottom, the two orbiting pair square classes in $H_{1}\big(SC_{4}(R_{4,2};1)\big)$ whose difference is its image under the $d^{1}$-differential. 
The excess of unoccupied space in $R_{4,2}$ ensures that these classes are homologous.}
\label{ClassinH1SC4R422anditsimage}
\end{figure}
\end{proof}

Having finished the $d\ge 4$ case, we have proven Theorem~\ref{first homology of height 2 box}:

\begin{proof}[Proof of Theorem~\ref{first homology of height 2 box}]
Proposition~\ref{first homology height 2 for 2w-1} proves the theorem in the case $n=2w-1$.
Lemmas~\ref{first homology height 2 for 2w-2}, \ref{first homology height 2 for 2w-3}, and \ref{first homology height 2 for n le 2w-4} prove the theorem for $n=2w-2$, $n=2w-3$, and $1\le n\le 2w-4$, respectively.
The $n<1$ and $n\ge 2w$ cases are trivial as the configuration space is either contractible or empty.
\end{proof}

In the next section we use our descriptions of the bases for $H_{1}\big(SC_{n}(R_{w,2})\big)$ to calculate $H_{1}\big(SC_{n}(R_{w,h})\big)$ for $w,h\ge 3$ and $n\le wh-3$.

\section{$H_{1}\big(SC_{n}(R_{w,h})\big)$ for $w,h\ge3$}\label{section h>2}

In this section, we determine $H_{1}\big(SC_{n}(R_{w,h})\big)$ for all $n\le wh-3$.
Namely, we prove the following theorem:

\begin{T1}
  \newthmtext
\end{T1}

Since $H_{1}\big(C_{n}(\R^{2})\big)$ is isomorphic to $\Z$ for $n\ge 2$ and is trivial otherwise, this result, Alvarado-Gardu\~{n}o and Gonz\'{a}lez's result for $n=wh-2$ \cite[Theorem 2.2] {alvarado2025square}, 
Alvarado-Gardu\~{n}o, Gonz\'{a}lez, and Kahle's result for $n=wh-1$  \cite[Theorem 1.2]{MattOmarJesus}, and the triviality of these spaces for $n\ge wh$ gives the complete description of $H_{1}\big(SC_{n}(R_{w,h})\big)$ for $w,h\ge 3$, which we have recorded in Corollary~\ref{First homology for w,h ge 3}.

Theorem~\ref{space homological stability unordered} reduces to two statements about the behavior of orbiting pair square classes in $H_{1}\big(SC_{n}(R_{w,h})\big)$ when $w,h\ge 3$ and $n\le wh-3$:
\begin{enumerate}
  \item[(G)]\label{statement G:G} \emph{Generation.} $H_{1}\big(SC_{n}(R_{w,h})\big)$ is generated by orbiting pair square classes.
  
  \item[(H)]\label{statement H:H} \emph{Homologousness.}  Any two orbiting pair square classes in $H_{1}\big(SC_{n}(R_{w,h})\big)$ are homologous.
\end{enumerate}
Together, (G) and (H) say that $H_{1}\big(SC_{n}(R_{w,h})\big)$ is cyclic as long as $w,h\ge 3$ and $n\le wh-3$. 
Since an orbiting pair square
class maps to a generator of $H_{1}\big(C_{n}(\R^2)\big)$ under the inclusion $SC_{n}(R_{w,h})\hookrightarrow C_{n}(\R^2)$ and a cyclic group surjecting onto $\mathbb{Z}$ must be $\mathbb{Z}$, this map is an isomorphism.
 
Statement (H) is a fact about the geometry of sliding squares in the rectangle, and we prove it in
Subsection~\ref{orbiting pair square classes are homologous subsection}. 
The difficulty lies in (G), which we prove by induction on $w$. 
The inductive step recovers the homology of $SC_{n}(R_{w,h};p)$ from the homology of $SC_{n}(R_{w,h};p+1)$ and lowers $p$ one step at a time using the augmented
right Mayer--Vietoris spectral sequence of Subsection~\ref{right Mayer--Vietoris spectral sequence subsection}. 
The $E^{1}_{0,1}$-entry of that spectral sequence is $H_{1}\big(SC_{n}(R_{w,h};1;p)\big)$, which we control with the augmented slide spectral sequence of Subsection~\ref{slide spectral sequence subsection}.
 
The reason this induction is delicate is that if $h\ge 4$, the slide spectral sequence can produce classes that are \emph{not} orbiting pair square classes. 
This behavior only happens if $n = (w-1)h+p$.
In that case, the rectangle is so full that the single square labeled $2$ is trapped in the second rightmost column, and the nerve of the slide cover acquires first homology.
As a result, the $\tilde{E}^{2}_{1,0}$-entry contributes a free summand to $H_{1}\big(SC_{n}(R_{w,h};1;p)\big)$, whose generators we call \emph{exotic classes} (Definition~\ref{H1 SC(1p1) for n-p-1=(w-1)h}).
Exotic classes must be carried through the entire induction, though in Proposition~\ref{d1 on exotic classes is 0} we prove that they die in $H_{1}\big(SC_{n}(R_{w,h})\big)$, which is why they leave no trace in Theorem~\ref{space homological stability unordered}.
 
The section is organized as follows.
\begin{itemize}
    \item Subsection~\ref{descent lemma subsection} explains how the induction that proves (G) works.
    It reduces to a single \emph{descent lemma}, and the proof of this lemma leaves exactly two entries of an augmented right Mayer--Vietoris spectral sequence unidentified.
    The next two subsections identify them.

    \item Subsection~\ref{zeroth row subsection} determines the $E^{2}_{1,0}$-entry.
    We do this by counting the components of $SC_{n}(R_{w,h};s+1;p)$, which computes the zeroth row of the spectral sequence.

    \item Subsection~\ref{E101 entry subsection} determines the $E^{1}_{0,1}$-entry of this spectral sequence.
    Namely, it shows that this entry, which can be identified with $H_{1}\big(SC_{n}(R_{w,h};1;p)\big)$, is a quotient of $H_{1}\big(SC_{n}(R_{w,h};p+1)\big)^{\oplus (p+1)}$ together with a free summand generated by exotic classes.
    
    \item Subsection~\ref{generation by orbiting pair subsection} proves statement (G).
    
    \item Subsection~\ref{orbiting pair square classes are homologous subsection} proves statement (H).
    
    \item Subsection~\ref{Proof of Theorem space homological stability unordered subsection} assembles the two into a proof of Theorem~\ref{space homological stability unordered}.
\end{itemize}

A reader who only wants the shape of the argument can read Subsections \ref{descent lemma subsection} and \ref{generation by orbiting pair subsection} and take the identifications on faith.

\subsection{Strategy: the descent lemma}\label{descent lemma subsection}

Statement (G), the claim that $H_{1}\big(SC_{n}(R_{w,h})\big)$ is generated by orbiting pair square classes, is proved by induction on $w$, but what propagates through the induction is not the generation of $H_{1}\big(SC_{n}(R_{w,h})\big)$ itself.
Rather, it is the generation of $H_{1}\big(SC_{n}(R_{w,h};p+1)\big)$ for every $p$, and the induction moves in two ways:

\begin{itemize}
    \item \emph{Across w.} If we have $h$ squares labeled $2$, we can push them into the rightmost column to see that $SC_{n}(R_{w+1,h};h)\simeq SC_{n-h}(R_{w,h})$.
    This converts information about a thinner rectangle into information about a wider one, giving the base case of the next argument.

    \item \emph{Down p.} A \emph{descent lemma} (Lemma~\ref{intermediate generated by orbiting square classes}) converts information about $SC_{n}(R_{w,h};p+1)$ into information about $SC_{n}(R_{w,h};p)$.
\end{itemize}

Starting with $R_{2,h}$ and alternating between these two arguments yields statement (G).
As the first move is a straight-line homotopy, almost all of the work is concentrated in the descent lemma.
This lemma is proved by running the augmented right Mayer--Vietoris spectral sequence for $SC_{n}(R_{w,h};p)$.
Its abutment is trivial by Proposition~\ref{augmented right MVSS has trivial abutment}, so $H_{1}\big(SC_{n}(R_{w,h};p)\big)=E^{1}_{-1,1}\big[SC_{n}(R_{w,h};p)\big]$ is determined by the $E^{1}_{0,1}$- and $E^{2}_{1,0}$-entries.
In the next two subsections we determine these entries.

\subsection{The zeroth row of the augmented right Mayer--Vietoris spectral sequence}\label{zeroth row subsection}

In this subsection we determine the $E^{2}_{1,0}$-entry of the augmented right Mayer--Vietoris spectral sequence of $SC_{n}(R_{w,h};p)$.
We do this by computing the entire zeroth row of the first page of this spectral sequence and using our knowledge of the $d^{1}$-differential to show that in all but one case the $E^{2}_{1,0}$-entry is $0$.

\begin{prop}\label{connected unordered square configuration spaces two rightmost}
Assume $w,h\ge 3$, $n\le wh-3$, $0< s+1\le 3$, and $0<p\le h$.
The square configuration space $SC_{n}(R_{w,h};s+1; p)$ is connected in each of the following cases:
\begin{enumerate}
  \item[(i)] $s+p+1 > h$;
  
  \item[(ii)] $s+p+1 \le h$ and $n \le (w-1)h+p-1$;
  
  \item[(iii)] $s+p+1 < h$ and $n = (w-1)h+p$.
\end{enumerate}
Otherwise, $SC_n(R_{w,h};s+1;p)$ has $\binom{s+p+1}{p}$ components.
\end{prop}

\begin{proof}
Throughout the proof, we use the fact that we can slide the squares labeled $3$ into the rightmost column of $R_{w,h}$ and the squares labeled $2$ into the second rightmost column.
First, we treat the three connected cases, then we treat the disconnected ones.

\emph{Case (i):} If $s+p+1>h$, then only squares labeled $3$ can be in the rightmost column of $R_{w,h}$.
If we forget that column, we get a configuration in $SC_{n-p}(R_{w-1, h}; s+1)$.
This configuration space is connected by Proposition~\ref{connected unordered square configuration spaces one rightmost}, so $SC_{n}(R_{w,h};s+1; p)$ is as well.

\emph{Case (ii):} If $s+p+1\le h$ and $n\le (w-1)h+p-1$, and the $s+1$ squares labeled $2$ are not in the rightmost column of $R_{w,h}$, then we may slide the $p$ squares labeled $3$ up and down so that the grid squares to the right of the squares labeled $2$ are unoccupied, allowing us to move them into the rightmost column. 
Therefore, we may assume that all squares labeled $2$ or $3$ are in the rightmost column of $R_{w,h}$.
Treating all squares labeled $2$ or $3$ as rightmost, Proposition~\ref{connected unordered square configuration spaces one rightmost} ensures that they can be moved into the first $s+p+1$ entries in the rightmost column.

We show that the squares labeled $2$ can be moved to the top of this column.
Since $n\le (w-1)h+p-1$, if there is a square labeled $3$ above a square labeled $2$, we may slide the squares labeled $1$ about so that the grid squares immediately to the left of the squares labeled $2$ are unoccupied and such that the square one to the left and one up from the highest square labeled $2$ that is below a square labeled $3$ is unoccupied. 
Moreover, we can do this while keeping the squares labeled $1$ out of the rightmost column of $R_{w,h}$.
Therefore, we can slide the squares labeled $2$ into the second rightmost column, and the highest square labeled $2$ that was below a square labeled $3$ up one, while sliding the square labeled $3$ that was above this square down one.
Sliding the squares labeled $2$ back into the rightmost column, we see that at least one of them could be moved up.
Repeating this process if necessary, all the squares labeled $2$ can be put directly above all of the squares labeled $3$ in the rightmost column, showing that this configuration space is connected.

\emph{Case (iii):} If $s+p+1<h$ and $n=(w-1)h+p$, and the $s+1$ squares labeled $2$ are not in the rightmost column of $R_{w,h}$, then we may slide the $p$ squares labeled $3$ up and down so that the grid squares to the right of the squares labeled $2$ are unoccupied, allowing us to move them into the rightmost column. 
Therefore, we may assume that all squares labeled $2$ or $3$ are in the rightmost column of $R_{w,h}$.
Treating all squares labeled $2$ or $3$ as rightmost, Proposition 
\ref{connected unordered square configuration spaces one rightmost} allows us to assume that the squares labeled $2$ and $3$ are at the top of this column.

We prove that the squares can be slid around the rectangle so that the squares labeled $2$ are above all the squares labeled $3$.
If there is a square labeled $3$ directly above a square labeled $2$, note that since $n=(w-1)h+p$, we can slide the squares labeled $1$ about the left part of $R_{w,h}$ so that the squares immediately to the left of the squares labeled $2$ are unoccupied.
Furthermore, since $n\le wh-3$, $s+1\le 3$, and $s+p+1<h$, after pushing the squares labeled $3$ that are below all the squares labeled $2$ to the bottom of the rightmost column, we can slide the lowest square labeled $3$ that is immediately above a square labeled $2$ down so that it is now below it.
Sliding the squares labeled $2$ back to the right, we see that they can be moved up in the rightmost column.
Repeating this process if necessary, we see that the squares labeled $2$ can be put at the top of the rightmost column with the squares labeled $3$ below them, showing that this configuration space is connected.
 
\emph{The disconnected cases:} It remains to show that outside of these cases $SC_{n}(R_{w,h}; s+1;p)$ has $\binom{s+p+1}{p}$ components.
We do this by dividing this situation into two cases:
\begin{enumerate}
    \item[(iv)] $s+p+1=h$ and $n\ge (w-1)h+p$;

    \item [(v)] $s+p+1<h$ and $n\ge (w-1)h+p+1$.
\end{enumerate}

\emph{Case (iv):} If $s+p+1=h$ and $n\ge (w-1)h+p$, then there is no room to slide the squares labeled $2$ around the rectangle.
If $n>(w-1)h+p$, the squares labeled $2$ are constrained to the rightmost column of $R_{w,h}$, and if $n=(w-1)h+p$, then they can be slid one to the right but not up or down.
As $s+p+1=h$, the $p$ squares labeled $3$ cannot be rearranged while the squares labeled $2$ are in the second rightmost column, so no new interweaving of the squares labeled $2$ and $3$ can be produced.
Since the space would be connected if we relabeled all of the squares labeled $3$ to $2$, there is one component for each way of placing the $s+1$ squares labeled $2$ between the $p$ squares labeled $3$, i.e., there are $\binom{s+p+1}{p}$ components.

\emph{Case (v):} If $s+p+1<h$ and $n\ge (w-1)h+p+1$, then the squares labeled $2$ must be in the rightmost column of $R_{w,h}$.
Since the space would be connected if we relabeled all of the squares labeled $3$ to $2$, we see that there is one component for each way of placing the $s+1$ squares labeled $2$ between the $p$ squares labeled $3$, so there are $\binom{s+p+1}{p}$ components.
\end{proof}

With this in mind we compute the $E^{2}_{1,0}$-entry of the augmented right Mayer--Vietoris spectral sequence of $SC_{n}(R_{w,h};p)$.
Since this spectral sequence has trivial abutment, this entry can be identified with a subquotient of $H_{1}\big(SC_{n}(R_{w,h};p)\big)=E^{1}_{-1,1}\big[SC_{n}(R_{w,h};p)\big]$.

\begin{prop}\label{E201 entry of augmeneted right Mayer Vietoris ss}
Assume $w,h\ge 3$, $0\le p\le h$, and $n\le wh-3$.
If $n=p+2$, then $E^{2}_{1,0}\big[SC_{n}(R_{w,h};p)\big]\cong \Z$; otherwise, $E^{2}_{1,0}\big[SC_{n}(R_{w,h};p)\big]=0$.
\end{prop}

\begin{proof}
If $n=p+2$, then $n-p-1=1\neq (w-1)h$.
It follows from Proposition~\ref{connected unordered square configuration spaces two rightmost} that $E^{1}_{1,0}\big[SC_{n}(R_{w,h}; p)\big]=H_{0}\big(SC_{n}(R_{w,h}; 2;p)\big)\cong \Z$.
Similarly, one can use Propositions~\ref{connected unordered square configuration spaces one rightmost} and \ref{connected unordered square configuration spaces two rightmost} to see that $E^{1}_{0,0}\big[SC_{n}(R_{w,h}; p)\big]=H_{0}\big(SC_{n}(R_{w,h}; 1;p)\big)\cong \Z$ and $E^{1}_{-1,0}\big[SC_{n}(R_{w,h}; p)\big]=H_{0}\big(SC_{n}(R_{w,h}; p)\big)\cong \Z$, whereas $E^{1}_{2,0}\big[SC_{n}(R_{w,h}; p)\big]=H_{0}\big(SC_{n}(R_{w,h}; 3;p)\big)=0$ as $SC_{n}(R_{w,h}; 3;p)$ is empty.
It follows from the definition of the differentials of the augmented right Mayer--Vietoris spectral sequence that $E^{2}_{1,0}\big[SC_{n}(R_{w,h}; p)\big]\cong \Z$. 

If $n\neq p+2$, then Proposition~\ref{connected unordered square configuration spaces two rightmost} tells us there are two possibilities to consider: $E^{1}_{1,0}\big[SC_{n}(R_{w,h};p)\big]\cong \Z$ and $E^{1}_{1,0}\big[SC_{n}(R_{w,h};p)\big]\cong \Z^{\binom{p+2}{p}}$.
In the first case, one can note that Proposition~\ref{connected unordered square configuration spaces two rightmost} implies that $SC_{n}(R_{w,h};1;p)$ is also connected, so the kernel of 
\[
d^{1}:E^{1}_{1,0}=H_{0}\big(SC_{n}(R_{w,h};2;p)\big)\to H_{0}\big(SC_{n}(R_{w,h};1;p)\big)=E^{1}_{0,0}
\]
is all of $E^{1}_{1,0}\big[SC_{n}(R_{w,h};p)\big]$.
This follows from the observation that relabeling either square labeled $2$ to $1$ gives homologous classes since $SC_{n}(R_{w,h};1;p)$ is connected.
Since $n\le wh-3$, we have that $SC_{n}(R_{w,h};2;p)$ is nonempty, and the image of 
\[
d^{1}:E^{1}_{2,0}=H_{0}\big(SC_{n}(R_{w,h};3;p)\big)\to H_{0}\big(SC_{n}(R_{w,h};2;p)\big)= E^{1}_{1,0}
\]
is all of $E^{1}_{1,0}\big[SC_{n}(R_{w,h};p)\big]$.
This follows from the fact that all three classes coming from relabeling a square labeled $2$ to $1$ in any representative of the natural basis for $H_{0}\big(SC_{n}(R_{w,h}; 3;p)\big)$ are homologous, so since two of them cancel, the third remains. 
Therefore $E^{2}_{1,0}\big[SC_{n}(R_{w,h};p)\big]\cong 0$.

If $E^{1}_{1,0}\big[SC_{n}(R_{w,h};p)\big]=H_{0}\big(SC_{n}(R_{w,h};2;p)\big)\cong \Z^{\binom{p+2}{p}}$, then $SC_{n}(R_{w,h};3;p)$ is non-empty as $n\le wh-3$.
Moreover, Proposition~\ref{connected unordered square configuration spaces two rightmost} implies $E^{1}_{2,0}\big[SC_{n}(R_{w,h};p)\big]\cong \Z^{\binom{p+3}{p}}$.
Further, Proposition~\ref{connected unordered square configuration spaces two rightmost} implies $E^{1}_{0,0}\big[SC_{n}(R_{w,h};p)\big]\cong \Z^{p+1}$.
To see this note that the only other possibility is that $E^{1}_{0,0}\big[SC_{n}(R_{w,h};p)\big]\cong \Z$, which is only possible if $n=wh-2$ given $E^{1}_{1,0}\big[SC_{n}(R_{w,h};p)\big]\cong \Z^{\binom{p+2}{p}}$.
In this case, the $E^{1}_{0,0}$-, $E^{1}_{1,0}$-, and $E^{1}_{2,0}$-entries of the augmented right Mayer--Vietoris spectral sequence correspond to the $2$-skeleton of a simplicial complex, whose $(q-1)$-simplices correspond to ways of shuffling $q$ elements into a set of $p$ elements.
This complex has trivial first homology, so $E^{2}_{1,0}\big[SC_{n}(R_{w,h};p)\big]\cong0$.
\end{proof}

\subsection{The $E^{1}_{0,1}$-entry}\label{E101 entry subsection}
In this subsection we determine the $E^{1}_{0,1}$-entry of the augmented right Mayer--Vietoris spectral sequence of $SC_{n}(R_{w,h};p)$.
We show that in most cases, this entry, which is given by $H_{1}\big(SC_{n}(R_{w,h};1;p)\big)$, is a quotient of $H_{1}\big(SC_{n}(R_{w,h};p+1)\big)^{\oplus (p+1)}$, though in certain situations new exotic classes arise.
Namely, we prove the following proposition:

\begin{prop}\label{H1 SC(1p1)}
Given  $w,h\ge 2$, $0<p<h$, and $n\le wh-3$, set
\[
\epsilon:=\begin{cases}
    p(h-p-2)&\text{if }h\ge 4\text{ and }n=(w-1)h+p, \text{ and}\\0&\text{otherwise.}
\end{cases}
\]
Then there is a natural short exact sequence 
\[
0\to Q\to H_{1}\big(SC_{n}(R_{w,h};1;p)\big)\to \Z^{\epsilon}\to0,
\]
where $Q$ is a quotient of $H_{1}\big(SC_{n}(R_{w,h};p+1)\big)^{\oplus (p+1)}$ and the free summand is generated by the exotic classes defined in Definition~\ref{H1 SC(1p1) for n-p-1=(w-1)h}.
\end{prop}

We use the augmented slide spectral sequence to recover this description of $H_{1}\big(SC_{n}(R_{w,h};1;p)\big)$.
Since the first page of this spectral sequence is determined by the intersections of the sets $V_{l}$ defined in Definition~\ref{V sets for slide ss}, we determine how these sets can intersect.

\begin{prop}\label{connected intersection of Vs n-p-1 < (w-1)h}
Let $V_{1},\dots, V_{p+1}$ be the subspaces of $SC_{n}(R_{w,h}; 1; p)$ used to define the augmented slide spectral sequence, and assume that $w,h\ge 2$ and $0\le p< h$.
Given $1\le s\le p$, fix $1\le l_{0}<\cdots<l_{s}\le p+1$ and write $V:=V_{l_{0}}\cap \cdots\cap V_{l_{s}}$.
\begin{enumerate}
    \item If $n<(w-1)h+p$, then $V\neq \emptyset$ if and only if $l_{s}-l_{0}\le h-p-1$, and in that case $V$ is connected.

    \item If $n=(w-1)h+p$, then $V\neq \emptyset$ if and only if $l_{s}-l_{0}\le h-p-1$, and in that case $V$ has $h-p+l_{0}-l_{s}$ components.

    \item If $n>(w-1)h+p$, then $V=\emptyset$. 
\end{enumerate}
\end{prop}

\begin{proof}
We begin by noting that in $V_{l_{i}}$, if the square labeled $2$ is not in the same column of $R_{w,h}$ as the squares labeled $3$, then the highest row it can fully be in is the $l_{i}$-th row of $R_{w,h}$ when counting from the top.
This follows from the fact that there must be enough space to slide it directly to the right so that $l_{i}-1$ squares labeled $3$ can fit above it.
Similarly, the lowest row it can be fully contained in is the $(h-p-1+l_{i})$-th row from the top, as there must be space for $p+1-l_{i}$ squares beneath it when slid to the right.
Moreover, the square labeled $2$ can be anywhere between these rows.
It follows that in $V$, the highest this square labeled $2$ can be is the $l_{s}$-th row, and the lowest it can be is the $(h-p-1+l_{0})$-th row.
Thus, this intersection is empty if $l_{s}-l_{0}>h-p-1$.

Next, note that since $s\ge 1$, we may assume that in $V$ the square labeled $2$ is in the second rightmost column of $R_{w,h}$.
This follows from the observation that in $V$ the square labeled $2$ cannot be in the same column as the squares labeled $3$ and the fact that we can push the squares labeled $3$ to the far right of $R_{w,h}$ and ignore them, and then push the square labeled $2$ into the second rightmost column.

\emph{Case 1.)} If $n<(w-1)h+p$, then there is at least one unoccupied square in the $R_{w-1,h}$ found on the left side of $R_{w,h}$.
Thus, the square labeled $2$ can be slid up or down between the rows proving that $V$ is connected.

\emph{Case 2.)} If $n=(w-1)h+p$, then there are $(w-1)h$ squares labeled $1$ or $2$ in the $R_{w-1,h}$ found on the left side of $R_{w,h}$.
Therefore these squares cannot move.
It follows that $V$ has $h-p+l_{0}-l_{s}$ components, one for each row between the $l_{s}$-th and the $(h-p-1+l_{0})$-th row inclusive. 

\emph{Case 3.)} If $n>(w-1)h+p$, then the square labeled $2$ must be in the rightmost column of $R_{w,h}$.
Since it is in different positions relative to the squares labeled $3$ in $V_{l_{0}}$ and $ V_{l_{s}}$, the intersection is empty.
\end{proof}

It follows that as long as $n\neq (w-1)h+p$, the zeroth row of the $\tilde{E}^{2}$-page of the augmented slide spectral sequence of $SC_{n}(R_{w,h};1;p)$ is given by the reduced homology of the nerve complex of the cover $\{V_{l}\}$.
Next, we see that if  $n\neq (w-1)h+p$, this complex is contractible, so the $\tilde{E}^{2}_{1,0}$-entry of the augmented slide spectral sequence contributes nothing to $H_{1}\big(SC_{n}(R_{w,h};1;p)\big)$.

\begin{prop}\label{nerve complex is contractible}
Let $\mathcal{V}:=\{V_{l}\}_{l=1}^{p+1}$ be the cover of $SC_{n}(R_{w,h};1;p)$ used to define the slide spectral sequence.
If $w,h\ge 2$, $0\le p< h$ and $n\neq (w-1)h+p$, then the nerve complex $N(\mathcal{V})$ of this cover is homotopy equivalent to a union of points.
\end{prop}

\begin{proof}
If $n>(w-1)h+p$, then by Proposition~\ref{connected intersection of Vs n-p-1 < (w-1)h} every intersection $V_{l_{i}}\cap V_{l_{j}}$ is empty if $i\neq j$. 
Since the $V_{l}$ are homotopy equivalent to $SC_{n}(R_{w,h};p+1)$, which is connected by Proposition 
\ref{connected unordered square configuration spaces one rightmost}, it follows that $N(\mathcal{V})$ is a union of points.

If $n\le(w-1)h+p$, then Proposition~\ref{connected intersection of Vs n-p-1 < (w-1)h} proves that $N(\mathcal{V})$ is the union of simplices of the form $[l:l+1:\cdots: l+s]$, for $1\le l\le p+1-s$ where $s=\min\{h-p-1, p\}$.
This complex is contractible.
\end{proof}

At the exceptional value $n=(w-1)h+p$, the zeroth row of the $\tilde{E}^{2}$-page of the augmented slide spectral sequence of $SC_{n}(R_{w,h};1;p)$ is \emph{not} given by the reduced homology of the nerve complex of the cover $\{V_{l}\}$.
Instead, it is given by the complex formed from taking $(h-p-1)$ copies of the nerve complex and identifying the corresponding $0$-cells.
This complex  has first homology isomorphic to $\Z^{p(h-p-2)}$; see Proposition~\ref{H1 SC(1p1)}.
It follows that the $\tilde{E}^{2}_{1,0}$-entry of the augmented slide spectral sequence of $SC_{n}(R_{w,h};1;p)$ is $\Z^{p(h-p-2)}$, so it is a free summand of $H_{1}\big(SC_{n}(R_{w,h};1;p)\big)$; we define a family of $p(h-p-2)$ generators that we call exotic classes.
As we are only interested in $0<p<h$, we restrict our definition to the setting where $w\ge 2$ and $h\ge 4$, as this is necessary for $n=(w-1)h+p$ to be at most $wh-3$.

\begin{defn}\label{H1 SC(1p1) for n-p-1=(w-1)h}
Suppose $w\ge2$, $h\ge 4$, $0< p< h$ and $n=(w-1)h+p\le wh-3$.
An \emph{exotic class} in $H_{1}\big(SC_{n}(R_{w,h};1;p)\big)$ is a class of the form depicted in Figure~\ref{exotic class}, where the other squares labeled $1$ or $3$ are fixed outside of the embedded $R_{2,4}$ in $R_{w,h}$.
These classes are a basis for the image of the $\tilde{E}^{2}_{1,0}$-entry of the slide spectral sequence of $SC_{n}(R_{w,h};1;p)$ in $H_{1}\big(SC_{n}(R_{w,h};1;p)\big)$ under the $\tilde{d}^{2}$-differential.

\begin{figure}[h]
\centering
\captionsetup{width=.8\linewidth}
\includegraphics[width=8cm]{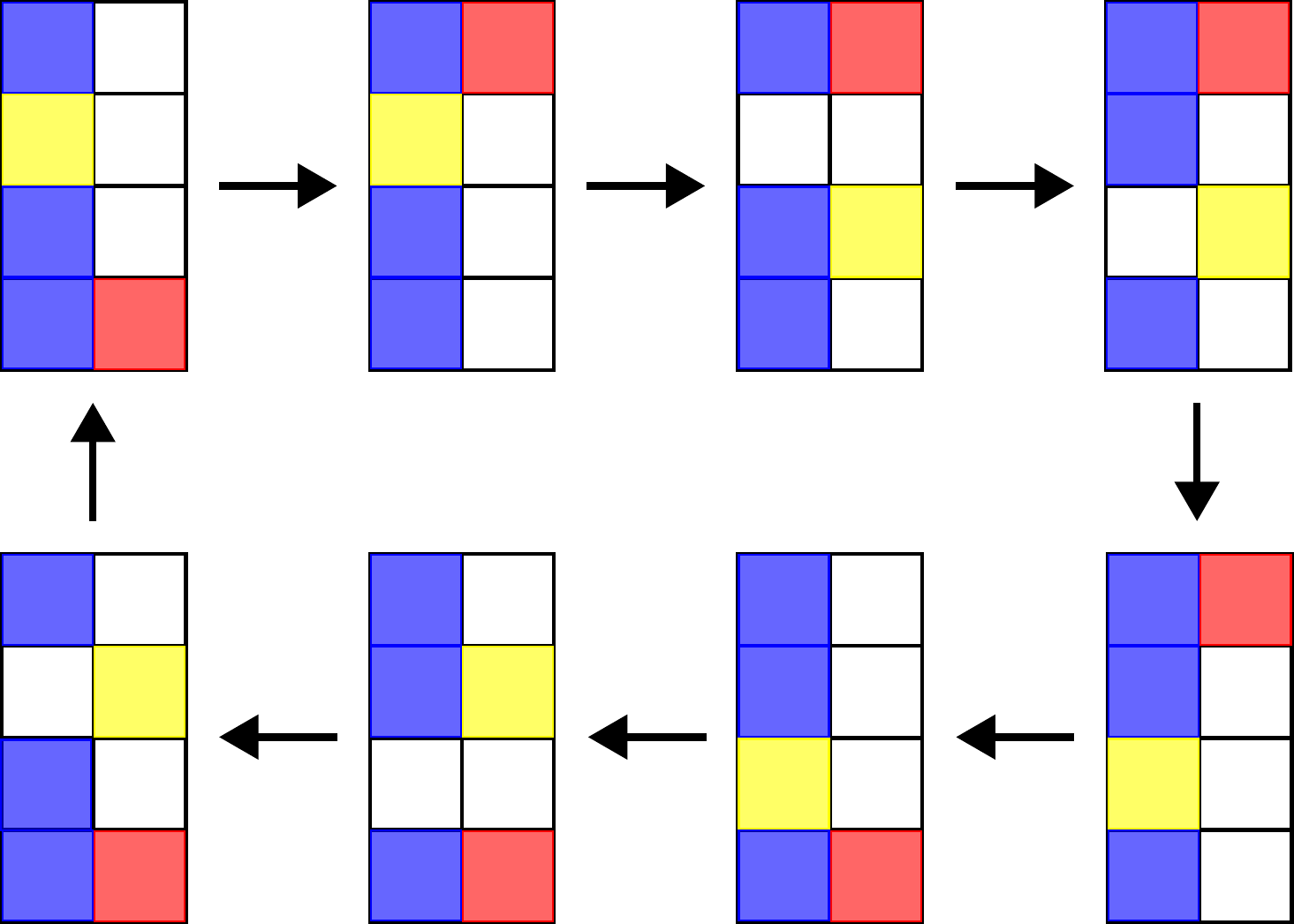}
\caption{An exotic class in $H_{1}\big(SC_{5}(R_{2,4}; 1; 1)\big)$ that arises from the $\tilde{E}^{2}_{1,0}\big[SC_{5}(R_{2,4}; 1; 1)\big]$-entry of the augmented slide spectral sequence.
}
\label{exotic class}
\end{figure}

Similarly, the image of an exotic class in $H_{1}\big(SC_{n}(R_{w,h}; 1;p)\big)$ under the $d^{1}$-differential of the augmented right Mayer--Vietoris spectral sequence of $SC_{n}(R_{w,h}; p)$ is an \emph{exotic class} in $H_{1}\big(SC_{n}(R_{w,h}; p)\big)$; see Figure~\ref{exotic class single rightmost}.
By iterating the augmented slide and right Mayer--Vietoris spectral sequence, we get \emph{exotic classes} in $H_{1}\big(SC_{n}(R_{w,h}; 1;p)\big)$ and $H_{1}\big(SC_{n}(R_{w,h};p)\big)$ for general $n$ and $p$.

\begin{figure}[h]
\centering
\captionsetup{width=.8\linewidth}
\includegraphics[width=10cm]{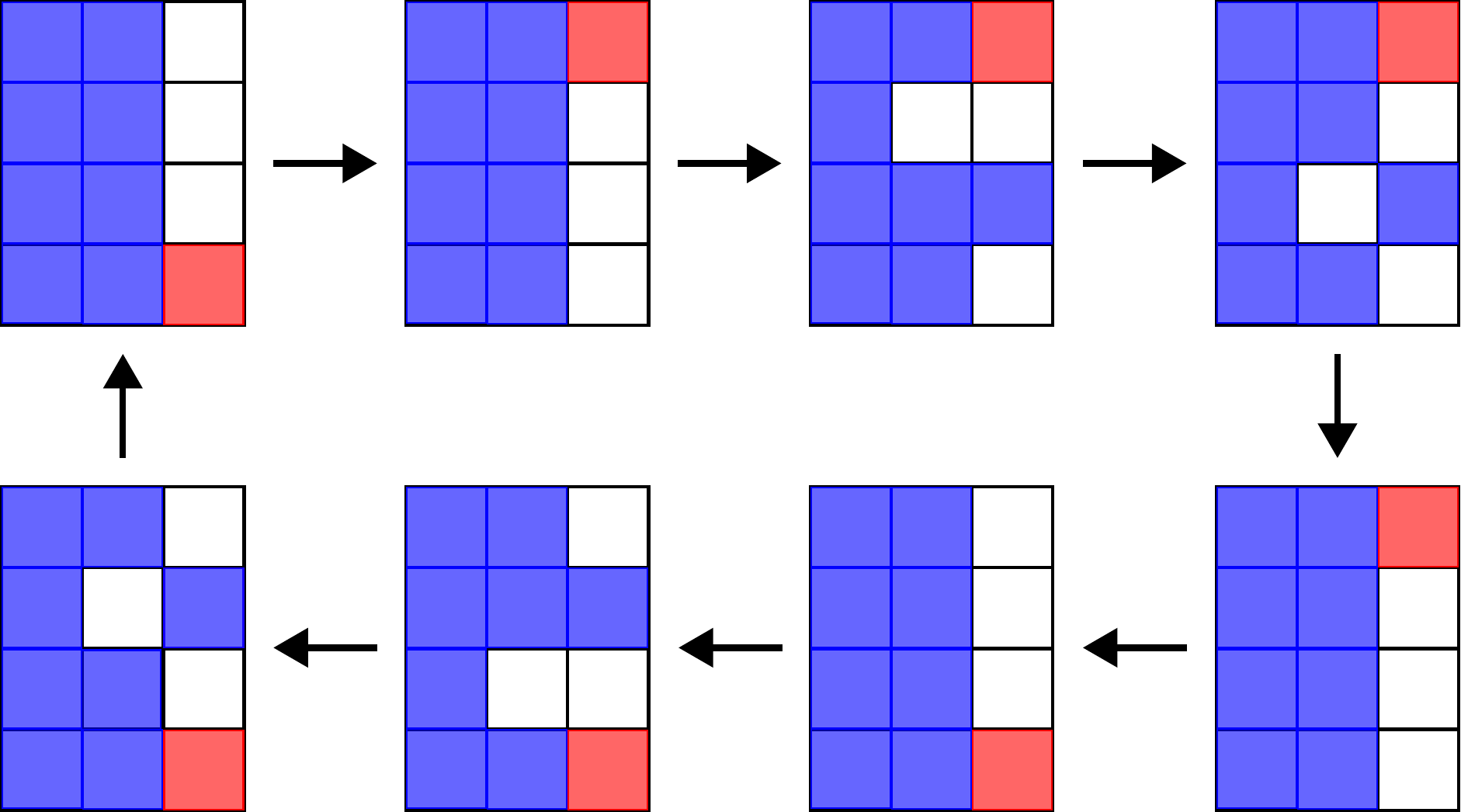}
\caption{An exotic class in $H_{1}\big(SC_{9}(R_{3,4}; 1)\big)$ that arises from applying the $d^{1}$-differential of the augmented right Mayer--Vietoris spectral sequence to $E^{1}_{0,1}\big[SC_{9}(R_{3,4}; 1)\big]\cong H_{1}\big(SC_{9}(R_{3,4};1;1)\big)$.}
\label{exotic class single rightmost}
\end{figure}
\end{defn}

Next, we count the number of exotic classes in $H_{1}\big(SC_{n}(R_{w,h};1;p)\big)$ and use this to prove Proposition~\ref{H1 SC(1p1)}, giving a description of $H_{1}\big(SC_{n}(R_{w,h};1;p)\big)$ that plays a vital role in the rest of this section.

\begin{proof}[Proof of Proposition~\ref{H1 SC(1p1)}] 
By Proposition~\ref{Vl is SCp2}, in the augmented slide spectral sequence of $SC_{n}(R_{w,h};1;p)$ the $\tilde{d}^{1}$-differential from $\tilde{E}^{1}_{0,1}$ to $\tilde{E}^{1}_{-1,1}$ is a map
\[
H_{1}\big(SC_{n}(R_{w,h};p+1)\big)^{\oplus(p+1)}\to H_{1}\big(SC_{n}(R_{w,h};1;p)\big).
\]
The only other potentially non-trivial entry mapping into $\tilde{E}^{r}_{-1,1}$ under the $\tilde{d}^{r}$-differential is the $\tilde{E}^{r}_{1,0}$-entry.
Since the abutment of this spectral sequence is trivial, it suffices to determine that entry.

If $n=(w-1)h+p$ and $h\ge 4$, note that the sets $V_{l}$ used to define the augmented slide spectral sequence are connected by Propositions~\ref{connected unordered square configuration spaces one rightmost} and~\ref{Vl is SCp2}.
Additionally, if $l_{s}-l_{0}\le h-p-1$, the intersections $V_{l_{0}}\cap\cdots\cap V_{l_{s}}$ are nonempty via Proposition~\ref{connected intersection of Vs n-p-1 < (w-1)h}, and this intersection has $h-p+l_{0}-l_{s}$ components, one for each possible location of the square labeled $2$ in the second rightmost column of $R_{w,h}$.
Noting that it suffices to consider $l_{s}-l_{0}=1$, it follows that $\tilde{E}^{2}_{s,0}\big[SC_{n}(R_{w,h}; 1; p)\big]$ is the $s$-th homology group of the complex that arises from taking $h-p-1$ copies of the nerve complex $N(\mathcal{V})$ described in the $n<(w-1)h+p$ case of Proposition~\ref{nerve complex is contractible} and for each $i$, identifying all the zero cells labeled by $V_{l_{i}}$.
By the proof of Proposition~\ref{nerve complex is contractible}, this is homotopy equivalent to a wedge of circles corresponding to pairs $(l, l+1)$, where for each pair there are $h-p-2$ loops.
By standard spectral sequence techniques, one can check that the images of these loops  under the $\tilde{d}^{2}$-differential are the exotic classes depicted Figure~\ref{exotic class}.
Since this space is homotopy equivalent to a wedge of circles and the $\tilde{E}^{r}_{1,0}\big[SC_{n}(R_{w,h}; 1; p)\big]$-entry must vanish on subsequent pages, this gives a free $\Z$-module of rank $p(h-p-2)$ generated by exotic classes.
As this module is free, $H_{1}\big(SC_{n}(R_{w,h}; 1;p)\big)$ 
is canonically a direct sum of a quotient of $H_{1}\big(SC_{n}(R_{w,h};p+1)\big)^{\oplus(p+1)}$ and the free $\Z$-module generated by $p(h-p-2)$ exotic classes.

If $n\neq (w-1)h+p$ or $h<4$, then Proposition~\ref{nerve complex is contractible} proves that $\tilde{E}^{2}_{s,0}\big[SC_{n}(R_{w,h}; 1; p)\big]=0$.
Therefore, $H_{1}\big(SC_{n}(R_{w,h}; 1;p)\big)$ is canonically a quotient of $H_{1}\big(SC_{n}(R_{w,h};p+1)\big)^{\oplus(p+1)}$.
\end{proof}

In particular, this implies that every class in $H_{1}\big(SC_{n}(R_{w,h};1;p)\big)$ is a sum of exotic classes and classes in $H_{1}\big(SC_{n}(R_{w,h};p+1)\big)^{\oplus(p+1)}\cong H_{1}\big(SC_{n}(R_{w,h};0;p+1)\big)^{\oplus(p+1)}$, where in the $i$-th summand the $i$-th square labeled $3$ is relabeled $2$.
In the next subsection, we use this fact to prove that $H_{1}\big(SC_{n}(R_{w,h})\big)$ is generated by orbiting pair square classes as long as $w,h\ge 3$ and $n\le wh-3$.

\subsection{Generating $H_{1}\big(SC_{n}(R_{w,h})\big)$ by orbiting pair square classes}
\label{generation by orbiting pair subsection}

This subsection is devoted to the proof of statement (G) from the beginning of the section:

\begin{lem}\label{generated by orbiting pairs}
If $w,h\ge 3$, and $n\le wh-3$, then $H_{1}\big(SC_{n}(R_{w,h})\big)$ is generated by orbiting pair square classes.
\end{lem}

The proof of this has four parts, which we take in the order they are used.
\begin{itemize}
  \item Lemma~\ref{intermediate generated by orbiting square classes} is the descent lemma promised in Subsection~\ref{descent lemma subsection}.
  It lowers the number of rightmost squares by one, propagating generation from $H_{1}\big(SC_{n}(R_{w,h};p+1)\big)$ to $H_{1}\big(SC_{n}(R_{w,h};p)\big)$. 
  
  \item Proposition~\ref{3 by 3 rectangle generated by orbiting pairs} supplies the base case from the computations of Section~\ref{h=2 section}.
  
  \item Lemma~\ref{generated by orbiting square classes} runs an induction on $w$ to obtain the generation of $H_{1}\big(SC_{n}(R_{w,h};p)\big)$ by orbiting pair square classes and  exotic classes for all $w,h \ge 3$.
  
  \item Proposition~\ref{d1 on exotic classes is 0} removes the exotic classes.
\end{itemize}
Lemma~\ref{generated by orbiting pairs} is direct consequence of these results.
Note that the exotic classes are unavoidable in the inductive step, and only Proposition~\ref{d1 on exotic classes is 0} disposes of them.
We begin by proving our descent lemma, which will play a crucial role in the inductive step of our proof of Lemma~\ref{generated by orbiting square classes}.

\begin{lem}\label{intermediate generated by orbiting square classes}
Assume $w,h\ge 3$, $1\le p+1\le h$, and $n\le wh-3$.
If $H_{1}\big(SC_{n}(R_{w,h}; p+1)\big)$ is generated by orbiting pair square classes and exotic classes, then so is $H_{1}\big(SC_{n}(R_{w,h}; p)\big)$.
\end{lem}

\begin{proof}
Consider the augmented right Mayer--Vietoris spectral sequence for $SC_{n}(R_{w,h}; p)$.
There are two cases we consider: $n=p+2$ and $n\neq p+2$.

If $n=p+2$, then $n-p-1=1\neq (w-1)h$.
It follows that $E^{1}_{0,1}\big[SC_{n}(R_{w,h}; p)\big]=H_{1}\big(SC_{n}(R_{w,h}; 1;p)\big)=0$ as $H_{1}\big(SC_{n}(R_{w,h}; 1;p)\big)$ is a quotient of $H_{1}\big(SC_{n}(R_{w,h}; p+1)\big)^{\oplus(p+1)}=0$ by Proposition~\ref{H1 SC(1p1)}.
This follows from the fact that there is a straight-line homotopy $SC_{n}(R_{w,h}; p+1)\simeq SC_{1}(R_{w-1,h})\simeq pt$ that pushes the squares labeled $2$ into the rightmost column of $R_{w,h}$ and the square labeled $1$ out of that column.
Therefore, $H_{1}\big(SC_{n}(R_{w,h}; p)\big)$ is the image of $E^{2}_{1,0}\big[SC_{n}(R_{w,h}; p)\big]$ under the $d^{2}$-differential.
This entry is $\Z$ as by Proposition~\ref{connected unordered square configuration spaces two rightmost}, the $E^{1}_{1,0}\big[SC_{n}(R_{w,h}; p)\big]$-entry is $\Z$, the $d^{1}$-differential from this entry is $0$ as $SC_{n}(R_{w,h}; 1;p)$ is connected by Proposition~\ref{connected unordered square configuration spaces two rightmost}, and the $E^{1}_{2,0}\big[SC_{n}(R_{w,h}; p)\big]$-entry is $0$.
By Proposition~\ref{spectral sequence d2 differential }, the image of this differential is an orbiting pair square class, so $H_{1}\big(SC_{n}(R_{w,h}; p)\big)$ is generated by an orbiting pair square class.

If $n\neq p+2$, a similar argument works, but, in this case, $E^{2}_{1,0}\big[SC_{n}(R_{w,h}; p)\big]=0$ by Proposition~\ref{E201 entry of augmeneted right Mayer Vietoris ss}.
It follows that the $d^{1}$-differential from $E^{1}_{0,1}\big[SC_{n}(R_{w,h}; p)\big]=H_{1}\big(SC_{n}(R_{w,h}; 1;p)\big)$ to $E^{1}_{-1,1}\big[SC_{n}(R_{w,h}; p)\big]=H_{1}\big(SC_{n}(R_{w,h}; p)\big)$ is surjective as the augmented right Mayer--Vietoris spectral sequence of $SC_{n}(R_{w,h}; p)$ has trivial abutment by Proposition~\ref{augmented right MVSS has trivial abutment}.
Therefore, $H_{1}\big(SC_{n}(R_{w,h}; p)\big)$ is a quotient of $H_{1}\big(SC_{n}(R_{w,h}; 1;p)\big)$.
Note $H_{1}\big(SC_{n}(R_{w,h}; 1;p)\big)$ is a direct sum of a quotient of $H_{1}\big(SC_{n}(R_{w,h}; p+1)\big)^{\oplus(p+1)}$ and a free summand spanned by exotic classes by Proposition~\ref{H1 SC(1p1)}.
As $H_{1}\big(SC_{n}(R_{w,h}; p+1)\big)$ is generated by orbiting pair square classes and exotic classes by assumption, $H_{1}\big(SC_{n}(R_{w,h}; 1;p)\big)$, hence $H_{1}\big(SC_{n}(R_{w,h}; p)\big)$, must be as well.
\end{proof}

Next, we prove a proposition that simplifies the handling of the base case of the induction used in the proof of Lemma~\ref{generated by orbiting square classes}.

\begin{prop}\label{3 by 3 rectangle generated by orbiting pairs}
If $n\le 2h-3$, then $H_{1}\big(SC_{n}(R_{2,h})\big)$ is generated by orbiting pair square classes.
\end{prop}

\begin{proof}
Lemmas~\ref{first homology height 2 for 2w-3} and~\ref{first homology height 2 for n le 2w-4} tell us that $H_{1}\big(SC_{n}(R_{2,h})\big)$ is generated by orbiting pair square classes for $n\le 2h-3$ via the equivalence of $H_{1}\big(SC_{n}(R_{2,h})\big)$ and $H_{1}\big(SC_{n}(R_{h,2})\big)$.
\end{proof}

Now we are ready to show that $H_{1}\big(SC_{n}(R_{w,h}; p)\big)$ is generated by orbiting pair square classes and exotic classes.

\begin{lem}\label{generated by orbiting square classes}
If $w,h\ge 3$, $0\le p\le h$, and $n\le wh-3$, then $H_{1}\big(SC_{n}(R_{w,h}; p)\big)$ is generated by orbiting pair square classes and exotic classes.
\end{lem}

\begin{proof}
We proceed by induction on $w$.
Fixing $h$, Proposition~\ref{3 by 3 rectangle generated by orbiting pairs} tells us that $H_{1}\big(SC_{n}(R_{2,h})\big)$ is generated by orbiting pair square classes for $n\le 2h-3$.
Since $SC_{n}(R_{3,h};h)\simeq SC_{n-h}(R_{2,h})$, it follows that if $n\le 3h-3$, then $H_{1}\big(SC_{n}(R_{3,h};h)\big)$ is generated by orbiting pair square classes.
By repeated application of Lemma~\ref{intermediate generated by orbiting square classes}, we see that $H_{1}\big(SC_{n}(R_{3,h}; p)\big)$ is generated by orbiting pair square and exotic classes for $0\le p\le h$ as long as $n\le 3h-3$.
Thus, the base case of our induction holds.

Next, assume that the lemma holds for all $3\le w\le \omega$ for some $\omega$.
To see that it holds for $w=\omega+1$, note that the inductive hypothesis and the fact that $SC_{n}(R_{\omega+1,h};h)\simeq SC_{n-h}(R_{\omega,h})$ prove that $H_{1}\big(SC_{n}(R_{\omega+1,h}; h)\big)$ is generated by orbiting pair square classes and exotic classes as long as $n\le (\omega+1) h-3$.
Again, repeated application of Lemma~\ref{intermediate generated by orbiting square classes}, shows that $H_{1}\big(SC_{n}(R_{\omega+1,h}; p)\big)$ is generated by orbiting pair square classes and exotic classes for $0\le p\le h$ as long as $n\le (\omega+1) h-3$. 
This completes the inductive step on $w$ for general $h\ge 3$, finishing the proof of the lemma.
\end{proof}

Finally, we prove that exotic classes are trivial in $H_{1}\big(SC_{n}(R_{w,h})\big)$ for $w,h\ge 3$ and $n\le wh-3$.

\begin{prop}\label{d1 on exotic classes is 0}
If $w,h\ge 3$ and $n\le wh-3$, then exotic classes are trivial in $H_{1}\big(SC_{n}(R_{w,h})\big)$.
\end{prop}

\begin{proof}
Since there are no exotic classes when $h=3$, we focus on the case $h\ge 4$.

By the definition of the augmented right Mayer--Vietoris spectral sequence for $SC_{n}(R_{w,h})$, we have that $E^{1}_{0,1}\big[SC_{n}(R_{w,h})\big]\cong H_{1}\big(SC_{n}(R_{w,h}; 1)\big)$.
Consider an exotic class in $H_{1}\big(SC_{n}(R_{w,h}; 1)\big)$ as in Figure~\ref{exotic class single rightmost}.
We will see that the image of this class under the $d^{1}$-differential is trivial.

By definition, the $d^{1}$-differential of the augmented right Mayer--Vietoris spectral sequence for $SC_{n}(R_{w,h})$ on the $E^{1}_{0,1}\big[SC_{n}(R_{w,h})\big]$-entry has the effect of changing the label of the square labeled $2$ to $1$, i.e., removing all labels.
As can be seen in Figure~\ref{trivial exotic class}, this makes an exotic class trivial in homology.
\begin{figure}[h]
\centering
\captionsetup{width=.8\linewidth}
\includegraphics[width=10cm]{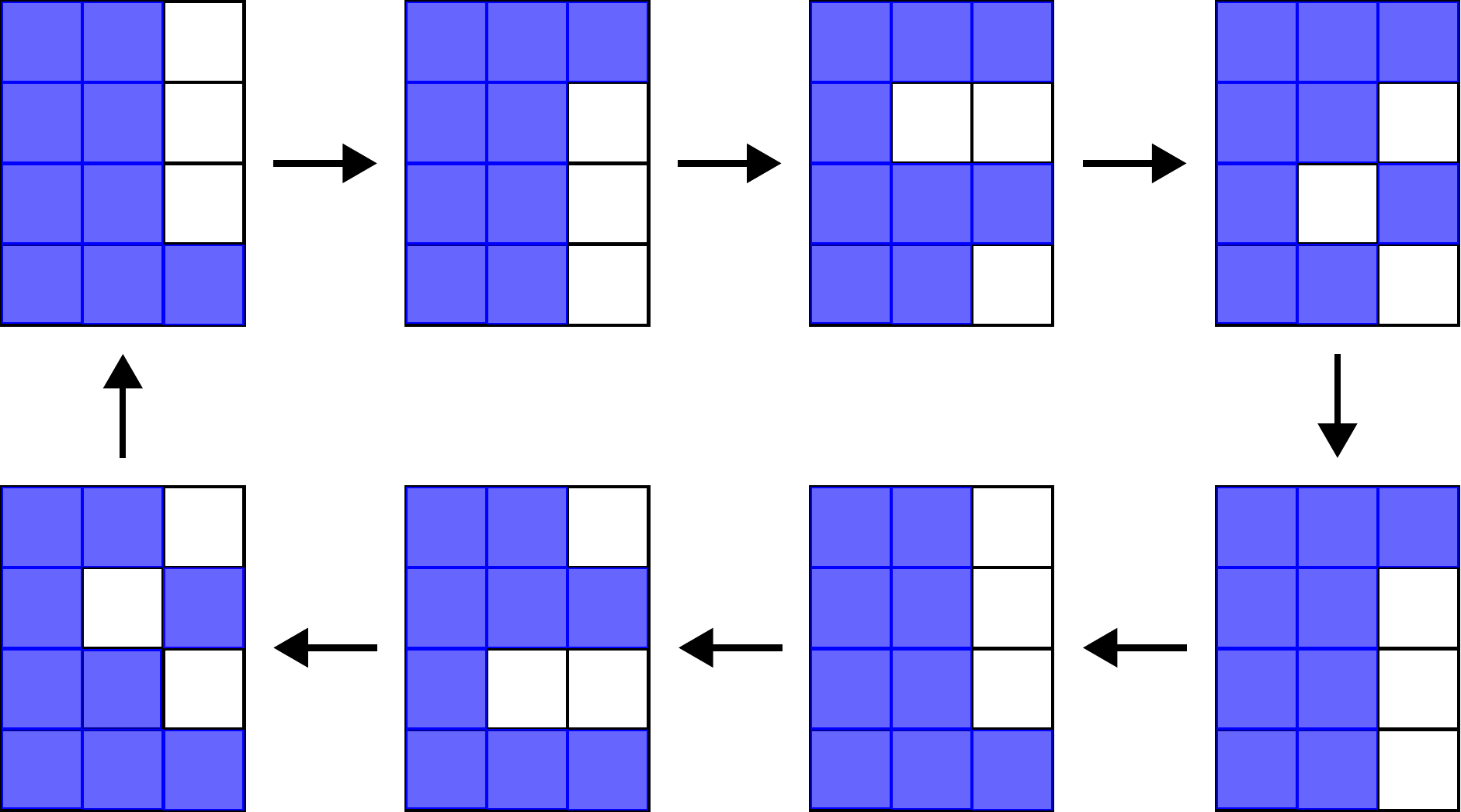}
\caption{An exotic class in $H_{1}\big(SC_{9}(R_{3,4})\big)$ that arises from applying the $d^{1}$-differential of the augmented right Mayer--Vietoris spectral sequence to $E^{1}_{0,1}\big[SC_{9}(R_{3,4})\big]\cong H_{1}\big(SC_{9}(R_{3,4};1)\big)$; see Figure~\ref{exotic class single rightmost}.
Observe that the top, left, and bottom squares can be rearranged, so this class is trivial, as the two squares moving at the middle right do not wind about each other.}
\label{trivial exotic class}
\end{figure}
\end{proof}

Finally, we use these intermediate results to prove Lemma~\ref{generated by orbiting pairs}.

\begin{proof}[Proof of Lemma~\ref{generated by orbiting pairs}] 
Lemma~\ref{generated by orbiting square classes} proves that $H_{1}\big(SC_{n}(R_{w,h};p)\big)$ is generated by orbiting pair square classes and exotic classes for all $w,h\ge 3$, $0\le p\le h$, and $n\le wh-3$.
In the case $p=0$, Proposition~\ref{d1 on exotic classes is 0} proves that exotic classes are trivial, implying that $H_{1}\big(SC_{n}(R_{w,h})\big)$ is generated by orbiting pair square classes as long as $w,h\ge 3$ and $n\le wh-3$.
\end{proof}

In the next subsection, we prove that if $w,h\ge 3$ and $n\le wh-3$, all orbiting pair square classes in $H_{1}\big(SC_{n}(R_{w,h})\big)$ are homologous, implying that this group is cyclic.

\subsection{Orbiting pair square classes are homologous}\label{orbiting pair square classes are homologous subsection}

We now prove statement (H), that is, we prove the following lemma:

\begin{lem}\label{location of square classes does not matter}
If $w,h\ge 3$ and $n\le wh-3$, then any two orbiting pair square classes in $ H_{1}\big(SC_{n}(R_{w,h})\big)$ are homologous.
\end{lem}

Unlike Statement (G), the proof of this lemma does not involve a spectral sequence; rather, it is based on the observation that a pair of squares orbiting in a $2 \times 2$ rectangle and a pair of squares orbiting in a $T$-tetromino are homologous inside a common pentomino.

\begin{defn}
If $T$ denotes the $T$-shaped tetromino, an \emph{orbiting pair $T$ class} $\tau\in H_{1}\big(SC_{2}(T)\big)$ is the class depicted in Figure~\ref{Orbiting Pair Classes}.

\begin{figure}[h]
\centering
\captionsetup{width=.8\linewidth}
\includegraphics[width=5cm]{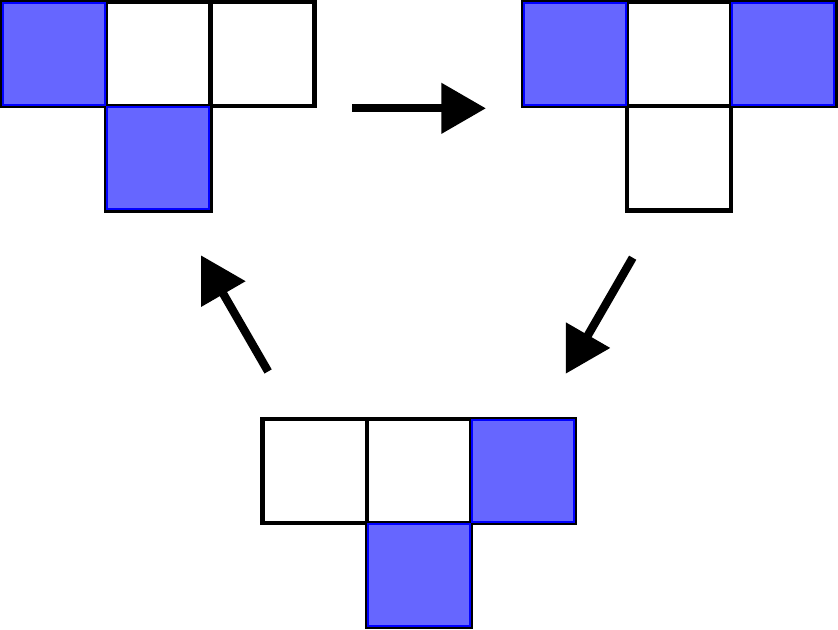}
\caption{An orbiting pair $T$ class.
}
\label{Orbiting Pair Classes}
\end{figure}

Embedding $T$ into $R_{w,h}$ yields \emph{orbiting pair $T$ classes} in $H_{1}\big(SC_{n}(R_{w,h})\big)$.
Inside the embedded $T$ two  squares orbit each other.
The remaining $n-2$ singleton squares are fixed outside of the embedded region. 
\end{defn}

We will prove that certain orbiting pair square and $T$ classes are homologous in $H_{1}\big(SC_{n}(R_{w,h})\big)$, and we will need the following fact about configurations of two squares in the $P$ pentomino.

\begin{prop}\label{2 squares in P is homotopy equivalent to 2 squares in R22}
Let $P$ be the area 5 pentomino depicted in Figure~\ref{PPentomino}.
Then $SC_{2}(P)\simeq SC_{2}(R_{2,2})$.

\begin{figure}[h]
\centering
\captionsetup{width=.8\linewidth}
\includegraphics[width=1.25cm]{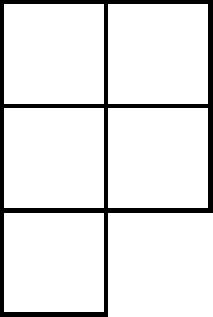}
\caption{The "$P$" pentomino.
}
\label{PPentomino}
\end{figure}
\end{prop}

\begin{proof}
Consider the $2$ by $2$ rectangle $R_{2,2}$ inside $P$.
There is a deformation retraction of $SC_{2}(P)$ onto $SC_{2}(R_{2,2})$ via the straight-line homotopy that pushes the two squares up into the $R_{2,2}\subset P$.
\end{proof}

By the previous proposition and the well-known fact that $H_{1}\big(C_{2}(\R^{2})\big)\cong \Z$ is generated by the class of an orbiting pair of points, we have that $H_{1}\big(SC_{2}(P)\big)\cong\Z$ is generated by the pair of squares orbiting each other in either an embedded $R_{2,2}$ or an embedded $T$, i.e., in $H_{1}\big(SC_{2}(P)\big)$ the orbiting pair square and $T$ classes are homologous. 
We formalize this in the following proposition.

\begin{prop}\label{orbiting square class is homologous to orbiting t class}
Let $\sigma$ and $\tau$ denote the orbiting pair square and $T$ classes in $H_{1}\big(SC_{2}(P)\big)$ that arise from the natural embeddings of $R_{2,2}$ and $T$ in $P$.
Then $\sigma$ and $\tau$ are homologous in $H_{1}\big(SC_{2}(P)\big)$.
\end{prop}

\begin{proof}
By Proposition~\ref{2 squares in P is homotopy equivalent to 2 squares in R22} 
\[
H_{1}\big(SC_{2}(P)\big)\cong H_{1}\big(SC_{2}(R_{2,2})\big)\cong H_{1}\big(C_{2}(\R^{2})\big)\cong \Z.
\]
Since $\sigma$ and $\tau$ generate $H_{1}\big(SC_{2}(P)\big)$ and have the same orientation, they must be homologous.
\end{proof}

We use this fact to prove that as long as $w,h\ge 3$ and $n\le wh-3$, all orbiting pair square classes in $H_{1}\big(SC_{n}(R_{w,h})\big)$ are homologous:

\begin{proof}[Proof of Lemma~\ref{location of square classes does not matter}]
Let $\sigma_{1}$ and $\sigma_{2}$ be orbiting pair square classes in $H_{1}\big(SC_{n}(R_{w,h})\big)$.
Since the singletons can be moved to any configuration around the orbiting pair, it suffices to show that if $\sigma_{1}$ and $\sigma_{2}$ are adjacent as in the top of Figure~\ref{Adjacent Orbiting Pair Classes}, then they are homologous.

\begin{figure}[h]
\centering
\captionsetup{width=.8\linewidth}
\includegraphics[width=16cm]{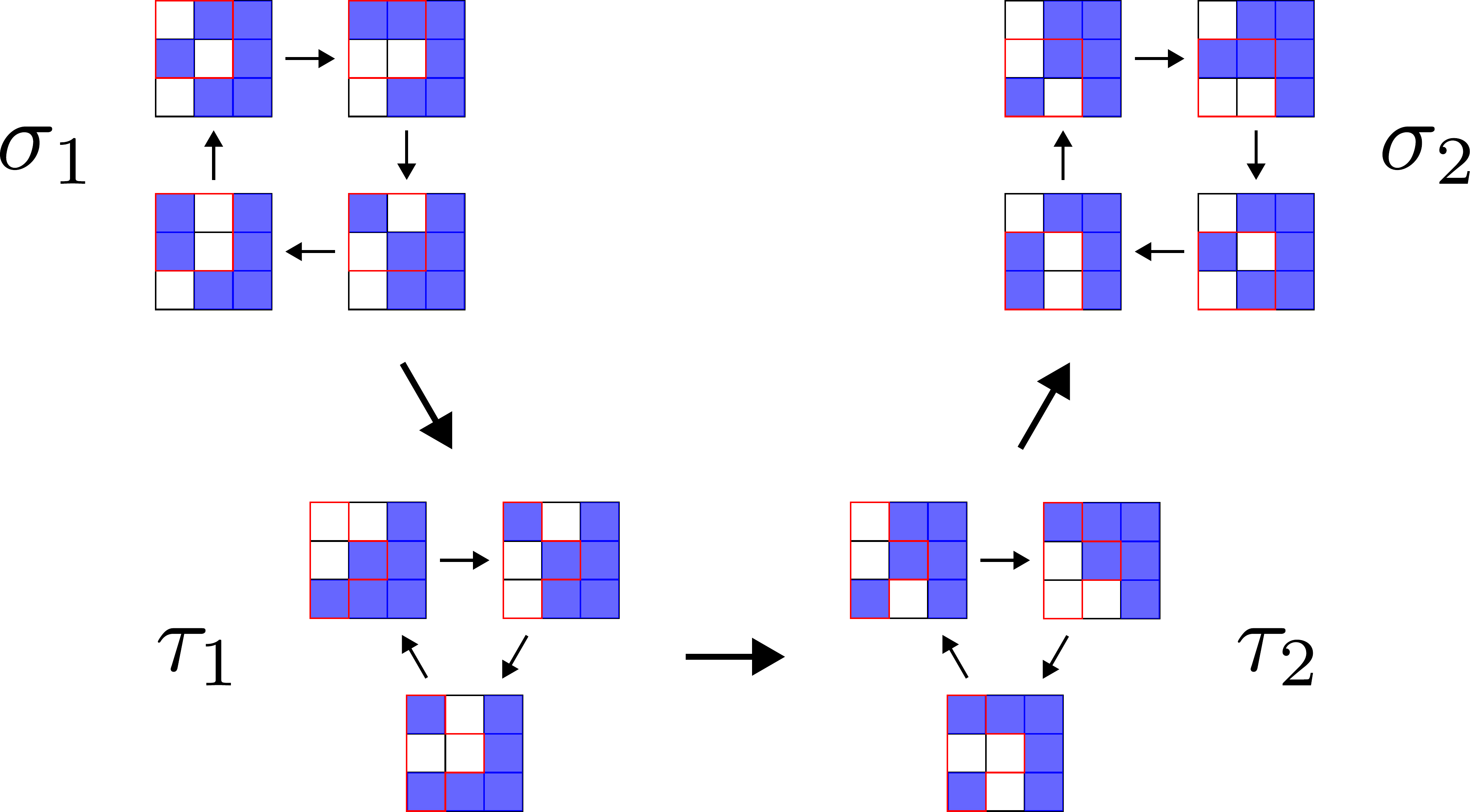}
\caption{The equivalence in homology between two adjacent orbiting pair square classes $\sigma_{1}$ and $\sigma_{2}$.
The large diagonal arrows are given by Proposition~\ref{orbiting square class is homologous to orbiting t class}, and the large horizontal arrow is a homotopy equivalence between the resulting orbiting pair $T$ classes that arises by sliding the other squares about $R_{3,3}$ while avoiding the embedded $T$ tetromino.
}
\label{Adjacent Orbiting Pair Classes}
\end{figure}

Given two adjacent orbiting pair square classes $\sigma_{1}$ and $\sigma_{2}$ in $H_{1}\big(SC_{n}(R_{w,h})\big)$, note that by our assumptions on $n$, we know that $\sigma_{1}$ can be realized such that there is at least one unoccupied square in the embedded $R_{2,2}$ used to realize $\sigma_{2}$ not lying in the embedded $R_{2,2}$ used to realize $\sigma_{1}$.
By Proposition~\ref{orbiting square class is homologous to orbiting t class}, it follows that $\sigma_{1}$ is homologous to an orbiting pair $T$ class $\tau_{1}$ by extending the embedded $R_{2,2}$ to an embedded $P$ pentomino by using the unoccupied square in the $R_{2,2}$ corresponding to $\sigma_{2}$.
Since $w,h\ge 3$, the singleton squares can be slid around $R_{w,h}$ avoiding the embedded $T$ tetromino used to realize $\tau_{1}$ to get a homologous class $\tau_{2}$ in which none of the singleton squares are in the $R_{2,2}$ used to realize $\sigma_{2}$.
Again, applying Proposition~\ref{orbiting square class is homologous to orbiting t class}, we see that $\tau_{2}$ and $\sigma_{2}$ are homologous, completing the proof.
See Figure~\ref{Adjacent Orbiting Pair Classes} for a visualization of this equivalence.
\end{proof}

\subsection{Proof of Theorem~\ref{space homological stability unordered}}\label{Proof of Theorem space homological stability unordered subsection}

We conclude this section with a proof of Theorem~\ref{space homological stability unordered}, showing that a single orbiting pair square class generates $H_{1}\big(SC_{n}(R_{w,h})\big)$, which is free, as long as $w,h\ge 3$ and $n\le wh-3$.

\begin{proof}[Proof of Theorem~\ref{space homological stability unordered}]
By Lemma~\ref{generated by orbiting pairs} we have that $H_{1}\big(SC_{n}(R_{w,h})\big)$ is generated by orbiting pair square classes if $w, h\ge 3$ and $n\le wh-3$. These classes are all homologous by Lemma~\ref{location of square classes does not matter}, so $H_{1}\big(SC_{n}(R_{w,h})\big)$ is cyclic.
Since the induced map
\[
H_{1}\big(SC_{n}(R_{w,h})\big)\to H_{1}\big(C_{n}(\R^{2})\big)\cong \Z
\]
sends an orbiting pair square class to a generator, this map is a surjection, implying that the inclusion of $SC_{n}(R_{w,h})$ into $C_{n}(\R^{2})$ induces an isomorphism on first homology for $w,h\ge 3$ and $n\le wh-3$.
\end{proof}

Having completed the proofs of our main theorems, we finish the paper with several auxiliary calculations and a collection of conjectures about the structure of $SBr_{n}(R_{w,h})$ and the higher homology groups of $SC_{n}(R_{w,h})$.

\section{Further calculations and conjectures}\label{calcs and conjs}

In the first subsection of this section we compute several of the higher differentials of the augmented right Mayer--Vietoris spectral sequence of $SC_{n}(R_{n,n})$.
This is done by noting that this spectral sequence is equivalent to Kupers and Miller's augmented arc resolution spectral sequence of $C_{n}(\R^{2})$ \cite{kupers2018n}.
In the next subsection, we present the $\Z/2\Z$ and rational Betti numbers of $SC_{n}(R_{w,h})$ for small $w$, $h$, and $n$.
Finally, in the last subsection, we use this information to make conjectures about the higher homology groups of $SC_{n}(R_{w,h})$ and the structure of the square-section braid groups $SBr_{n}(R_{w,h})$.

\subsection{Higher differentials of the augmented right Mayer--Vietoris spectral sequence}\label{higher differential subsection}

In this subsection, we compute the $d^{r}$-differentials from the $E^{r}_{n-1, 0}$-entry of the augmented right Mayer--Vietoris spectral sequence of $SC_{n}(R_{n,n})$ with $\Z/2\Z$-coefficients.
We do this by equating this spectral sequence with the augmented arc resolution spectral sequence of $C_{n}(\R^{2})$ defined by Kupers and Miller in \cite{kupers2018n}.
Namely, we prove

\begin{lem}\label{higher differentials}
Let $n=m2^{s}$, where $m$ is odd, and let $A_{n}$ be the generator of $H_{0}\Big(Arc_{n-1}\big(C_{n}(\R^{2})\big);\Z/2\Z\Big)$, the $(n-1,0)$-entry of the first page of the augmented arc resolution spectral sequence of $C_{n}(\R^{2})$ with $\Z/2\Z$-coefficients.
Then
\[
d^{q}(A_{n})=0\qquad\text{for}\quad q<2^{s},
\]
and
\[
d^{2^{s}}(A_{n})=[P_{s}\otimes A_{(m-1)2^{s}}].
\]
where the square bracket denotes the class taken in homology.
\end{lem}

We use $\Z/2\Z$-coefficients as the $\Z/2\Z$-homology of $C_{n}(\R^{2})$ is significantly richer than the rational homology, given that 
\[
H_{*}\big(C_{\bullet}(\R^{2});\Z/2\Z\big)\cong \Z/2\Z[x_{0}, x_{1},\dots, ]
\]
is the Fuks algebra, where $x_{i}$ has homological degree $2^{i}-1$ and weight $2^{i}$.
Moreover, in the following subsection, we use $\Z/2\Z$-coefficients in our computer calculations.

We begin by recalling the definition of the spaces that come together to form the semisimplicial space used to define the augmented arc resolution spectral sequence; for more on these spaces and this spectral sequence see \cite{kupers2018n, miller2019higher}.

\begin{defn}
Let $M$ be the interior of a (not necessarily compact) smooth manifold $\overline{M}$ with nonempty boundary $\partial M$.
Fix an embedding $\gamma:[0,1]\to \partial M$.
Let
\[
\text{Arc}_{p}\big(F_{n}(M)\big)\subset F_{n}(M)\times \text{Emb}\big(\sqcup_{p+1}[0,1], \overline{M}\big)
\]
be the subspace of points and arcs $(x_{1}, \dots, x_{n}; \alpha_{0},\dots, \alpha_{p})$ satisfying the following conditions:
\begin{itemize}
    \item $\alpha_{i}(0)\in \gamma\big([0,1]\big)$,
    
    \item $\alpha_{i}(1)\in \{x_{1},\dots, x_{n}\}$,
    
    \item $\alpha_{i}(t)\notin \partial M\cup \{x_{1}, \dots, x_{n}\}$ for $t\in (0,1)$, and

    \item $\gamma^{-1}\big(\alpha_{i}(0)\big)>\gamma^{-1}\big(\alpha_{j}(0)\big)$ whenever $i>j$.
\end{itemize}
Let $\text{Arc}_{p}\big(C_{n}(M)\big)$ denote the quotient of $\text{Arc}_{p}\big(F_{n}(M)\big)$ by the natural $S_{n}$-action.
\end{defn}

Letting $p$ vary, we get a semi-simplicial space analogous to the right Mayer--Vietoris complex of Definition~\ref{arc complex defn}

\begin{defn}
The $\text{Arc}_{p}\big(C_{n}(M)\big)$ come together to form a semisimplicial complex $\text{Arc}_{\bullet}\big(C_{n}(M)\big)$, with simplicial differentials $d_{h}$ defined by taking the alternating sum of the face maps $d_{l}:\text{Arc}_{p}\big(C_{n}(M)\big)\to \text{Arc}_{p-1}\big(C_{n}(M)\big)$ that forget the $l$-th arc; see Figure~\ref{ArcResHorDiff}.
As $\text{Arc}_{-1}\big(C_{n}(M)\big)\simeq C_{n}(M)$, the complex $\text{Arc}_{\bullet}\big(C_{n}(M)\big)$ is the \emph{unordered arc resolution of $C_{n}(M)$}.
The resulting augmented geometric realization spectral sequence is the \emph{augmented arc resolution spectral sequence of $C_{n}(M)$}.
We write $E^{r}_{p,q}[M](n)$ for the augmented arc resolution spectral sequence of $C_{n}(M)$, and $E^{r}_{p,q}(n)$ if $M=\R^{2}$.
\begin{figure}[H]
    \centering
    \includegraphics[width=10.625cm]{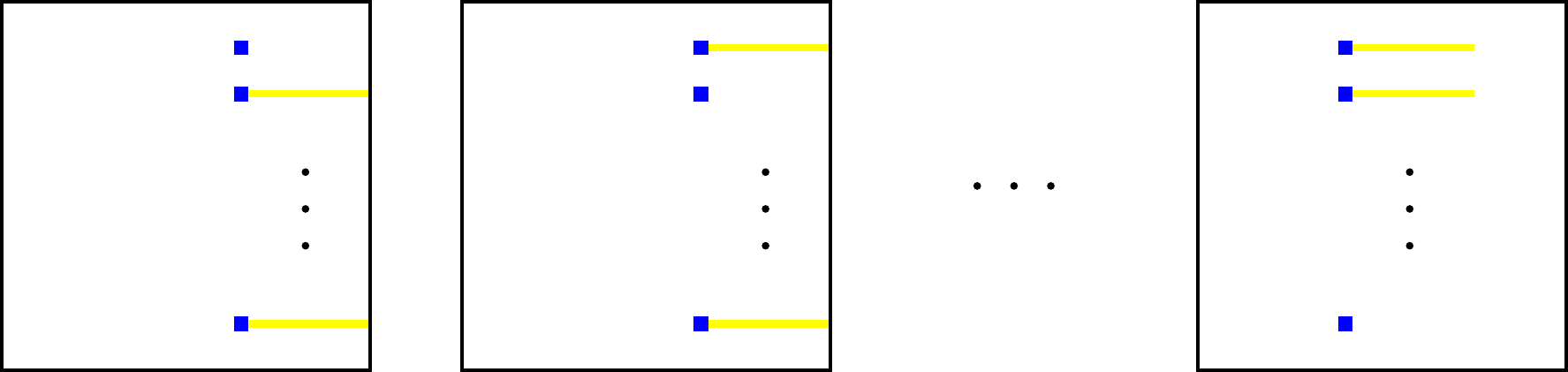}
    \caption{$n$ elements in the image of $d_{h}(A_{n})$, where $A_{n}\in Arc_{n-1}\big(C_{n}(\R^{2})\big)$.
    Viewing $A_{n}$ as a class in $H_{0}\Big(Arc_{n-1}\big(C_{n}(\R^{2})\big)\Big)$, this is its image under the $d^{1}$-differential.}
    \label{ArcResHorDiff}
\end{figure}
\end{defn}

We note that for $M=\R^{2}$, this spectral sequence is equivalent to the augmented right Mayer--Vietoris spectral sequence of $SC_{n}(R_{n,n})$ as the semisimplicial spaces $Arc_{\bullet}\big(C_{n}(\R^{2})\big)$ and $ MV_{\bullet}\big(SC_{n}(R_{n,n})\big)$ are homotopy equivalent.

\begin{prop}\label{right MV is Arc}
The augmented right Mayer--Vietoris spectral sequence of $SC_{n}(R_{n,n})$ and the augmented arc resolution spectral sequence of $C_{n}(\R^{2})$ are isomorphic.
\end{prop}

\begin{proof}
We define our arc resolution spectral sequence by noting that $C_{n}(\R^{2})$ is homotopy equivalent to $C_{n}(R_{n,n})$ and defining the image of $\gamma$ to be the right side of the rectangle parameterized downwards.
Next, note that $Arc_{p}\big(C_{n}(R_{n,n})\big)$ is equivalent to the subspace in which all of the arcs are trivial, i.e., the points attached to the arcs lie on the right side of $R_{w,h}$.
Given this equivalence, one can check that the homotopy equivalence of \cite[Theorem 18]{plachta2021configuration} induces a homotopy equivalence $Arc_{p}\big(C_{n}(\R^{2})\big)\simeq MV_{p}\big(SC_{n}(R_{n,n})\big)$.
Moreover, this sequence of homotopies commutes with the face maps.
Since the semisimplicial spaces $Arc_{\bullet}\big(C_{n}(\R^{2})\big)$ and $ MV_{\bullet}\big(SC_{n}(R_{n,n})\big)$ are homotopy equivalent, the resulting spectral sequences are isomorphic.
\end{proof}

In addition to the differentials $d_{h}$, which induce the $d^{1}$-differentials of the augmented arc resolution spectral sequence, there are chain differentials $d_{v}:C_{q}\Big(Arc_{p}\big(C_{n}(M)\big)\Big)\to C_{q-1}\Big(Arc_{p}\big(C_{n}(M)\big)\Big)$.
These maps induce the $d^{0}$-differentials of the spectral sequence.
Together, $d_{h}$ and $d_{v}$ allow us to compute the higher differentials of the arc resolution spectral sequence.
For example, if $X\in H_{q}\Big(Arc_{p}\big(C_{n}(M)\big)\Big)\cong E^{1}_{p,q}[M](n)$ is such that that $d^{1}(X)=0$, then we can take a chain representative for $X$, which we also call $X$, and apply $d_{h}$ to it. 
Since $d_{h}(X)$ represents a trivial class in $H_{q}\Big(Arc_{p-1}\big(C_{n}(M)\big)\Big)$, it has a lift $d^{-1}_{v}\big(d_{h}(X)\big)\in C_{q+1}\Big(Arc_{p-1}(M)\big)\Big)$.
Applying $d_{h}$ to $d^{-1}_{v}\big(d_{h}(X)\big)$, we get a cycle in $C_{q+1}\Big(Arc_{p-2}(M)\big)\Big)$, i.e., a class in $H_{q+1}\Big(Arc_{p-2}\big(C_{n}(M)\big)\Big)$.
This is the image of $X\in H_{q}\Big(Arc_{p}\big(C_{n}(M)\big)\Big)$ under the image of the $d^{2}$-differential of the augmented arc resolution spectral sequence; see Figure~\ref{d2Differential}.
We note that this zigzagging process leads to the differentials $d^{r}$, formalizing it as follows.
\begin{figure}[h]
    \centering
    \includegraphics[width=16cm]{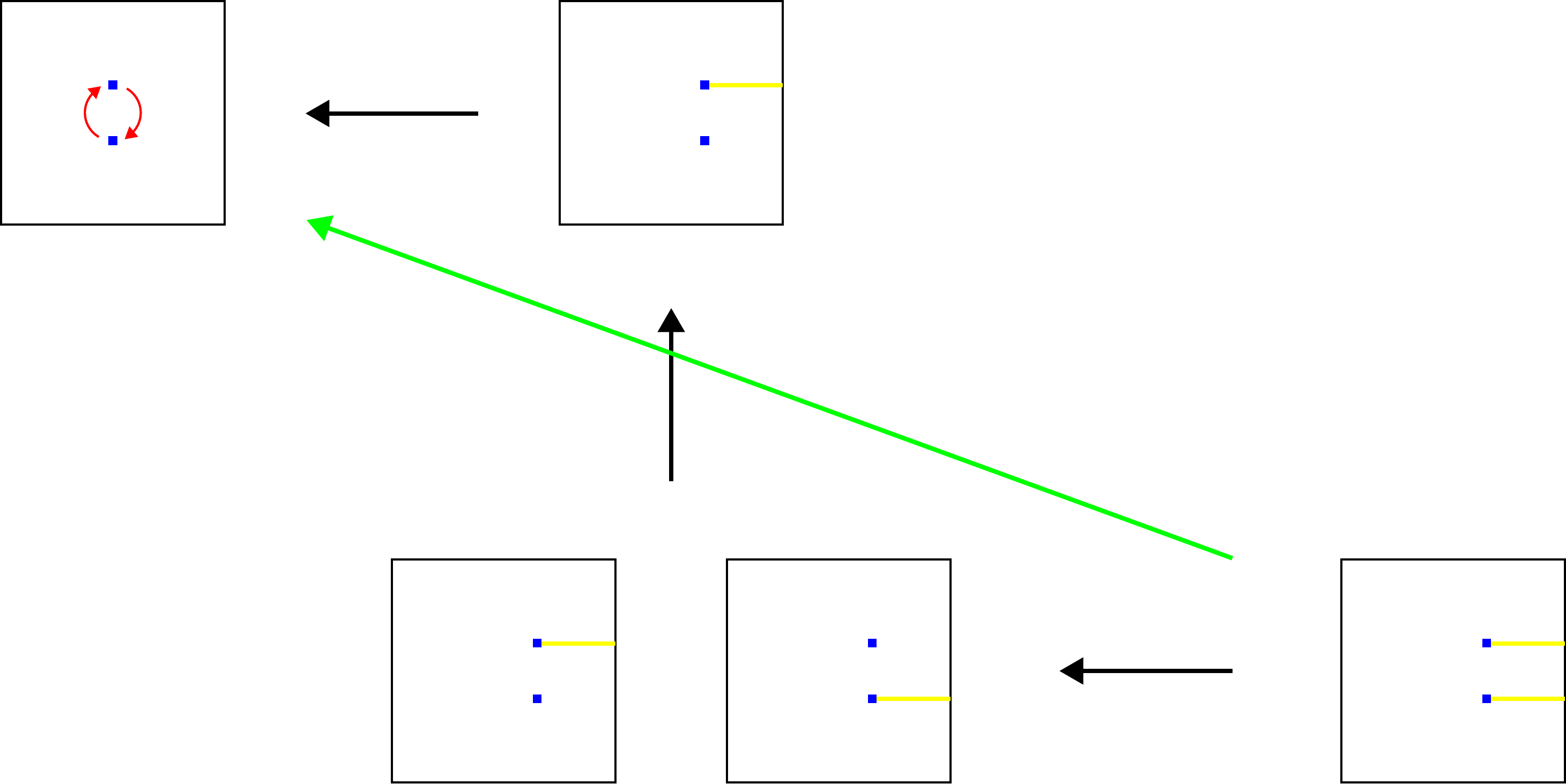}
    \caption{The composition $d_{h}\Big(d^{-1}_{v}\big(d_{h}(A_{2})\big)\Big)$ yielding $d^{2}(A_{2})$.
    The composition of the bottom right horizontal and the vertical arrows yield the zigzag $Z(A_{2})$.}
    \label{d2Differential}
\end{figure}.

\begin{defn}
Given a chain $X\in C_{q}\Big(Arc_{p}\big(C_{n}(M)\big)\Big)$, the \emph{zigzag} of $X$, denoted $Z(X)$, is a chain $d^{-1}_{v}\big(d_{h}(X)\big)\in C_{q+1}\Big(Arc_{p-1}\big(C_{n}(M)\big)\Big)$, whenever such a chain exists; see Figure~\ref{d2Differential}.
We write $Z_{r}$ for the $r$-fold compositions of the zigzag $Z$, and $Z_{0}$ for $d_{h}$.
\end{defn}

\begin{remark}
The zigzag $Z(X)\in C_{q+1}\Big(Arc_{p-1}(M)\big)\Big)$ need not be unique; however, the class represented by the chain $d_{h}\big(Z(X)\big)$ is.
Additionally, we have that $d^{r}=d_{h}\circ Z_{r-1}$.
\end{remark}

Our goal is to compute the differentials of the augmented arc resolution spectral sequence of $\R^{2}$.
With that in mind, we take advantage of the monoidal structure on configuration space.

\begin{defn}
Let $M$ be the interior of a smooth $d$-manifold $\overline{M}$ with an embedding $[0,1]\hookrightarrow \partial\overline{M}$.
Choose an interval in the boundary of $\R^{d-1}\times (-\infty, 0]$, the half-closed disk. Fix an embedding 
\[
\overline{e}:\overline{M}\sqcup \big(\R^{d-1}\times (-\infty, 0]\big)\hookrightarrow \overline{M}
\]
such that
\begin{itemize}
    \item on the interior of the domain, $\overline{e}$ restricts to an embedding of $M\sqcup \R^{d}$ into $M$ such that $\overline{e}|_{M}$ is isotopic to the identity;

    \item the embedding $\overline{e}$ restricts to an embedding of the two boundary intervals of $\overline{M}$ and $\R^{d-1}\times(-\infty, 0]$ into the boundary interval of $\overline{M}$.
\end{itemize}

The embedding $\overline{e}$ induces a map of spaces
\[
e_{p,p'}:\text{Arc}_{p}\big(C_{n}(\R^{d})\big)\times \text{Arc}_{p'}\big(C_{n'}(M)\big)\to \text{Arc}_{p+p'+1}\big(C_{n+n'}(M)\big).
\]
\end{defn}

Specializing to $d=2$ and $M=\R^{2}$, these maps yield a product of arc complexes, and the differentials of the resulting augmented arc resolution spectral sequences satisfy a Liebniz rule as a result.

\begin{prop}\label{Leibniz rule for arc resolution}
The maps $e_{p,p'}$ induce maps
\[
t^{r}:E^{r}_{p,q}(m)\otimes E^{r}_{p',q'}(n)\to E^{r}_{p+p'+1, q+q'}(m+n).
\]
These maps satisfy the following Leibniz rule with respect to the differentials:
if $X\in E^{r}_{p,q}(m)$ and $Y\in E^{r}_{p',q'}(n)$, then
\[
d^{r}\big(t^{r}(X\otimes Y)\big)=t^{r}\big(d^{r}(X)\otimes Y\big)+(-1)^{p+q}t^{r}\big(X\otimes d^{r}(Y)\big).
\]
\end{prop}

\begin{proof}
This follows directly from \cite[Lemma 3.15]{miller2019higher}.
\end{proof}

\begin{defn}
The maps induced by $\overline{e}$ make the chains and the homology of $Arc_{*}\big(C_{\bullet}(\R^{2})\big)$ into a differential graded algebra.
We call this the \emph{arc resolution algebra}, and we write $a$ for the generator of $H_{0}\Big(Arc_{0}\big(C_{1}(\R^{2})\big)\Big)$ and $pt$ for the generator of $H_{0}\Big(Arc_{-1}\big(C_{1}(\R^{2})\big)\Big)=H_{0}\big(C_{1}(\R^{2})\big)$.
Taking the tensor product of $n$ arcs $a$, yields $A_{n}$, the generator of  $H_{0}\Big(Arc_{n-1}\big(C_{n}(\R^{2})\big)\Big)$.

Given $X\in H_{0}\Big(Arc_{-1}\big(C_{n}(\R^{2})\big)\Big)$ and $Y\in H_{0}\Big(Arc_{p}\big(C_{n}(\R^{2})\big)\Big)$, we define the \emph{twist}, $\psi(X,Y)$ to be the classes depicted in Figures~\ref{PureTwists} and \ref{MixedTwists}, and we recursively define the \emph{$l$-pure twists} by 
\[
P_{0}:=pt \text{ and } P_{l}:=\psi(P_{l-1}, P_{l_1});
\]
see Figure~\ref{PureTwists}.
\begin{figure}[H]
    \centering
    \includegraphics[width=8.75cm]{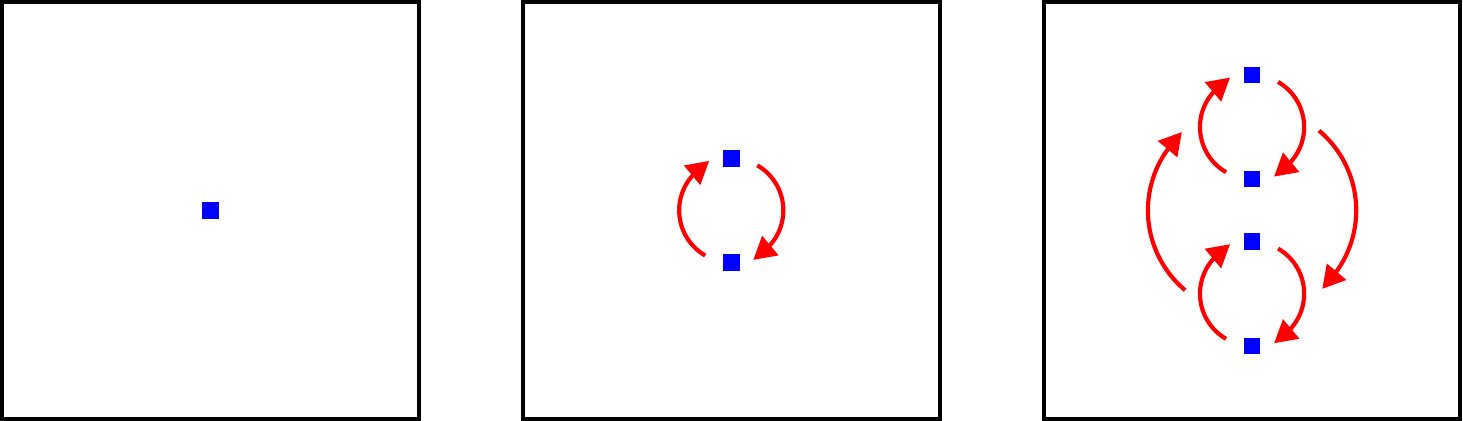}
    \caption{The pure twists $P_{0}$, $P_{1}$, and $P_{2}$.}
    \label{PureTwists}
\end{figure}.
Similarly, we define an \emph{$(l+1)$-mixed twist} by
\[
M_{l+1}:=\psi(P_{l}, A_{2^{l}});
\]
see Figure~\ref{MixedTwists}.
\begin{figure}[H]
    \centering
    \includegraphics[width=5.625cm]{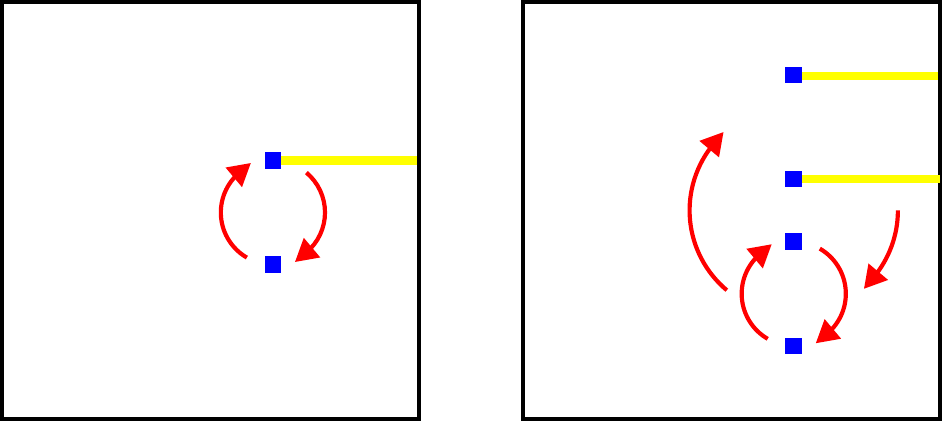}
    \caption{The mixed twists $M_{1}$ and $M_{2}$.}
    \label{MixedTwists}
\end{figure}.
\end{defn}

Returning to the question of determining the images of the $d^{r}$, we note that when working with $\Z/2\Z$-coefficients, twists are often the lifts of classes  under $d^{-1}_{v}$.
Namely, 
\[
d_{v}\big(A\otimes \psi(X,Y)\otimes B\big)=A\otimes(X\otimes Y)\otimes B+A\otimes(Y\otimes X)\otimes B;
\]
see Figure~\ref{d2Differential}.
With this in mind, we note that certain iterated zigzags over $\Z/2\Z$ satisfy a simple Leibniz rule, which will greatly simplify our computations.

\begin{prop}\label{derivation prop}
For every \(s\ge 0\), $Z_{2^{s}}$ is a derivation over $\Z/2\Z$, namely
\[
Z_{2^{s}}(X\otimes Y) = Z_{2^{r}}(X)\otimes Y + X\otimes Z_{2^{s}}(Y),
\]
whenever the relevant zigzags are defined.
\end{prop}

\begin{proof}
The zigzag $Z$ is a derivation on chains, so iterating the Leibniz rule gives
\[
Z_{2^{s}}(X\otimes Y) = \sum_{k=0}^{2^{s}} \binom{2^{s}}{k} Z_{k}(X)\otimes Z_{(2^{s}-k)}(Y).
\]

For $1\le k\le 2^{s}-1$, the coefficient $\binom{2^{s}}{k}$ is even. 
Indeed, in characteristic $2$, we have $(a+b)^{2^{s}}=a^{2^{s}}+b^{2^{s}}$.
Therefore, all intermediate terms vanish over $\Z/2\Z$, leaving
\[
Z_{2^{s}}(X\otimes Y) = Z_{2^{s}}(X)\otimes Y + X\otimes Z_{2^{s}}(Y)
\]
as desired.
\end{proof}

We use this to calculate $Z_{r}(A_{n})$, a fundamental step in our proof of Lemma~\ref{higher differentials}.

\begin{prop}\label{zigzag of An}
If $n=l2^{s+1}$, then 
\[
Z_{2^{s}}(A_{n})=\sum_{k=0}^{l-1} A_{k2^{s+1}} \otimes M_{s+1} \otimes A_{(l-k-1)2^{s+1}}.
\]
\end{prop}

\begin{proof}
We proceed by induction on $s$.
For $s=0$, we have that $n=2l$, and 
\[
Z_{0}(A_{n})=d_{h}(A_{n})=\sum_{j=0}^{2l-1}A_{j}\otimes P_{0}\otimes A_{2l-j-1}.
\]
As we are working with $\Z/2\Z$-coefficients we can pair of the form $A_{2k}\otimes P_{0}\otimes A_{1}\otimes A_{2l-2k-2}$ and  $A_{2k}\otimes A_{1}\otimes P_{0} \otimes A_{2l-2k-2}$.
Such a pair admits a lift $\psi(P_{0}, A_{0})=M_{1}$ under $d^{-1}_{v}$, so
\[
Z_{1}(A_{n})=Z(A_{n})=\sum_{k=0}^{l-1}A_{2k}\otimes M_{1}\otimes A_{2(l-k-1)},
\]
proving the statement for $s=0$.

Now suppose the proposition holds for all $s\le \sigma$.
We wish to show that it holds for $s=\sigma+1$.
Given $n=l2^{\sigma+1}$, we first apply $Z_{2^{\sigma-1}}$.
By the inductive hypothesis,
\[
Z_{2^{\sigma-1}}(A_{n})=\sum^{2l-1}_{j=0}A_{j2^{\sigma}}\otimes M_{\sigma}\otimes A_{(2l-j-1)2^{\sigma}}.
\]
Instead of computing $Z_{2^{\sigma-1}}(A_{j2^{\sigma}}\otimes M_{\sigma}\otimes A_{(2l-j-1)2^{\sigma}})$ to get $Z_{2^{\sigma}}(A_{n})$, which would lead to an ill-defined term, we perform $2^{\sigma-1}-1$ zigzags on $M_{\sigma}$ and compute the final zigzag separately to get $Z_{2^{\sigma}}(A_{n})$.
Successively applying $Z_{2^{\sigma-2}}, Z_{2^{\sigma-3}},\dots, Z_{2^{0}}$ to the internal arc block of $M_{\sigma}$, we see that at each stage the size of this block is reduced by a factor of $2$, and at the end of this process exactly one arc remains; see Figure~\ref{SuccessiveZs}.
\begin{figure}[H]
    \centering
    \includegraphics[width=8.75cm]{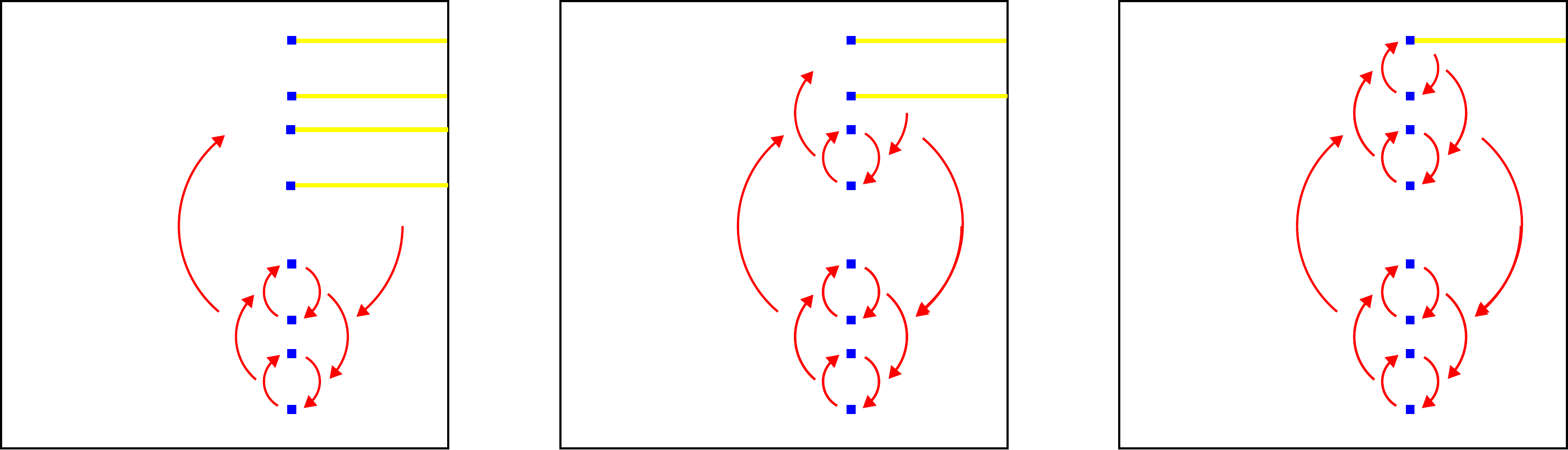}
    \caption{Successive applications of $Z_{2}$ and $Z_{1}$ on $M_{3}$.}
    \label{SuccessiveZs}
\end{figure}.

Expanding these zigzags by Proposition~\ref{derivation prop}, we see that
\begin{multline*}
Z_{2^{\sigma-1}-1}(A_{j2^{\sigma}}\otimes M_{\sigma}\otimes A_{(2l-j-1)2^{\sigma}})\\
=\sum_{a+b+c=2^{\sigma-1}-1}\binom{2^{\sigma-1}-1}{a,b,c}Z_{a}(A_{j2^{\sigma}})\otimes Z_{b}(M_{\sigma})\otimes Z_{c}(A_{(2l-j-1)2^{\sigma}})\\
=\sum_{\substack{a+b+c=2^{\sigma-1}-1\\b\neq 2^{\sigma-1}-1}}\binom{2^{\sigma-1}-1}{a,b,c}Z_{a}(A_{j2^{\sigma}})\otimes Z_{b}(M_{\sigma})\otimes Z_{c}(A_{(2l-j-1)2^{\sigma}})\\
+A_{j2^{\sigma}}\otimes Z_{2^{\sigma-1}-1}(M_{\sigma})\otimes A_{(2l-j-1)2^{\sigma}}.
\end{multline*}

To determine $Z_{1}$ of $A_{j2^{\sigma}}\otimes Z_{2^{\sigma-1}-1}(M_{\sigma})\otimes A_{(2l-j-1)2^{\sigma}}$, we first apply $d_{h}$, and then take a lift along $d_{v}^{-1}$.
\begin{align*}
d_{h}\big(A_{j2^{\sigma}}\otimes Z_{2^{\sigma-1}-1}(M_{\sigma})\otimes A_{(2l-j-1)2^{\sigma}}\big)&=d_{h}(A_{j2^{\sigma}})\otimes Z_{2^{\sigma-1}-1}(M_{\sigma})\otimes A_{(2l-j-1)2^{\sigma}}\\
&+A_{j2^{\sigma}}\otimes d_{h}\big(Z_{2^{\sigma-1}-1}(M_{\sigma})\big)\otimes A_{(2l-j-1)2^{\sigma}}\\
&+A_{j2^{\sigma}}\otimes Z_{2^{\sigma-1}-1}(M_{\sigma})\otimes d_{h}(A_{(2l-j-1)2^{\sigma}}).
\end{align*}
The first and third terms have well-defined lifts, but the middle term does not.
Setting $Q_{j}=A_{j2^{\sigma}}\otimes Z^{\sigma-1}(M_{\sigma})\otimes A_{(2l-j-1)2^{\sigma}}$, we pair off $Q_{2k}$ and $Q_{2k+1}$, and return to them later.

Computing $Z_{1}$ of the terms in the sum, we see that by the trinomial theorem and the fact that we are working with $\Z/2\Z$-coefficients, this sum reduces to 
\[
Z_{2^{\sigma-1}}(A_{j2^{\sigma}})\otimes M_{\sigma}\otimes A_{(2l-j-1)2^{\sigma}}+A_{j2^{\sigma}}\otimes M_{\sigma}\otimes Z_{2^{\sigma-1}}(A_{(2l-j-1)2^{\sigma}}).
\]
Taking the sum of these terms as $j$ ranges from $0$ to $2l-1$, we get $0$ by the inductive hypothesis as we can pair off terms.
It follows that
\[
Z_{2^{\sigma}}(A_{n})=\sum^{l-1}_{k=0}d^{-1}_{h}(Q_{2k}+Q_{2k+1}).
\]
Each pair $Q_{2k}+Q_{2k+1}$ lifts to the pure twist $P_{\sigma}$ together with the adjacent block of $2^{\sigma}$ arcs, i.e., the $(\sigma+1)$-mixed twist $M_{\sigma+1}$, sandwiched between $A_{k2^{\sigma+1}}$ and $A_{(l-k-1)2^{\sigma+1}}$, yielding
\[
Z_{2^{\sigma}}(A_{n})=\sum^{l-1}_{k=0}A_{k2^{\sigma+1}}\otimes M_{\sigma+1}\otimes A_{(l-k-1)2^{\sigma+1}},
\]
as desired.
\end{proof}

This allows us to prove Lemma~\ref{higher differentials}.

\begin{proof}[Proof of Lemma~\ref{higher differentials}]

We proceed by induction on $s$.

If $s=0$, then $n=m$, and the chain $A_{n}$ consists of an odd number of arcs.
Applying the horizontal differential to $A$ yields $n$ homologous classes in $H_{0}\Big(Arc_{n-2}\big(C_{n}(\R^{2})\big);\Z/2\Z\Big)$, each of which corresponds to a single unattached point and $n-1=m-1$ arcs, i.e., $\left[P_{0}\otimes A_{m-1}\right]$.
Since we are working with $\Z/2\Z$-coefficients, and $n=m$ is odd, we see that the image of this class under the $d^{1}$-differential is $\left[P_{0}\otimes A_{m-1}\right]$, as desired.

Next, suppose that the lemma holds for all $n=m2^{s}$, where $s\le \sigma$, for some non-negative integer $\sigma$.
We will we show that it holds for all $n$ of the form $m2^{\sigma+1}$.
Consider $A_{n}$, where $n=m2^{\sigma+1}$.
We have that
\[
d^{2^{\sigma+1}}(A_{n})=d_{h}\big(Z_{2^{\sigma+1}-1}(A_{n})\big)=d_{h}\Big(Z_{2^{\sigma}-1}\big(Z_{2^{\sigma}}(A_{n})\big)\Big)
\]
Since $2^{\sigma+1}$ divides the number of arcs, Proposition~\ref{zigzag of An} yields
\[
Z_{2^{\sigma}}(A_{n})=\sum_{j=0}^{m-1}A_{j2^{\sigma+1}}\otimes M_{\sigma+1}\otimes A_{(m-j-1)2^{\sigma+1}}.
\]
Applying $Z_{2^{\sigma}-1}$ to the summands, Proposition~\ref{derivation prop} yields  
\begin{multline*}
Z_{2^{\sigma}-1}(A_{j2^{\sigma+1}}\otimes M_{\sigma+1}\otimes A_{(m-j-1)2^{\sigma+1}})\\
=\sum_{a+b+c=2^{\sigma}-1}\binom{2^{\sigma}-1}{a,b,c}Z_{a}(A_{j2^{\sigma+1}})\otimes Z_{b}(M_{\sigma+1})\otimes Z_{c}(A_{(m-j-1)2^{\sigma+1}}),
\end{multline*}
so
\[
Z_{2^{\sigma+1}-1}(A_{n})=\sum_{j=0}^{m-1}\sum_{a+b+c=2^{\sigma}-1}\binom{2^{\sigma}-1}{a,b,c}Z_{a}(A_{j2^{\sigma+1}})\otimes Z_{b}(M_{\sigma+1})\otimes Z_{c}(A_{(m-j-1)2^{\sigma+1}}).
\]

Applying $d_{h}$ to this sum, we see that since $a$ and $c$ are at most $2^{\sigma}-1$, and $j2^{\sigma+1}$ and $(m-j-1)2^{\sigma+1}$ are divisible by $2^{\sigma+1}$, the inductive hypothesis implies  implies that $d_{h}\big(Z_{a}(A_{j2^{\sigma+1}})\big)$ and $d_{h}\big(Z_{c}(A_{(m-j-1)2^{\sigma+1}})\big)$ are $0$.
Therefore,
\[
d^{2^{\sigma+1}}(A_{n})=\sum_{j=0}^{m-1}\sum_{a+b+c=2^{\sigma}-1}\binom{2^{\sigma}-1}{a,b,c}Z_{a}\left(A_{j2^{\sigma+1}}\right)\otimes d_{h}\big(Z_{b}(M_{\sigma+1})\big)\otimes Z_{c}\left(A_{(m-j-1)2^{\sigma+1}}\right).
\]

Next, noting that $M_{\sigma+1}=\psi\big(P_{\sigma}, A_{2^{\sigma}}\big)$, we see that if $b<2^{\sigma}-1$, then by the inductive hypothesis, $A_{2^{\sigma}}$ supports another zigzag, so $d_{h}\big(Z_{b}(M_{\sigma+1})\big)$ is trivial in homology.
Therefore,
\[
d^{2^{\sigma+1}}(A_{n})=\sum_{j=0}^{m-1}A_{j2^{\sigma+1}}\otimes d_{h}\big(Z_{2^{\sigma}-1}(M_{\sigma+1})\big)\otimes A_{(m-j-1)2^{\sigma+1}}.
\]
Since $Z_{2^{\sigma}-1}=Z_{1}\circ\cdots\circ Z_{2^{\sigma-1}}$; see Figure~\ref{SuccessiveZs}, we see that there is exactly one arc in $\big(Z_{2^{\sigma}-1}(M_{\sigma+1})\big)$.
Apply $d_{h}$, we see that we obtain a $(\sigma+1)$-pure twist $P_{\sigma+1}$.
It follows that
\[
d^{2^{\sigma+1}}(A_{n})=\sum_{j=0}^{m-1}\left[A_{j2^{\sigma+1}}\otimes P_{\sigma+1}\otimes A_{(m-j-1)2^{\sigma+1}}\right]=\left[A_{j2^{\sigma+1}}\otimes P_{\sigma+1}\otimes A_{(m-j-1)2^{\sigma+1}}\right].
\]
Here, the latter equality comes from the observations that the $m$ summands represent the same class, $m$ is odd, and we are working with $\Z/2\Z$ coefficients.

Lastly, we note that since $A_{n}$ generates $H_{0}\Big(Arc_{n-1}\big(C_{n}\R^{2})\big);\Z/2\Z\Big)\cong \Z/2\Z$, it follows that $d^{q}(A_{n})=0$ for $q<2^{\sigma+1}$, completing the inductive step and the proof.
\end{proof}

Having computed the differentials of the augmented arc resolution spectral sequence of $C_{n}(\R^{2})$, hence the differentials of the augmented right Mayer--Vietoris spectral sequence of $SC_{n}(R_{n,n})$, we provide several computer aided computations of the $\Z/2\Z$ and rational Betti numbers of $SC_{n}(R_{w,h})$.

\subsection{Computer calculations}\label{computer calculations subsection}

This subsection contains the $\Z/2\Z$-Betti numbers of all non-empty $SC_{n}(R_{w,h})$ for $3\le n\le 7$ and $w,h\le n$, in addition to the rational Betti numbers for $3\le n\le 5$ and $w,h\le n$.
We obtained these Betti numbers by modifying the Python code used by Alpert, Bauer, Kahle, MacPherson, and Spendlove to calculate the Betti numbers of $SF_{n}(R_{w,h})$ \cite[Table 1]{alpert2023homology}.
As we work with unordered configuration spaces, the resulting cellular complex is smaller than theirs, so our code\footnote{The code used to compute the Betti numbers can be found at \url{https://github.com/jz1010/path-complex-labeled}. We implemented this code on the University of Chicago's computing cluster.} was able to compute $\Z/2\Z$-Betti numbers for one more square than in the ordered case.

\begin{longtable}{|c|c|c!{\vrule width 2pt}c|c|c|c|c|c|c|}
\caption{The $\Z/2\Z$-Betti numbers of $SC_{n}(R_{w,h})$ for $3\le n\le 7$ and $w,h\le n$.
Entries are in boldface if they do not agree with the corresponding $\Z/2\Z$-Betti number of $C_{n}(\R^{2})$.}
\label{table z2z}\\
\hline
$n$ & $h$ & $w$ & $\beta_0$ & $\beta_1$ & $\beta_2$ & $\beta_3$ & $\beta_4$ & $\beta_5$ & $\beta_6$ \\
\hline
\endfirsthead
\multicolumn{10}{l}{\small\itshape continued from previous page}\\
\hline
$n$ & $h$ & $w$ & $\beta_0$ & $\beta_1$ & $\beta_2$ & $\beta_3$ & $\beta_4$ & $\beta_5$ & $\beta_6$ \\
\hline
\endhead
\hline
\multicolumn{10}{r}{\small\itshape continued on next page}\\
\endfoot
\hline
\endlastfoot

3 & 2 & 2 & 1 & 1 & 0 & 0 & 0 & 0 & 0 \\
3 & 2 & 3 & 1 & \textbf{2} & 0 & 0 & 0 & 0 & 0 \\
3 & 3 & 3 & 1 & 1 & 0 & 0 & 0 & 0 & 0 \\
\hline
4 & 2 & 2 & 1 & \textbf{0} & \textbf{0} & \textbf{0} & 0 & 0 & 0 \\
4 & 2 & 3 & 1 & \textbf{3} & \textbf{0} & \textbf{0} & 0 & 0 & 0 \\
4 & 2 & 4 & 1 & \textbf{3} & 1 & \textbf{0} & 0 & 0 & 0 \\
4 & 3 & 3 & 1 & 1 & \textbf{0} & \textbf{0} & 0 & 0 & 0 \\
4 & 3 & 4 & 1 & 1 & 1 & \textbf{0} & 0 & 0 & 0 \\
4 & 4 & 4 & 1 & 1 & 1 & 1 & 0 & 0 & 0 \\
\hline
5 & 2 & 3 & 1 & \textbf{2} & \textbf{0} & \textbf{0} & 0 & 0 & 0 \\
5 & 2 & 4 & 1 & \textbf{4} & \textbf{2} & \textbf{0} & 0 & 0 & 0 \\
5 & 2 & 5 & 1 & \textbf{4} & \textbf{3} & \textbf{0} & 0 & 0 & 0 \\
5 & 3 & 3 & 1 & 1 & \textbf{0} & \textbf{0} & 0 & 0 & 0 \\
5 & 3 & 4 & 1 & 1 & \textbf{2} & \textbf{0} & 0 & 0 & 0 \\
5 & 3 & 5 & 1 & 1 & 1 & \textbf{0} & 0 & 0 & 0 \\
5 & 4 & 4 & 1 & 1 & 1 & 1 & 0 & 0 & 0 \\
5 & 4 & 5 & 1 & 1 & 1 & \textbf{2} & 0 & 0 & 0 \\
5 & 5 & 5 & 1 & 1 & 1 & 1 & 0 & 0 & 0 \\
\hline
6 & 2 & 3 & 1 & \textbf{0} & \textbf{0} & \textbf{0} & \textbf{0} & 0 & 0 \\
6 & 2 & 4 & 1 & \textbf{5} & 1 & \textbf{0} & \textbf{0} & 0 & 0 \\
6 & 2 & 5 & 1 & \textbf{5} & \textbf{6} & \textbf{0} & \textbf{0} & 0 & 0 \\
6 & 2 & 6 & 1 & \textbf{5} & \textbf{6} & \textbf{1} & \textbf{0} & 0 & 0 \\
6 & 3 & 3 & 1 & 1 & \textbf{0} & \textbf{0} & \textbf{0} & 0 & 0 \\
6 & 3 & 4 & 1 & 1 & \textbf{3} & \textbf{0} & \textbf{0} & 0 & 0 \\
6 & 3 & 5 & 1 & 1 & 1 & 2 & \textbf{0} & 0 & 0 \\
6 & 3 & 6 & 1 & 1 & 1 & \textbf{1} & \textbf{0} & 0 & 0 \\
6 & 4 & 4 & 1 & 1 & \textbf{2} & \textbf{3} & 1 & 0 & 0 \\
6 & 4 & 5 & 1 & 1 & 1 & \textbf{4} & 1 & 0 & 0 \\
6 & 4 & 6 & 1 & 1 & 1 & \textbf{4} & \textbf{2} & 0 & 0 \\
6 & 5 & 5 & 1 & 1 & 1 & 2 & 1 & 0 & 0 \\
6 & 5 & 6 & 1 & 1 & 1 & 2 & \textbf{2} & 0 & 0 \\
6 & 6 & 6 & 1 & 1 & 1 & 2 & 1 & 0 & 0 \\
\hline
7 & 2 & 4 & 1 & \textbf{3} & \textbf{0} & \textbf{0} & \textbf{0} & 0 & 0 \\
7 & 2 & 5 & 1 & \textbf{6} & \textbf{7} & \textbf{0} & \textbf{0} & 0 & 0 \\
7 & 2 & 6 & 1 & \textbf{6} & \textbf{10} & \textbf{3} & \textbf{0} & 0 & 0 \\
7 & 2 & 7 & 1 & \textbf{6} & \textbf{10} & \textbf{4} & \textbf{0} & 0 & 0 \\
7 & 3 & 3 & 1 & \textbf{5} & \textbf{0} & \textbf{0} & \textbf{0} & 0 & 0 \\
7 & 3 & 4 & 1 & 1 & \textbf{4} & \textbf{0} & \textbf{0} & 0 & 0 \\
7 & 3 & 5 & 1 & 1 & 1 & \textbf{6} & \textbf{0} & 0 & 0 \\
7 & 3 & 6 & 1 & 1 & 1 & 2 & \textbf{0} & 0 & 0 \\
7 & 3 & 7 & 1 & 1 & 1 & \textbf{1} & \textbf{0} & 0 & 0 \\
7 & 4 & 4 & 1 & 1 & \textbf{2} & \textbf{3} & 1 & 0 & 0 \\
7 & 4 & 5 & 1 & 1 & 1 & \textbf{5} & \textbf{3} & 0 & 0 \\
7 & 4 & 6 & 1 & 1 & 1 & \textbf{5} & \textbf{5} & 0 & 0 \\
7 & 4 & 7 & 1 & 1 & 1 & \textbf{5} & \textbf{4} & 0 & 0 \\
7 & 5 & 5 & 1 & 1 & 1 & 2 & 1 & 0 & 0 \\
7 & 5 & 6 & 1 & 1 & 1 & 2 & \textbf{3} & 0 & 0 \\
7 & 5 & 7 & 1 & 1 & 1 & 2 & \textbf{2} & 0 & 0 \\
7 & 6 & 6 & 1 & 1 & 1 & 2 & \textbf{2} & \textbf{1} & 0 \\
7 & 6 & 7 & 1 & 1 & 1 & 2 & 1 & \textbf{1} & 0 \\
7 & 7 & 7 & 1 & 1 & 1 & 2 & 1 & 0 & 0 \\
 
\end{longtable}

\begin{longtable}{|c|c|c!{\vrule width 2pt}c|c|c|c|c|}
\caption{The rational Betti numbers of $SC_{n}(R_{w,h})$ for $3\le n\le 5$ and $w,h\le n$.
Entries are in boldface if they do not agree with the corresponding Betti number of $C_{n}(\R^{2})$.}
\label{table z}\\
\hline
$n$ & $h$ & $w$ & $\beta_0$ & $\beta_1$ & $\beta_2$ & $\beta_3$ & $\beta_4$   \\
\hline
\endfirsthead
\endfoot
\hline
\endlastfoot

3 & 2 & 2 & 1 & 1 & 0 & 0 & 0  \\
3 & 2 & 3 & 1 & \textbf{2} & 0 & 0 & 0  \\
3 & 3 & 3 & 1 & 1 & 0 & 0 & 0  \\
\hline
4 & 2 & 2 & 1 & \textbf{0} & 0 & 0 & 0  \\
4 & 2 & 3 & 1 & \textbf{3} & 0 & 0 & 0 \\
4 & 2 & 4 & 1 & \textbf{3} & \textbf{1} & 0 & 0  \\
4 & 3 & 3 & 1 & 1 & 0 & 0 & 0  \\
4 & 3 & 4 & 1 & 1 & \textbf{1} & 0 & 0  \\
4 & 4 & 4 & 1 & 1 & 0 & 0 & 0  \\
\hline
5 & 2 & 3 & 1 & \textbf{2} & 0 & 0 & 0 \\
5 & 2 & 4 & 1 & \textbf{4} & \textbf{2} & 0 & 0 \\
5 & 2 & 5 & 1 & \textbf{4} & \textbf{3} & 0 & 0  \\
5 & 3 & 3 & 1 & 1 & 0 & 0 & 0  \\
5 & 3 & 4 & 1 & 1 & \textbf{2} & 0 & 0  \\
5 & 3 & 5 & 1 & 1 & \textbf{1} & 0 & 0  \\
5 & 4 & 4 & 1 & 1 & \textbf{1} & \textbf{1} & 0  \\
5 & 4 & 5 & 1 & 1 & 0 & \textbf{1} & 0  \\
5 & 5 & 5 & 1 & 1 & 0 & 0 & 0  \\
\end{longtable}

\subsection{Conjectures and future directions}\label{Conjectures and future directions subsection}

Informed by the calculations of the previous two subsections, we make several conjectures about the higher homology groups of $SC_{n}(R_{w,h})$ and the relation between the square-section braid group $SBr_{n}(R_{w,h})$ and the braid group $Br_{n}$.
We also pose several questions about other disk configuration spaces. 

This paper was motivated in large part by Question~\ref{how much space quest}, i.e., how much space do you need to physically realize the braid group?
Computations of Alvarado-Gardu\~{n}o, Gonz\'{a}lez, and Kahle show that in the $h=2$ setting you need at least $5$ unoccupied squares to realize the Artin relations \cite[Theorem 1.5 and Remark 1.6]{MattOmarJesus}.
This might be an artifact of the height restriction as one should only need $4$ unoccupied squares to change the order in which two twists are realized.
That said, realizing the braid relation might take more space.
Since $SC_{3}(R_{3,3})\simeq C_{3}(\R^{2})$, at most $6$ unoccupied squares are needed.
Given Lemma~\ref{generated by orbiting pairs}, which tells us that nothing other than orbiting pair square classes, i.e., the square analogs of half-twists, are needed to generate $H_{1}\big(SC_{n}(R_{w,h})\big)$ if $w,h\ge 3$ and $n\le wh-3$, we view it as unlikely that $SBr_{n}(R_{w,h})$ is generated by anything other than orbiting pairs.
As such, we conjecture

\begin{conj}
If $w,h\ge 3$, then $SBr_{n}(R_{w,h})\cong Br_{n}$ if and only if $n\le wh-6$.
\end{conj}

Despite our belief that $SBr_{wh-5}(R_{w,h})\not\cong Br_{wh-5}$, we have reason to believe that the second homology groups of $SC_{wh-5}(R_{w,h})$ and $C_{wh-5}(\R^{2})$ agree.
Recall that $H_{2}\big(C_{n}(\R^{2})\big)\cong \Z/2\Z$ for $n\ge 4$.
Moreover, this group is generated by the product of two sets of orbiting pairs of points. 
Proposition~\ref{orbiting square class is homologous to orbiting t class} gives a way of moving an orbiting pair square class about $R_{w,h}$ granted there is $1$ unoccupied square.
Given that an orbiting pair square class requires $2$ unoccupied squares, this would suggest that if there are $5$ unoccupied squares in $R_{w,h}$, one could interchange the locations of the two orbiting pair square classes if $w$ and $h$ are at least $5$ while moving the singleton squares out of the way.
With this in mind, we conjecture

\begin{conj}
If $w,h\ge 5$, then $H_{2}\big(SC_{n}(R_{w,h})\big)\cong H_{2}\big(C_{n}(\R^{2})\big)$ if and only if $n\le wh-5$.
\end{conj}

In the $h=2$ setting, Alvarado-Gardu\~{n}o, Gonz\'{a}lez, and Kahle proved that $SBr_{n}(R_{w,2})\cong SBr_{n}(R_{n,2})$ if $n\le 2w-5$ \cite[Theorem 1.5]{MattOmarJesus}, leaving open the question of whether $SC_{n}(R_{w,2})$ is aspherical in these cases \cite[Remark 1.1]{MattOmarJesus}.
While we do not answer this question, we believe that the techniques of Section~\ref{h=2 section} can be extended to compute the higher homologies of $SC_{n}(R_{w,2})$, and that these computations will show that $H_{k}\big(SC_{n}(R_{w,2})\big)$ is generated by $k$-fold products of orbiting pair and triple square classes.
We believe that the counts of these classes will differ in $H_{k}\big(SC_{n}(R_{w,2})\big)$ and $H_{k}\big(SC_{n}(R_{n,2})\big)$ for large enough $k$.
As $SC_{n}(R_{n,2})$ is aspherical, this would prove that if $2<w<n\le2w-5$, then $SC_{n}(R_{w,2})$ is not aspherical.
Additionally, our computations of $\Z/2\Z$-Betti numbers of $SC_{6}(R_{5,2})$ in Table~\ref{table z2z} disagree with the Betti numbers of $SBr_{6}(R_{5,2})\cong SBr_{6}(R_{6,2})$---Alvarado-Gardu\~{n}o, Gonz\'{a}lez, and Kahle claim these groups are isomorphic \cite[Remark 1.6]{MattOmarJesus}---proving that this space is not aspherical.
Extrapolating, we believe the following:

\begin{conj}\label{n le 2w-5 spherical conjecture}
If $2<w<n\le2w-4$, then $SC_{n}(R_{w,2})$ is not a $K\big(SBr_{n}(R_{w,2}),1\big)$.
\end{conj}

The $n=2w-2$ and $n=2w-1$ cases are handled by
\cite[Theorem 2.2]{alvarado2025square} and \cite[Theorem 1.2]{MattOmarJesus}, respectively.
These papers show that $SBr_{2w-2}(R_{w,2})$ and $SBr_{2w-1}(R_{w,2})$ are aspherical, so if Conjecture~\ref{n le 2w-5 spherical conjecture} holds, the only remaining case would be $n=2w-3$.
Computations by Claude Opus 5 suggest that $\Z/2\Z$-Betti numbers of $SBr_{9}(R_{6,2})$ agree with those of $SC_{9}(R_{6,2})$.
One can see a similar agreement of the $\Z/2\Z$-Betti numbers of $SBr_{5}(R_{4,2})$ and $SC_{5}(R_{4,2})$ as well as the $\Z/2\Z$-Betti numbers of $SBr_{7}(R_{5,2})$ and $SC_{7}(R_{5,2})$ in our Table~\ref{table z2z} and those arising from the presentations of these groups in \cite[Remark 1.6]{MattOmarJesus}.
Given this, we conjecture:

\begin{conj}
$SC_{2w-3}(R_{w,2})$ is a $K\big(SBr_{2w-3}(R_{w,2}),1\big)$.
\end{conj}

We approached Questions~\ref{how much space quest} and \ref{topology of thick braids quest} by considering square-section braids, but one could also consider configurations of round disks, either in a rectangle or a larger disk.
Of these disk configuration spaces, which have received a great deal of interest in recent years, e.g., \cite{alpert2021configuration, alpert2021configuration1, wawrykow2024topological}, the most pertinent to this paper are $CF_{n}(r)$ and $CC_{n}(r)$, the \emph{ordered} and \emph{unordered configuration spaces of $n$ open unit-diameter disks in the circle of diameter $r$}.
To the best of our knowledge, these spaces have only been studied by Alpert \cite{alpert2017restricting} and Ramsey \cite{ramsey2025thick}.
Ramsey was able to determine the fundamental groups of these spaces when there are $4$ or fewer disks \cite[Theorem C]{ramsey2025thick}, in addition to proving that these spaces are not aspherical for $n\ge 4$ and $r<n$, but much remains to be discovered.
With this in mind we ask

\begin{quest}
When is $\pi_{1}\big(CC_{n}(r)\big)\cong Br_{n}$?
\end{quest}

\begin{quest}
When is $H_{k}\big(CC_{n}(r)\big)\cong H_{k}\big(C_{n}(\R^{2})\big)$?
\end{quest}
 
\bibliographystyle{amsalpha}
\bibliography{StabilityinSpaceDirection}










\end{document}